\documentclass[11pt,a4paper]{article}
\usepackage{amsmath, amsfonts, amssymb,amsthm}
\usepackage{a4wide}
\usepackage{parskip}
\usepackage{enumitem}
\usepackage{xcolor}
\usepackage{mathrsfs}
\usepackage{cite}
\usepackage{pstricks}
\usepackage{hyperref}
\usepackage{latexsym}
\usepackage{cite}
\usepackage{comment}
\usepackage{graphicx}
\usepackage{subcaption}
\usepackage{float}
\def\QED{\hfill {$\square$}\goodbreak \medskip}
\usepackage[left=1in, right=1in,top=1in,bottom=1in]{geometry}
\allowdisplaybreaks

\newcommand{\be} {\begin{equation}}
	\newcommand{\ee} {\end{equation}}
\newcommand{\bea} {\begin{eqnarray}}
	\newcommand{\eea} {\end{eqnarray}}
\newcommand{\Bea} {\begin{eqnarray*}}
	\newcommand{\Eea} {\end{eqnarray*}}
\newcommand{\noi} {\noindent}

\newcommand{\de} {\delta}

\newcommand{\La} {\Lambda}

\newcommand{\e} {\mathcal{E}_{\beta}}
\newcommand{\f}{\Phi^u_{\beta}}
\newcommand{\sr}{\mathcal{S}(a,b)}
\newcommand{\ue}{u_{\varepsilon}}
\newcommand{\ve}{v_{\varepsilon}}
\newcommand{\fo}{\Phi_{0}^{v_{\varepsilon}}}
\newcommand{\fom}{\Phi_{\beta}^{v_{\varepsilon}}}
\newcommand{\SN}{|\mathbb S^{N-1}|}
\def\dx{\,{\rm d}x}
\def\dy{\,{\rm d}y}

\def\R{{\mathbb R}}

\def\C{{\mathcal C}}

\def\R{{\mathbb R}}
\def\S{S_{\rho}}
\def\t{{2^{\sharp}}}
\def\q{{L^q_r(\R^N;|x|^{-qb})}}
\def\cri{{L^{2^{\sharp}}_r(\R^N;|x|^{-2^\sharp b})}}
\def\spa{{\mathcal{X}}}
\def\n{{\| |x|^{-a}\na u}}

\def\M{{\mathcal{M}_{\rho,\beta}}}
\def\C{{C^q_{a,b,q}}}
\def\noi{{\noindent}}

\newcommand{\na} {\nabla}

\newcommand{\var} {\varepsilon}

\hypersetup{
	colorlinks=true,
	linkcolor=black,
	filecolor=black,
	urlcolor=black,
	citecolor=green,
}
\numberwithin{equation}{section}
\usepackage[normalem]{ulem}
\newtheorem{definition}{Definition}[section]
\newtheorem{theorem}{Theorem}[section]
\newtheorem{rem}{Remark}[section]
\theoremstyle{plain}
\newtheorem{lemma}{Lemma}[section]
\newtheorem{prop}{Proposition}[section]

\numberwithin{equation}{section}

\newcommand{\nr}[1]{\|#1\|_{\gamma_a}}
\def\ds{\,{\rm d}s}
\begin{document}
	\setlength{\abovedisplayskip}{3pt}
	\setlength{\belowdisplayskip}{3pt}
	\date{}
	\title{Normalized solutions for a critical Caffarelli–Kohn–Nirenberg equation: the effect of the weights on the coupling range}
     \author{ {\bf Asmita Rai$\,$\footnote{e-mail: {\tt asmita.rai65@gmail.com}}, Divya Goel$\,$\footnote{e-mail: {\tt divya.mat@iitbhu.ac.in}}} \\ Department of Mathematical Sciences, Indian Institute of Technology (BHU),\\ Varanasi 221005, India.}
	\maketitle
	\begin{abstract}
		We study normalized solutions of
		\[
		\begin{array}{cc}
			-\operatorname{div}(|x|^{-2a}\nabla u)=\lambda \frac{u}{|x|^{2a}}+\beta\frac{|u|^{q-2}u}{|x|^{bq}}
			+\frac{|u|^{2^{\sharp}-2}u}{|x|^{b{2^{\sharp}}}}\quad \text{in}~\mathbb{R}^N,\\[2mm]
			\displaystyle\int_{\R^N}\frac{|u|^2}{|x|^{2a}}dx=\rho^2,
		\end{array}
		\]
		where $0<a<\frac{N-2}{2}$, $a<b<a+1$, $2<q<2^\sharp$, and $2^\sharp$ is the critical Caffarelli--Kohn--Nirenberg exponent. Let $q_c$ denote the mass-critical exponent. For $2<q<q_c$ and small $\beta>0$, we obtain two positive radial solutions: a local minimizer with negative energy and a solution of mountain-pass type. For $q_c\le q<2^\sharp$, there is a number $\beta_0\ge0$, defined through the energy level below which compactness holds, such that a positive radial ground state exists for every $\beta>\beta_0$. In the mass-critical case $q=q_c$, this holds for $\beta_0<\beta<\beta^*$, where $\beta^*$ is explicit, and for $\beta\ge\beta^*$ the ground state level is zero and is not attained. Our main result is a necessary and sufficient condition on $N$, $a$, $b$, and $q$ for $\beta_0>0$. The condition holds only when $a$ is close to $\frac{N-2}{2}$, and it fails in the formal unweighted limit $a=b=0$. When $\beta_0>0$, the ground state level is not attained for $0<\beta<\beta_0$, and every normalized solution has energy above the compactness level. When $\beta_0=0$, we describe the behavior of ground states as $\beta\to0$. The proofs combine the Pohozaev constraint, sharp estimates for the extremals of the Caffarelli--Kohn--Nirenberg inequality, and a blow-up analysis in Emden--Fowler variables.

		\medskip
		\noindent\textbf{Keywords:} Normalized solutions; Caffarelli--Kohn--Nirenberg inequality; Hardy--Sobolev critical exponent; Inverse-square potential; Combined nonlinearities; Emden--Fowler transformation.
		
		\noindent\textbf{2020 Mathematics Subject Classification:} 35J60, 35J20, 35B33, 35J75, 35Q55, 35B06, 35B40.
	\end{abstract}

	\tableofcontents
	\section{Introduction}
	
	In this paper, we study normalized solutions of the problem
	\begin{equation}
		\begin{cases}
			-\operatorname{div}(|x|^{-2a}\na u)=\lambda|x|^{-2a}u+\beta|x|^{-bq}|u|^{q-2}u+|x|^{-b\t}|u|^{\t-2}u\quad \text{in}~\R^N,\\
			\displaystyle\int_{\R^N}|x|^{-2a}u^2\,dx=\rho^2,
		\end{cases}
		\tag{\(\mathcal{P}^{a}_{\beta}\)} 
		\label{eq:P}
	\end{equation}
	where $N\ge3$, $0<a<\frac{N-2}{2}$, $a<b<a+1$, $d:=1+a-b$, $\t:=\frac{2N}{N-2d}$, $2<q<\t$ and $\beta,\rho>0$. The number $\lambda\in\R$ is also unknown. The operator $-\operatorname{div}(|x|^{-2a}\na u)$ and the weights originate in the Caffarelli--Kohn--Nirenberg inequalities \cite{caffarelli1984first}, recalled in Theorem~\ref{ckn} below, and $\t$ is their critical exponent.
	
	To see where the problem comes from, note that the substitution $v=|x|^{-a}u$ (Remark~\ref{Rhardy}) transforms \eqref{eq:P} into
	\begin{align*}
		-\Delta v-\frac{\mu}{|x|^2}v=\lambda v+\beta\frac{|v|^{q-2}v}{|x|^{(b-a)q}}+\frac{|v|^{2^*(\varsigma)-2}v}{|x|^{\varsigma}}\quad\text{in }\R^N,\qquad\int_{\R^N}v^2\dx=\rho^2,
	\end{align*}
	where $\mu:=a(N-2-a)\in\big(0,\frac{(N-2)^2}{4}\big)$ and $\varsigma:=(b-a)\t\in(0,2)$, so that $\t=2^*(\varsigma):=\frac{2(N-\varsigma)}{N-2}$ is the Hardy--Sobolev critical exponent. This is the equation satisfied by the standing waves $\psi(t,x)=e^{-i\lambda t}v(x)$ of the Schr\"odinger equation with an inverse-square potential
	\begin{align*}
		i\partial_t\psi+\Delta\psi+\frac{\mu}{|x|^2}\psi+\beta\frac{|\psi|^{q-2}\psi}{|x|^{(b-a)q}}+\frac{|\psi|^{2^*(\varsigma)-2}\psi}{|x|^{\varsigma}}=0.
	\end{align*}
	Since solutions of this evolution equation formally conserve the mass $\int_{\R^N}|\psi|^2\dx$, 
	it is natural to prescribe the mass (we do not study the Cauchy problem for this equation, and the conservation of mass is used only as a formal motivation). In the variable $u$ it becomes $\int_{\R^N}|x|^{-2a}u^2\dx$, the quantity prescribed above. For $a=b=0$, the problem reduces to the nonlinear Schr\"odinger equation with combined power nonlinearities, whose normalized solutions were studied in \cite{soave2020normalized1,soave2020normalized}.
	
	The equation in \eqref{eq:P} can be studied in two ways. One may fix $\lambda$ and look for solutions without a constraint. For problems of this type with Hardy or Hardy--Sobolev weights, we refer to \cite{assunccao2006multiplicity,chen2010existence,ghoussoub2000multiple,catrina2001caffarelli,ghoussoub2004hardy} and the references therein. Alternatively, as in this paper, one prescribes the weighted mass $\rho^2$. Then $\lambda$ is not known in advance and appears as a Lagrange multiplier, and the solutions are called normalized solutions. They are the critical points of the energy functional
	\begin{align*}
		\e(u)=\frac12\int_{\R^N}\frac{|\na u|^2}{|x|^{2a}}\dx-\frac\beta q\int_{\R^N}\frac{|u|^q}{|x|^{bq}}\dx-\frac1\t\int_{\R^N}\frac{|u|^\t}{|x|^{b\t}}\dx
	\end{align*}
	on the weighted $L^2$-sphere
	\begin{align*}
		\S=\Big\{u:\int_{\R^N}|x|^{-2a}u^2\dx=\rho^2\Big\}.
	\end{align*}
	The function spaces and the notion of ground state are made precise in Section~2.
	
	The study of normalized solutions goes back to Cazenave and Lions \cite{cazenave1982orbital}, who considered $-\Delta u=\lambda u+|u|^{p-2}u$ with $\|u\|_2=\rho$. For $2<p<2+\frac4N,$ the energy is bounded from below on the $L^2$-sphere. They obtained minimizers and proved their orbital stability. For $p>2+\frac4N$ the energy is unbounded from below on the sphere, and minimization is no longer possible. In this case, Jeanjean \cite{jeanjean1997existence} obtained a solution of mountain-pass type, using the Pohozaev identity to bound the Palais--Smale sequences. For the combined nonlinearity $\alpha|u|^{q-2}u+|u|^{p-2}u$ with $2<q<p\le2^*:=\frac{2N}{N-2}$, Soave \cite{soave2020normalized1,soave2020normalized} treated the three regimes $2<q<p=2+\frac4N$, $2<q\le2+\frac4N<p\le2^*$ and $2+\frac4N<q<p\le2^*$. When $p=2^*$ and $q<2+\frac4N$, there is a second solution, of mountain-pass type with positive energy, besides the local minimizer: this is due to Jeanjean and Le \cite{jeanjean2022multiple} for $N\ge4$ and to Wei and Wu \cite{wei2022normalized} for $N\ge3$. Finally, the threshold device used below was introduced by Bartsch, Li, and Zou \cite{bartsch2023existence} for Sobolev critical systems.

	Normalized solutions with singular weights have also been studied. Cardoso, de Oliveira, and Miyagaki \cite{cardoso2024normalized} considered
	\begin{align*}
		-\Delta u+\lambda u=|x|^{-b}|u|^{q-2}u+|x|^{-d}|u|^{2^*(d)-2}u,\qquad 0<b,d<2,\qquad 2^*(d):=\frac{2(N-d)}{N-2},
	\end{align*}
	in the mass-supercritical range of $q$, using a minimax method of Jeanjean type. For the operator $-\Delta-\mu|x|^{-2}$ with Sobolev critical growth, normalized solutions were studied by Li and Zou \cite{li2023hardy}, by Fan, Li, and Tang \cite{fan2024hardy}, and by Kang and Tang \cite{kang2026multiple}. In particular, for $0<\mu<\frac{(N-2)^2}{4}$ and $2<q<2+\frac4N$, Kang and Tang \cite[Theorem 1.1]{kang2026multiple} obtain, under conditions on $\sqrt{\frac{(N-2)^2}{4}-\mu}$, a normalized solution of mountain-pass type, which together with the local minimizer of \cite{li2023hardy} yields two solutions. Via the substitution $v=|x|^{-a}u$, their problem corresponds to the case $\varsigma=0$, that is $b=a$, of Theorem~\ref{C1}. The case $b=a$ is excluded here because the assumption $b>a$ is used, for instance, in Lemma~\ref{B4sym}\textit{(ii)}, where the weight $|x|^{-(b-a)q}$ must be strictly decreasing. In the problem for $v$ above, by contrast, a Hardy potential occurs together with a Hardy--Sobolev critical term with $\varsigma>0$. Both nonlinearities carry weights, whereas in \cite{li2023hardy,fan2024hardy,kang2026multiple} the critical term has no weight ($\varsigma=0$) and in \cite{cardoso2024normalized} there is no Hardy potential. We are not aware of results on normalized solutions combining a Hardy potential with a Hardy--Sobolev critical term with $\varsigma>0$. Remark~\ref{Rlit} compares the restrictions on $\mu$ and on the coefficients in \cite{li2023hardy,fan2024hardy,kang2026multiple} with our results.
	
	We now recall the Caffarelli--Kohn--Nirenberg inequality \cite{caffarelli1984first}, which contains the Sobolev, Hardy, and Gagliardo--Nirenberg inequalities as special cases.
	
	\begin{theorem}[Caffarelli--Kohn--Nirenberg inequality \cite{caffarelli1984first}]\label{ckn}
		Let $p,r\ge1$, $q>0$, $0\le\delta\le1$ and $\bar\gamma,\bar\alpha,\bar\eta\in\R$ satisfy
		\begin{align*}
			\frac{1}{p}+\frac{\bar\alpha}{N}>0,\qquad\frac{1}{r}+\frac{\bar\eta}{N}>0,\qquad\frac{1}{q}+\frac{\bar\gamma}{N}>0
		\end{align*}
		and
		\begin{equation*}
			\frac{1}{q}+\frac{\bar\gamma}{N}=\delta\Big(\frac{1}{p}+\frac{\bar\alpha-1}{N}\Big)+(1-\delta)\Big(\frac{1}{r}+\frac{\bar\eta}{N}\Big).
		\end{equation*}
		Define $\bar\sigma$ by $\bar\gamma=\delta\bar\sigma+(1-\delta)\bar\eta$, and assume that, if $\delta>0$,
		\begin{equation*}
			0\le\bar\alpha-\bar\sigma,\qquad\text{and}\qquad\bar\alpha-\bar\sigma\le1\ \text{ when }\ \frac{1}{p}+\frac{\bar\alpha-1}{N}=\frac{1}{q}+\frac{\bar\gamma}{N}.
		\end{equation*}
		Then there is $C=C(N,r,p,q,\bar\gamma,\bar\alpha,\bar\eta)>0$ such that, for all $u\in C_0^\infty(\R^N)$,
		\begin{equation*}
			\||x|^{\bar\gamma}u\|_q\le C\,\||x|^{\bar\alpha}\na u\|_p^\delta\,\||x|^{\bar\eta}u\|_r^{1-\delta}.
		\end{equation*}
	\end{theorem}
	
	Applying Theorem~\ref{ckn} with
	\begin{align*}
		\bar\gamma=-b,\qquad\bar\alpha=\bar\eta=-a,\qquad p=r=2,
	\end{align*}
	we obtain
	\begin{equation}\label{c1}
		\||x|^{-b}u\|_q^q \leq \C \||x|^{-a}\nabla u\|^{\delta_qq}_2 ~\||x|^{-a}u\|_2^{(1-\delta_q)q},
		\tag{CKN}
	\end{equation}
	where
	\begin{align*}
		{\delta_q}=\frac{(N-2a+2b)q-2N}{2q}.
	\end{align*}
	This choice satisfies the hypotheses of Theorem~\ref{ckn}. Indeed, $\frac12-\frac aN>0$, and $bq<b\t<N$ gives $\frac1q-\frac bN>0$. The balance condition holds with $\delta=\delta_q$. Finally, $\bar\alpha-\bar\sigma=\frac{b-a}{\delta_q}\in(0,1)$, since $\delta_q$ is increasing in $q$, with $\delta_2=b-a$ and $\delta_{\t}=1$.  By density, \eqref{c1} holds for every $u\in\mathcal Y$ (see the proof of Lemma~\ref{B4sym}). We denote by $C_{a,b,q}>0$ the smallest constant such that \eqref{c1} holds for all $u\in\mathcal Y$. It depends on $q$.
	
	The \textit{mass-critical} exponent is
	\begin{equation*}
		q_c:=\frac{4+2N}{N-2a+2b},
	\end{equation*}
	that is, the exponent for which $q_c\delta_{q_c}=2$. It plays the role of $2+\frac4N$: for every $t\in\R$, the dilation $u\mapsto t\star u$ of Section~2 preserves $\S$, and under it the gradient term and the subcritical term in $\e$ scale like $e^{2t}$ and $e^{q\delta_qt}$, respectively. Since $q\delta_q<2$, $=2$ or $>2$ according as $q<q_c$, $q=q_c$ or $q>q_c$, we call these cases mass-subcritical, mass-critical and mass-supercritical. We also set
	\begin{align*}
		D:=N-2(1+a)+b,\qquad\mathcal T:=\frac dN\sr^{N/2d},
	\end{align*}
	where $\sr$ is the best constant of the critical Caffarelli--Kohn--Nirenberg inequality (Section~2), and $\mathcal T$ is the energy level at which compactness can be lost.
	
	Our main results are stated in Section~2. For $2<q<q_c$ and small $\beta$, Theorem~\ref{C1} provides two positive radial solutions: a ground state with negative energy, which is a local minimizer of $\e$ on $\S$, and a solution of mountain-pass type with positive energy. For $q_c\le q<\t$, the existence of a ground state is governed by the threshold
	\begin{align*}
		\beta_0:=\inf\big\{\beta\in(0,\hat\beta):\ m_\beta(\rho)<\mathcal T\big\},
	\end{align*}
	where $m_\beta(\rho)$ is the ground state level, $\hat\beta=\beta^*$ if $q=q_c$ and $\hat\beta=\infty$ if $q>q_c$, and $\beta^*$ is the explicit constant in \eqref{beta0def}. Theorems~\ref{Tcrit} and~\ref{Tsup} give a positive radial ground state of mountain-pass type for every $\beta\in(\beta_0,\hat\beta)$, and this interval is not empty. Theorem~\ref{Tbeta0} decides when the threshold is positive:
	\begin{align*}
		\beta_0>0\quad\Longleftrightarrow\quad q_c<\frac ND\quad\text{and}\quad q\le q^*,
	\end{align*}
	where $q^*$ is an explicit exponent. In that case, for $0<\beta<\beta_0$, the level $\mathcal T$ is not attained, and every normalized solution has energy larger than $\mathcal T$. For $q=q_c$ and $\beta\ge\beta^*$, the level is $0$ and is not attained (Theorem~\ref{Tcrit}). Finally, when $\beta_0=0$, the ground states converge after scaling, as $\beta\to0$, to an extremal of the Caffarelli--Kohn--Nirenberg inequality (Theorem~\ref{Tasym}).

	Two difficulties are already present in the unweighted problem, and we handle them as in \cite{jeanjean1997existence,soave2020normalized}. First, $\lambda$ is a Lagrange multiplier that is not known in advance, so Palais--Smale sequences of $\e|_{\S}$ need not be bounded. We therefore use Palais--Smale sequences that satisfy, in addition, $P_\beta(u_n)\to0$, where $P_\beta$ is the Pohozaev functional; this extra information gives the boundedness. Second, the critical term causes a loss of compactness. The gradient term and the critical term are both invariant under the dilations $u\mapsto\sigma^{\frac{N-2-2a}2}u(\sigma\,\cdot)$, $\sigma>0$, which also preserve radial symmetry. Hence a sequence of dilated extremals of $\sr$ converges weakly to $0$ in $\mathcal X$ while its energy stays at the level $\mathcal T$, and the embedding of $\mathcal X$ into $L^{\t}(\R^N;|x|^{-b\t})$ is not compact. For a Palais--Smale sequence at a level $m$ with $P_\beta(u_n)\to0$, Proposition~\ref{C32} therefore gives only the alternative
	\begin{align*}
		u_n\to u\ \text{ in }\mathcal X\qquad\text{or}\qquad u_n\rightharpoonup u\ \text{ with }\ \e(u)\le m-\mathcal T,
	\end{align*}
	up to a subsequence, and the minimax levels must be placed below $\mathcal T$.
	
	Two further difficulties arise from the weights. The level estimate is obtained with cut-off extremals of $\sr$, and their behavior depends on $a$ and $b$: they lie in $L^2(\R^N;|x|^{-2a})$ only when $a<\frac{N-4}{2}$, and the size of their subcritical norm changes at $q=\frac ND$ (Proposition~\ref{C11}). When $q_c<\frac ND$ and $q\le q^*$, they no longer bring the level below $\mathcal T$ for small $\beta$, and existence can only be proved for $\beta>\beta_0$. Finally, deciding whether $\beta_0$ is positive requires information on all ground states, not only on test functions.

		The positivity of $\beta_0$ has a simple explanation in the Emden--Fowler variable $\varphi(s)=e^{\gamma_as}u(e^s)$, $\gamma_a=\frac{N-2-2a}{2}$, in which the extremal profile decays like $e^{-\gamma_a|s|}$. If $a>\frac{N-4}{2}$, the mass of the profile is carried by its tail. If, in addition, $q\le\frac ND$, so is its subcritical norm, and the subcritical term cannot lower the energy of the cut-off extremals enough. For $\frac ND<q\le q^*$, the subcritical norm is carried by the core, but the energy it gains is still not larger than the error caused by the cut-off. The exponent $q^*$ is the one at which these two orders coincide. In both cases, the cut-off extremals do not bring the level below $\mathcal T$ for small $\beta$, and Section~\ref{Sbeta0} shows that no other function does (see Remark~\ref{b0R1}). By Remark~\ref{b0Rrange}\textit{(ii)}, this regime occurs only when $\frac{N-2}{2}-a<\frac{2}{N+2}$. In particular, it does not occur for $a=b=0$ (Remark~\ref{Rlit}).

	Sections~\ref{Sfib} and~\ref{Sproofs} adapt the scheme of \cite{soave2020normalized1,soave2020normalized,jeanjean2022multiple,wei2022normalized} to the weights. The compactness of the subcritical term comes from Proposition~\ref{Y1}, a weighted Strauss-type embedding (see \cite{xuan2005solvability,su2007weighted}), and the convergence in $L^2(\R^N;|x|^{-2a})$ in Proposition~\ref{C32} relies on $\lambda<0$, as in \cite{soave2020normalized,jeanjean2022multiple}. The Palais--Smale sequences are built from the fiber maps
	\begin{align*}
		\e(t\star u)=\frac{e^{2t}}{2}\||x|^{-a}\na u\|_2^2-\frac{\beta e^{q\delta_qt}}{q}\||x|^{-b}u\|_q^q-\frac{e^{\t t}}{\t}\||x|^{-b}u\|_{\t}^{\t},\qquad t\in\R,
	\end{align*}
	whose critical points are the values of $t$ for which $t\star u\in\M$. Recall that $q\delta_q<2$, $=2$ or $>2$ according as $q<q_c$, $q=q_c$ or $q>q_c$. Hence each fiber map has two critical points for $q<q_c$ and $\beta<\beta_1$, and exactly one for $q\ge q_c$ (Lemmas~\ref{C12} and~\ref{C10}). This leads to a minimizer on $\mathcal M^+_{\rho,\beta}$ and a minimax solution on $\mathcal M^-_{\rho,\beta}$ for $q<q_c$, and to a minimax over paths for $q\ge q_c$ (Lemma~\ref{LPS}). Both minimax arguments rely on Ghoussoub's principle \cite{ghoussoub1993duality}. The levels are placed below $\mathcal T$ by Lemmas~\ref{Lbeta0} and~\ref{Lbar} when $\beta$ is close to $\beta^*$ or larger than $\bar\beta$, and otherwise by the two-sided estimates of the cut-off extremals in Proposition~\ref{C11} (Lemmas~\ref{Lzero} and~\ref{C72}). For the second solution, we also use the interaction between the ground state and the extremal, as in \cite{jeanjean2022multiple,wei2022normalized}. Since $\beta\mapsto m_\beta(\rho)$ is nonincreasing (Lemma~\ref{Lmm}), compactness then holds for every $\beta>\beta_0$, a device taken from \cite{bartsch2023existence}.
	
	To prove that $\beta_0>0$, we argue by contradiction. If $\beta_0=0$, there are ground states $u_n$ for a sequence of couplings $\beta_n\to0$, and we show that this is impossible. The first step is to show that ground states are radial (Proposition~\ref{Psym}). This uses the substitution $v=|x|^{-a}u$ and Schwarz symmetrization, as in the proof of \cite[Lemma 2.1]{dolbeault2009symmetry} for extremals of the Caffarelli--Kohn--Nirenberg inequality. For radial functions, the problem reduces to an ordinary differential equation on the line in the variable $\varphi_n(s)=e^{\gamma_a s}u_n(e^s)$. As $\beta_n\to0$, the functions $\varphi_n$ approach the extremal profile, and their decay at infinity can be estimated precisely. Inserting these estimates in the Pohozaev identity gives a positive lower bound for $\beta_n$, which contradicts $\beta_n\to0$ (Section~\ref{Sbeta0}). The case $q=q_c$ is harder, because there the decay estimate must be sharp. The same analysis, without the contradiction, describes the ground states as $\beta\to0$ when $\beta_0=0$ (Theorem~\ref{Tasym}). Finally, for $q=q_c$ and $\beta\ge\beta^*$, the level can be computed explicitly (Theorem~\ref{Tcrit}\textit{(iii)}).
	
	Of our results, only the characterization of the parameters for which $\beta_0>0$ (Theorem~\ref{Tbeta0}) does not follow from these schemes. Theorem~\ref{Tcrit}\textit{(iii)} is the weighted counterpart of \cite[Theorem 1.1(2)]{wei2022normalized} (see Remark~\ref{Rcrit}), and Theorem~\ref{Tasym} is that of \cite[Theorem 1.4]{soave2020normalized}, which describes the behavior of ground states as the coupling tends to $0$. Its proof uses the blow-up analysis developed for Theorem~\ref{Tbeta0}.
	
	The paper is organized as follows. Section~2 introduces the function spaces and the variational setting and states the main results. Section~\ref{Sfib} contains the compactness result for Palais--Smale sequences and the analysis of the fiber maps, for $q<q_c$ and for $q\ge q_c$. Section~\ref{Sproofs} proves Theorems~\ref{C1}, \ref{Tcrit}, and~\ref{Tsup}, each proof being preceded by the energy level estimate it uses. Section~\ref{Sbeta0} proves Theorems~\ref{Tbeta0} and~\ref{Tasym} by a blow-up analysis in Emden--Fowler variables. Finally, Appendix~\ref{AppA} collects the regularity, the Pohozaev identity, the positivity, and the symmetry of solutions, and Appendix~\ref{AppB} the estimates for the cut-off extremals of $\sr$.
	
	\section{Functional framework and main results}\label{Smain}
	
	This section introduces the function spaces and the variational setting, and states the main results.
	
	For $2\le q<\infty$, let $L^q(\R^N;|x|^{-bq})$ be the weighted Lebesgue space
	\begin{align*}
		L^q(\R^N;|x|^{-bq}):=\Big\{u:\R^N\to\R \ \text{measurable}:\int_{\R^N}|x|^{-bq}|u|^q\dx<\infty\Big\},
	\end{align*}
	with the norm
	\begin{align*}
		\|u\|_{L^q(\R^N;|x|^{-bq})}:=\Big(\int_{\R^N}|x|^{-bq}|u|^q\dx\Big)^{1/q}.
	\end{align*}
	Its radial subspace is denoted by $L^q_r(\R^N;|x|^{-bq})$. The space $\mathcal{D}_{a}^{1,2}(\R^N)$ is the completion of $C_0^{\infty}(\R^N)$ with respect to the norm
	\begin{align*}
		\|u\|_a:=\Big(\int_{\R^N}|x|^{-2a}|\na u|^2\dx\Big)^{1/2},
	\end{align*}
	and $\mathcal{D}_{a,r}^{1,2}(\R^N)$ is its subspace of radial functions. We work in the space
	\begin{align*}
		\mathcal{X}:=\mathcal{D}_{a,r}^{1,2}(\R^N)\cap L^2_r(\R^N;|x|^{-2a}),\qquad \|u\|_{\mathcal{X}}:=\big(\||x|^{-a}\na u\|_2^2+\||x|^{-a}u\|_2^2\big)^{1/2}.
	\end{align*}
	We begin with the compactness of the embedding of $\mathcal{X}$ into $L^q_r(\R^N;|x|^{-qb})$.
	
	\begin{prop}\label{Y1}
		Let \( N \geq 3 \), \( 0 < a < \frac{N-2}{2} \), and \( a < b < 1 + a \). Then, for any \( q  \in (2, \t) \), the embedding
		\begin{align*}
			\mathcal{X} \hookrightarrow L^q_r(\mathbb{R}^N; |x|^{-qb} )
		\end{align*}
		is compact.
	\end{prop}
	
	\begin{proof}
		Let $\{u_n\}$ be bounded in $\mathcal X$. Up to a subsequence, $u_n\rightharpoonup u$ weakly in $\mathcal X$, and we must show that the convergence is strong in the target space. Set $w_n:=u_n-u$ and $\mathcal K:=\sup_n\|w_n\|_{\mathcal X}<\infty$, fix $q\in(2,\t)$ and $\var>0$, and write $\gamma_a:=\frac{N-2-2a}{2}$, as in \eqref{b0gk}. We first prove a pointwise bound for radial functions, and then estimate the weighted integral outside a large ball and inside it.
		
		Let $w\in\mathcal X$, and write $w(x)=w(r)$, $r=|x|$. For $0<r_1<r_2$, H\"older's inequality yields
		\begin{align}\label{Y1hol}
			|w(r_2)-w(r_1)|\le\Big(\int_{r_1}^{r_2}s^{2a+1-N}ds\Big)^{1/2}\Big(\int_{r_1}^{r_2}s^{N-1-2a}|w'(s)|^2ds\Big)^{1/2}.
		\end{align}
		Owing to the condition $2a+1-N<-1$, the function $w(r)$ admits a limit as $r\to\infty$. This limit must be $0$ because $\int_0^\infty s^{N-1-2a}w(s)^2ds<\infty$. Letting $r_2\to\infty$, we therefore obtain
		\begin{align}\label{Y1pt}
			|w(r)|\le C_{N,a}\,r^{-\gamma_a}\,\||x|^{-a}\na w\|_2,\qquad C_{N,a}:=\big(2\gamma_a|\mathbb S^{N-1}|\big)^{-1/2}.
		\end{align}
		In particular, for every $r>0$, the map $w\mapsto w(r)$ is a continuous linear functional on $\mathcal X$.
		
		We now estimate the integral outside $B_R$. By \eqref{Y1pt}, $|w_n(x)| \le C_{N,a}\mathcal K|x|^{-\gamma_a}$ for all $n$. Setting $\tau_q:=bq-2a+(q-2)\gamma_a$, we have $\tau_q>0$ since $bq>2a$, and therefore, for $R>0$,
		\begin{align*}
			\int_{|x|>R}\frac{|w_n|^q}{|x|^{bq}}\dx\le\big(C_{N,a}\mathcal K\big)^{q-2}R^{-\tau_q}\int_{|x|>R}\frac{w_n^2}{|x|^{2a}}\dx\le C_{N,a}^{q-2}\mathcal K^qR^{-\tau_q}.
		\end{align*}
		We fix $R$ so large that the right-hand side is less than $\frac\var3$ for all $n$.

		It remains to estimate the integral inside $B_R$. Let $\kappa:=N-q(b+\gamma_a)$, as in \eqref{b0gk}. Then $\kappa>0$, because $b+\gamma_a=\frac{N-2d}{2}$ and $q<\t=\frac{2N}{N-2d}$. Near the origin, \eqref{Y1pt} gives,  for $\delta\in(0,R)$,
		\begin{align*}
			\int_{B_\delta}\frac{|w_n|^q}{|x|^{bq}}\dx\le\big(C_{N,a}\mathcal K\big)^q|\mathbb S^{N-1}|\int_0^\delta r^{\kappa-1}dr=\big(C_{N,a}\mathcal K\big)^q|\mathbb S^{N-1}|\,\frac{\delta^\kappa}{\kappa},
		\end{align*}
		and we fix $\delta$ so small that this is less than $\frac\var3$ for all $n$. On the annulus, the sequence $\{w_n\}$ is bounded by $C_{N,a}\mathcal K\delta^{-\gamma_a}$ on $[\delta,R]$, and, by \eqref{Y1hol},
		\begin{align*}
			|w_n(r_2)-w_n(r_1)|\le|\mathbb S^{N-1}|^{-1/2}\delta^{\frac{2a+1-N}{2}}\mathcal K\,|r_2-r_1|^{1/2}\qquad(\delta\le r_1<r_2\le R).
		\end{align*}
		In addition, $w_n\rightharpoonup0$ in $\mathcal X$ and $w\mapsto w(r)$ is continuous, so $w_n(r)\to0$ for every $r>0$. The Arzel\`a--Ascoli theorem then implies that $w_n\to0$ uniformly on $[\delta,R]$. Since $bq<N$, it follows that
		\begin{align*}
			\int_{\delta<|x|<R}\frac{|w_n|^q}{|x|^{bq}}\dx\le\max_{[\delta,R]}|w_n|^q\int_{B_R}|x|^{-bq}\dx<\frac\var3
		\end{align*}
		for $n$ large. Together with the estimate outside $B_R$, this gives  $\int_{\R^N}|x|^{-bq}|w_n|^q\dx<\var$ for $n$ large. Consequently, $u_n\to u$ in $L^q_r(\R^N;|x|^{-bq})$.\QED
	\end{proof}
	
	The embedding of $\mathcal{X}$ into $L^{\t}_r(\R^N;|x|^{-\t b})$, on the other hand, is not compact, and the critical term is controlled by the best Caffarelli--Kohn--Nirenberg constant \cite{catrina2001caffarelli,bouchez2009extremal}
	\begin{equation*}
		\mathcal{S}(a,b):=\inf_{u\in\mathcal{D}_a^{1,2}(\R^N)\setminus\{0\}}\frac{\||x|^{-a}\na u\|_2^2}{\||x|^{-b}u\|_{\t}^{2}}.
	\end{equation*}
	For $0<a<\frac{N-2}{2}$ and $a<b<a+1$, the minimizers of $\sr$ are the functions $cU_\var$, with $c\ne0$ and $\var>0$, where
	\begin{equation}\label{c96}
		U_{\varepsilon}(x)=\frac{\left(2\cdot2^\sharp A_0\varepsilon\right)^{\frac{N-2d}{4d}}}{\left(\varepsilon+|x|^{\alpha}\right)^{\frac{N-2d}{2d}}},\qquad A_0:=\Big(\frac{N-2}{2}-a\Big)^2,\qquad \alpha:=\frac{2d(N-2-2a)}{N-2d}.
	\end{equation}
	This is due to Chou and Chu \cite[Theorem A]{chou1993best} (see also \cite[Proposition 2.6 and Remark 2.7]{catrina2001caffarelli}), who also proved that every positive solution of the corresponding Euler equation is radial \cite[Theorem B]{chou1993best}. Radial symmetry can fail in other ranges of the parameters: for $a<0$ and $b$ close to $a$, the minimizers are not radial \cite[Theorem 1.3]{catrina2001caffarelli}, \cite[Corollary 1.2]{felli2003perturbation}. This does not occur in the range $a>0$ considered here.
	
	The energy functional associated with \eqref{eq:P} is
	\begin{equation*}
		\e(u)=\frac{1}{2}\||x|^{-a}\na u\|_2^2-\frac{\beta}{q}\||x|^{-b}u\|_q^q-\frac{1}{\t}\||x|^{-b}u\|_{\t}^{\t},
	\end{equation*}
	restricted to the constraint
	\begin{align*}
		S_\rho:=\{u\in \mathcal{X}:\||x|^{-a}u\|_2^2=\rho^2\}.
	\end{align*}
	The functional $\e$ is well defined and of class $C^1$, and the same formula defines $\e$ on the larger space
	\begin{align*}
		\mathcal{Y}:=\mathcal{D}_{a}^{1,2}(\R^N)\cap L^2(\R^N;|x|^{-2a}),
	\end{align*}
	and $\mathcal{X}$ is the subspace of radial functions in $\mathcal Y$. We set $S^{\mathcal{Y}}_\rho:=\{u\in\mathcal{Y}:\||x|^{-a}u\|_2=\rho\}$, and a solution $u\in S^{\mathcal{Y}}_\rho$ of \eqref{eq:P} is called a ground state in $S^{\mathcal Y}_\rho$ if $\e(u)\le\e(w)$ for every solution $w\in S^{\mathcal{Y}}_\rho$ of \eqref{eq:P}. Similarly, a critical point $u$ of $\e|_{\S}$ is called a ground state of $\e|_{\S}$ if $\e(u)\le\e(w)$ for every critical point $w$ of $\e|_{\S}$.
	
	\begin{rem}\label{Rhardy}
		The substitution $v=|x|^{-a}u$ of the Introduction is an isomorphism of $\mathcal Y$ onto $\mathcal D^{1,2}(\R^N)\cap L^2(\R^N)$. For $u\in\mathcal Y$, $\mu=a(N-2-a)$ and $r>0$, it gives
		\begin{align*}
			\int_{\R^N}\frac{|\na u|^2}{|x|^{2a}}\dx&=\int_{\R^N}\Big(|\na v|^2-\mu\frac{v^2}{|x|^2}\Big)\dx,\\
			\int_{\R^N}\frac{u^2}{|x|^{2a}}\dx&=\int_{\R^N}v^2\dx,\\
			\int_{\R^N}\frac{|u|^r}{|x|^{br}}\dx&=\int_{\R^N}\frac{|v|^r}{|x|^{(b-a)r}}\dx;
		\end{align*}
		see \eqref{b4id} in the proof of Lemma~\ref{B4sym}. Thus \eqref{eq:P} is a normalized problem for the Hardy operator $-\Delta-\mu|x|^{-2}$ with a Hardy--Sobolev critical term and a weighted subcritical term, and $\sr$ is the corresponding Hardy--Sobolev constant. We use this formulation in Remark~\ref{Rlit} to compare our conditions with those of \cite{li2023hardy,fan2024hardy,kang2026multiple}.
	\end{rem}
	
	Since the functional $\e$ is unbounded from below on $\S$, it cannot be minimized directly. Instead, as in \cite{soave2020normalized1}, we use the Pohozaev manifold
	\begin{equation*}
		\mathcal{M}_{\rho,\beta}:=\{u\in S_{\rho}:P_\beta(u)=0\},\qquad m_\beta(\rho):=\inf_{\mathcal{M}_{\rho,\beta}}\e,
	\end{equation*}
	where
	\begin{align*}
		P_{\beta}(u)=\||x|^{-a}\na u\|_2^2-\beta\delta_q\||x|^{-b}u\|_q^q-\||x|^{-b}u\|_{\t}^{\t}.
	\end{align*}
	Every solution $u\in\S$ of \eqref{eq:P} lies in $\M$ and has a negative Lagrange multiplier, by the Pohozaev identity (Lemma~\ref{B2poh}).
	
	The functional $P_\beta$ is related to the minimax structure of $\e|_{\S}$ through the behavior of $\e$ under mass-preserving dilations. For $u\in S_\rho$ and $t\in\R$, let
	\begin{equation*}
		(t\star u)(x):=e^{\frac{N-2a}{2}t}u(e^tx).
	\end{equation*}
	Then $t\star u\in S_\rho$, and the fiber map of $u$ is
	\begin{align*}
		\Phi_{\beta}^u(t):=\e(t\star u)={}\frac{e^{2t}}{2}\||x|^{-a}\na u\|_2^2-\frac{\beta e^{q\delta_qt}}{q}\||x|^{-b}u\|_q^q-\frac{e^{\t t}}{\t}\||x|^{-b}u\|_{\t}^{\t}.
	\end{align*}
	
	\begin{rem}\label{CA1}
		Let $u\in\S$. Then $(\Phi^u_\beta)'(t)=P_\beta(t\star u)$, so that $t$ is a critical point of $\f$ if and only if $t\star u\in\M$.
	\end{rem}
	
	Using the fiber maps, we split $\mathcal{M}_{\rho,\beta}$ into three disjoint subsets,
	\begin{equation*}
		\mathcal{M}_{\rho,\beta}=\mathcal{M}_{\rho,\beta}^+\cup\mathcal{M}_{\rho,\beta}^-\cup\mathcal{M}_{\rho,\beta}^0,
	\end{equation*}
	where
	\begin{align*}
		\mathcal{M}_{\rho,\beta}^+&:=\{u\in \mathcal{M}_{\rho,\beta}:(\Phi^u_{\beta})''(0)>0\},\\
		\mathcal{M}_{\rho,\beta}^-&:=\{u\in \mathcal{M}_{\rho,\beta}:(\Phi^u_{\beta})''(0)<0\},\\
		\mathcal{M}_{\rho,\beta}^0&:=\{u\in \mathcal{M}_{\rho,\beta}:(\Phi^u_{\beta})''(0)=0\}.
	\end{align*}
	Set
	\begin{align*}
		\beta_1:=\frac{q(\t-2)}{2\C\rho^{(1-\delta_q)q}(\t-q\delta_q)}\left(\frac{\t\mathcal{S}(a,b)^{\t/2}(2-q\delta_q)}{2(\t-q\delta_q)}\right)^{\frac{2-q\delta_q}{\t-2}}
	\end{align*}
	and
	\begin{align*}
		\beta_2:=\frac{2\t d}{N\delta_q\C(\t-\delta_qq)\rho^{(1-\delta_q)q}}\left(\frac{\delta_qq\sr^{N/2d}}{2-\delta_qq}\right)^{\frac{2-\delta_qq}{2}}.
	\end{align*}
	These thresholds satisfy $\beta_1<\beta_2$. Indeed, with $\varpi:=q\delta_q\in(0,2)$, a direct computation yields
	\begin{align*}
		\log\frac{\beta_2}{\beta_1}=-\frac\varpi2\log\frac\varpi2-\Big(1-\frac\varpi2\Big)\log\Big(1-\frac\varpi2\Big)+\frac{2-\varpi}{\t-2}\log\frac{2(\t-\varpi)}{\t(2-\varpi)}.
	\end{align*}
	The first two terms are positive since $\frac\varpi2\in(0,1)$, and so is the last one, by $2(\t-\varpi)>\t(2-\varpi)$. Only $\beta_1$ enters the statement of Theorem~\ref{C1}, and $\beta_2$ is used in its proof. We now state the main results, starting with the mass-subcritical case.
	
	\begin{theorem}\label{C1}
		Let $N\ge3$, $\rho,~\beta>0$, $0<a<\frac{N-2}{2}$, $a<b<1+a$ and $2<q<q_c$. Assume that $\beta<\beta_1$. Then:
		\begin{itemize}
			\item[(i)] $\e|_{\S}$ has a positive radial ground state $\tilde{u}$, with Lagrange multiplier $\tilde\lambda<0$. Moreover, $\e(\tilde{u})<0$, and $\tilde{u}$ is an interior local minimizer of $\mathcal{E}_{\beta}$ on the set $\mathcal G_{R_0}$, where
			$$\mathcal{G}_k:=\{u\in\S:\n\|_2<k\}\quad(k>0)$$
			and $R_0>0$ is given in Lemma~\ref{C2}.
			\item[(ii)] $\e|_{\S}$ has a second positive radial critical point ${u}_0$, with $\e({u}_0)>0$ and Lagrange multiplier ${\lambda}_0<0$.
		\end{itemize}
		By the principle of symmetric criticality \cite{palais1979principle}, the critical points obtained in $\mathcal{X}$ are weak solutions of \eqref{eq:P} in $\R^N$. In addition, $\tilde u$ in \textit{(i)} is a ground state in $S^{\mathcal{Y}}_\rho$, and every ground state $u$ in $S^{\mathcal{Y}}_\rho$ is, up to sign, positive and radial, with $|x|^{-a}|u|$ radially nonincreasing (Proposition~\ref{Psym}).
	\end{theorem}
	
	\noindent The threshold $\beta_1$ is proportional to $\rho^{-(1-\delta_q)q}$. Thus Theorem~\ref{C1} applies, for every fixed $\beta>0$, to all sufficiently small $\rho>0$. The condition $\beta<\beta_1$ is used twice: first, it yields the geometry of $\e|_{\S}$ (Lemma~\ref{C2}); second, through $\beta<\beta_1<\beta_2$, it excludes alternative (I) of Proposition~\ref{C32} in the proof of Theorem~\ref{C1}\textit{(i)}. The case $\beta\ge\beta_1$ remains open here, and it is open even for the unweighted problem \cite[Section~1, after Remark~1.2]{soave2020normalized}.
	
	Turning to the mass-critical and mass-supercritical cases $q_c\le q<\t$, we set
	\begin{align}\label{beta0def}
		\beta^*:=\frac{q_c}{2C_{a,b,q_c}^{q_c}\rho^{\frac{4d}{N-2a+2b}}},\qquad\hat\beta:=\begin{cases}\beta^*,&q=q_c,\\ \infty,&q_c<q<\t,\end{cases}\qquad\beta_0:=\inf\big\{\beta\in(0,\hat\beta):m_\beta(\rho)<\mathcal T\big\},
	\end{align}
	where $C_{a,b,q}$ is the constant in \eqref{c1} (see Section~1) and $\mathcal T=\frac dN\sr^{N/2d}$. The set in the definition of $\beta_0$ is nonempty by Lemmas~\ref{Lbeta0} and~\ref{Lbar}. Moreover, $\beta\mapsto m_\beta(\rho)$ is nonincreasing (Lemma~\ref{Lmm}), and hence $m_\beta(\rho)<\mathcal T$ for every $\beta\in(\beta_0,\hat\beta)$. This is the inequality we need, since $\mathcal T$ is the level below which Palais--Smale sequences are compact (Proposition~\ref{C32}). Finally, we say that a ground state $\tilde u$ is of mountain-pass type if $\e(\tilde u)$ equals the minimax value $c_\Gamma$ over the class of paths $\Gamma$ in \eqref{cGamma}.
	
	\begin{theorem}[Mass-critical case]\label{Tcrit}
		Let $N\ge3$, $\rho>0$, $0<a<\frac{N-2}{2}$, $a<b<1+a$ and $q=q_c$.
		\begin{itemize}
			\item[(i)] $\beta_0<\beta^*$, and $\e|_{\S}$ has a ground state $\tilde u$ for every $\beta\in(\beta_0,\beta^*)$. An upper bound for $\beta_0$, in terms of an optimizer of \eqref{c1}, is given in Lemma~\ref{Lbeta0}.
			\item[(ii)] For $0<\beta<\beta^*$,
			\begin{align*}
				\Big(1-\frac{\beta}{\beta^*}\Big)^{\frac N{2d}}\mathcal T\le m_\beta(\rho)\le\min\Big\{\mathcal T,\ \omega_w^{-1}\Big(1-\frac{\beta}{\beta^*}\Big)^{\frac N{2d}}\mathcal T\Big\},
			\end{align*}
			where $\omega_w\in(0,1]$ is given in Lemma~\ref{Lbeta0}(iii).
			\item[(iii)] For $\beta\ge\beta^*$, $\M\ne\emptyset$, $m_\beta(\rho)=0$, and $m_\beta(\rho)$ is not attained.
		\end{itemize}
		In \textit{(i)}, $\tilde u$ is positive and radial, and solves \eqref{eq:P} for some $\lambda<0$. Moreover, $\e(\tilde u)=m_\beta(\rho)\in(0,\mathcal T)$, and $\tilde u$ is of mountain-pass type. Finally, $\tilde u$ is a ground state in $S^{\mathcal{Y}}_\rho$, and every ground state $u$ in $S^{\mathcal{Y}}_\rho$ is, up to sign, positive and radial, with $|x|^{-a}|u|$ radially nonincreasing (Proposition~\ref{Psym}).
	\end{theorem}
	
	\begin{rem}\label{Rcrit}
		Every solution of \eqref{eq:P} in $\S$ lies in $\M$ (Lemma~\ref{B2poh}), and, for solutions in $S^{\mathcal Y}_\rho$, the same holds for the Pohozaev set in $S^{\mathcal Y}_\rho$, as in Lemma~\ref{Lsol}\textit{(i)}. When $q=q_c$, in addition, $\e(u)=\frac dN\||x|^{-b}u\|_\t^\t>0$ for $u\in\M$. For $\beta\ge\beta^*$, the level in Theorem~\ref{Tcrit}\textit{(iii)} is zero, and therefore no solution has energy $m_\beta(\rho)$. This, however, does not exclude solutions with positive energy. For the unweighted problem, the corresponding non-attainment result is proved in \cite[Theorem 1.1(2)]{wei2022normalized}. If $q_c\ge\frac ND$, then $\beta_0=0$ by Theorem~\ref{Tbeta0}. Consequently, by Theorem~\ref{Tcrit}\textit{(i)},~\textit{(iii)}, $\e$ attains its infimum on $\M$ if and only if $0<\beta<\beta^*$.
	\end{rem}
	
	\begin{theorem}[Mass-supercritical case]\label{Tsup}
		Let $N\ge3$, $\rho>0$, $0<a<\frac{N-2}{2}$, $a<b<1+a$ and $q_c<q<\t$, and let
		\begin{align*}
			\bar\beta:=\frac{1}{\delta_q\C\rho^{(1-\delta_q)q}}\left(\frac{N(q\delta_q-2)}{2d\,q\delta_q\,\sr^{N/2d}}\right)^{\frac{q\delta_q-2}{2}}.
		\end{align*}
		Then $\beta_0\le\bar\beta$, and $\e|_{\S}$ has a ground state $\tilde u$ for every $\beta>\beta_0$. Moreover, $\tilde u$ is positive and radial, and solves \eqref{eq:P} for some $\lambda<0$. It satisfies $\e(\tilde u)=m_\beta(\rho)\in(0,\mathcal T)$, and it is of mountain-pass type. Finally, $\tilde u$ is a ground state in $S^{\mathcal{Y}}_\rho$, and every ground state $u$ in $S^{\mathcal{Y}}_\rho$ is, up to sign, positive and radial, with $|x|^{-a}|u|$ radially nonincreasing (Proposition~\ref{Psym}).
	\end{theorem}
	
	\noindent For $2<q<\t$ we also use the quantities
	\begin{equation}\label{b0gk}
		\gamma_a:=\frac{N-2-2a}{2}\in\Big(0,\frac{N-2}{2}\Big),\qquad \kappa:=N-q(b+\gamma_a),\qquad D=N-2(1+a)+b=b+2\gamma_a.
	\end{equation}
	The next theorem determines when $\beta_0$ is positive.
	
	\begin{theorem}[Sign of $\beta_0$]\label{Tbeta0}
		Let $N\ge3$, $\rho>0$, $0<a<\frac{N-2}{2}$, $a<b<1+a$ and $q_c\le q<\t$. Let $\beta_0$ be given by \eqref{beta0def}, and define
		\begin{align*}
			q^*:=\begin{cases}
				q_c, & \text{if } q_c\ge\frac ND,\\[1mm]
				\t-\dfrac{4(N-2-2a)}{(N-2d)(4+2a-N)}, & \text{if } q_c<\frac ND.
			\end{cases}
		\end{align*}
		\begin{itemize}
			\item[(i)] $\beta_0>0$ if and only if $q_c<\frac ND$ and $q\le q^*$. In particular, if $q_c\ge\frac ND$, then $\beta_0=0$ for every $q\in[q_c,\t)$. If $a\le\frac{N^2-8}{2(N+2)}$, then $q_c>\frac ND$ for every $b\in(a,1+a)$, and hence $\beta_0=0$.
			\item[(ii)] If $\beta_0>0$, then $m_\beta(\rho)=\mathcal T$ for $0<\beta\le\beta_0$. For $0<\beta<\beta_0$ this level is not attained, and every solution of \eqref{eq:P} in $S^{\mathcal Y}_\rho$ has energy larger than $\mathcal T$. Whether $m_{\beta_0}(\rho)$ is attained is not known.
		\end{itemize}
	\end{theorem}
	
	\begin{rem}\label{Rqstar}
		On the $q$-axis, the criterion of the previous theorem reads as follows. If $q_c<\frac ND$, then $q_c<\frac ND<q^*<\t$ by Lemma~\ref{b0L:alg}\textit{(i)}, so that $\beta_0>0$ exactly for $q\in[q_c,q^*]$, and $\beta_0=0$ for $q\in(q^*,\t)$. If $q_c\ge\frac ND$, on the other hand, $\beta_0=0$ for every $q\in[q_c,\t)$. The parameter algebra behind Theorem~\ref{Tbeta0}\textit{(i)} is collected in Section~\ref{ss:b0alg} (Lemma~\ref{b0L:alg} and Remark~\ref{b0Rrange}).
	\end{rem}
	
	\begin{rem}\label{Rlit}
		For the formal choice $a=b=0$, excluded here by the assumption $b>a$, we have $D=N-2$ and $q_c=2+\frac4N>\frac N{N-2}$ for every $N\ge3$, so the criterion gives $\beta_0=0$. This agrees with \cite[Theorem 1.1(3)]{wei2022normalized}, where a ground state is obtained for every coupling in the unweighted mass-supercritical case. More generally, by Remark~\ref{b0Rrange}\textit{(ii)}, $\beta_0>0$ requires $a>\frac{N^2-8}{2(N+2)}$, that is, $\frac{N-2}{2}-a<\frac{2}{N+2}$.
		
		We now compare our conditions with those in the literature on the Hardy operator. In the variable $v=|x|^{-a}u$ of the Introduction and of Remark~\ref{Rhardy}, $q_c$ is the mass-critical exponent for the weight $|x|^{-(b-a)q}$, and $\sqrt{\bar\mu-\mu}=\gamma_a$, where $\mu=a(N-2-a)$ and $\bar\mu:=\frac{(N-2)^2}{4}$. For $\varsigma=0$, that is, $b=a$, one has $\kappa=N-\frac{(N-2)q}{2}$.
		
		Let $\theta_{\rm LZ}(N,q)$ denote the function $\theta(N,q)$ of \cite[(1.16)]{li2023hardy}. For $2<q<2^*$,
		\begin{align*}
			\theta_{\rm LZ}(N,q)=\begin{cases}0,&q<2+\frac4N,\\[1mm] \dfrac{q(1-\gamma_q)}{q(1-\gamma_q)+2},&q\ge2+\frac4N,\end{cases}\qquad\text{where }\gamma_q:=\frac{N(q-2)}{2q}.
		\end{align*}
		Since $q(1-\gamma_q)=N-\frac{(N-2)q}{2}=\kappa$, we obtain
		$\theta_{\rm LZ}(N,q)=\frac{\kappa}{\kappa+2}$ for $q\ge2+\frac4N$. The condition $0\le\mu<\bar\mu-\theta_{\rm LZ}(N,q)^2$ of \cite[Theorem 1.4]{li2023hardy} therefore reads $\gamma_a(\kappa+2)>\kappa$, that is, $f(q)<0$, with $f$ as in \eqref{b0f}. This is the condition under which Theorem~\ref{Tbeta0} yields $\beta_0=0$ when $q_c<\frac ND<q$ (Remark~\ref{b0Rrange}\textit{(i)}). In this sense, the restriction on $\mu$ in \cite{li2023hardy} has the same origin as $\beta_0>0$ here: when it fails, the extremal profile carries the mass and the subcritical norm in its tail. We do not prove, however, that the restriction is necessary when $\varsigma=0$.
		
		In \cite[Theorem 1.1]{fan2024hardy}, on the other hand, there is no restriction on $\mu$, but the coefficient of the subcritical term in assumption (F5) must be large, which plays the role of the condition $\beta>\beta_0$. The conditions in \cite[Theorem 1.1]{kang2026multiple} concern the mountain-pass solution for $q<2+\frac4N$ and are also stated in terms of $\sqrt{\bar\mu-\mu}=\gamma_a$. For $\varsigma>0$, by contrast, Lemma~\ref{C72} needs no such condition. Its proof uses instead that $u=|x|^av$ is bounded and positive at the origin (Lemma~\ref{B3pos}).
	\end{rem}
	
	\noindent Our last result describes the ground states as the coupling tends to zero.
	
	\begin{theorem}\label{Tasym}
		Let $q_c\le q<\t$. Assume that $\beta_0=0$, that is (Theorem~\ref{Tbeta0}), $q_c\ge\frac ND$, or $q_c<\frac ND$ and $q>q^*$. Let $\beta_n\to0^+$, with $\beta_n<\beta^*$ if $q=q_c$. Let $u_n$ be a positive ground state of $\e|_{\S}$ with $\beta=\beta_n$ and Lagrange multiplier $\lambda_n$. Then
		\begin{align*}
			m_{\beta_n}(\rho)\to\mathcal T,\qquad\||x|^{-a}\na u_n\|_2^2\to\sr^{N/2d},\qquad\beta_n\||x|^{-b}u_n\|_q^q\to0,\qquad\lambda_n\to0.
		\end{align*}
		Moreover, up to a subsequence, there are $\sigma_n>0$ with $\sup_n\sigma_n<\infty$, and a positive radial function $U\in\mathcal D^{1,2}_a(\R^N)$ that attains $\sr$, such that
		\begin{align*}
			\sigma_n^{\frac{N-2-2a}{2}}u_n(\sigma_n\,\cdot)\to U\quad\text{in }\mathcal D^{1,2}_a(\R^N).
		\end{align*}
		If $a>\frac{N-4}{2}$, then $\sigma_n\to0$.
	\end{theorem}
	
	\section{Compactness and fiber maps}\label{Sfib}
	
	This section collects the tools for the existence proofs of Section~\ref{Sproofs}. We first state three results valid for every $2<q<\t$: a compactness result for Palais--Smale sequences (Proposition~\ref{C32}), the manifold structure of $\M$ (Lemma~\ref{Lman}), and a minimax principle of Ghoussoub (Lemma~\ref{a13}). We then study the fiber maps $\f$ for $q<q_c$ and for $q\ge q_c$ in Sections~\ref{ss:fibsub} and~\ref{ss:fibsup}, respectively. Section~\ref{ss:fibsup} ends with the minimax argument for $q\ge q_c$ (Lemma~\ref{LPS}).
	
	Critical points obtained in this way solve the equation: by the principle of symmetric criticality \cite{palais1979principle}, if $u\in\mathcal X$ is a critical point of $\e|_{\S}$, then $u$ is a weak solution of \eqref{eq:P} in $\R^N$. Regularity, positivity and the Pohozaev identity for such solutions are proved in Appendix~\ref{AppA} (Lemmas~\ref{B1reg}--\ref{B3pos}).
	
	Throughout the paper, for $u\in\mathcal Y$ we write
	\begin{align*}
		A(u):=\||x|^{-a}\na u\|_2^2,\qquad B(u):=\||x|^{-b}u\|_q^q,\qquad C(u):=\||x|^{-b}u\|_\t^\t.
	\end{align*}
	In this notation,
	\begin{align*}
		\e(u)=\frac12A(u)-\frac\beta qB(u)-\frac1\t C(u),\qquad P_\beta(u)=A(u)-\beta\delta_qB(u)-C(u).
	\end{align*}
	The letters $C$ and $c$ without an argument denote positive constants, which may change from line to line, whereas the functional $C(u)$ always carries its argument.
	
	\begin{prop}\label{C32}
		Let $N\ge3$, $2<q<\t$ and $\beta,\rho>0$. If $q=q_c$, assume in addition that $\beta<\beta^*$, where $\beta^*$ is given in \eqref{beta0def}. Let $\{u_n\}\subset \S$ be a Palais--Smale sequence for $\e|_{\S}$ at a level $m$, with
		\begin{align*}
			m<\frac{d}{N}\mathcal{S}(a,b)^{N/2d}\qquad\text{and}\qquad m\ne0,
		\end{align*}
		and suppose that $P_{\beta}(u_n)\to0$ as $n\to\infty$. Then one of the following alternatives holds:
		\begin{itemize}
			\item[(I)] up to a subsequence, $u_n\rightharpoonup u$ weakly in $\mathcal{X}$ but not strongly, where $u\in\S$, $u\ne 0$ is a solution of \eqref{eq:P} for some $\lambda<0$, and
			\begin{align*}
				\e(u)\le m-\frac{d}{N}\mathcal{S}(a,b)^{N/2d};
			\end{align*}
			\item[(II)] up to a subsequence, $u_n\to u$ strongly in $\mathcal{X}$, $\e(u)=m$, and $u$ solves \eqref{eq:P} for some $\lambda<0$.
		\end{itemize}
	\end{prop}
	\begin{proof}
		We follow \cite[Lemma 3.2 and proof of Proposition 3.1]{soave2020normalized}, with Proposition~\ref{Y1} in place of the compact embedding of $H^1_{\rm rad}(\R^N)$ into $L^q(\R^N)$, \eqref{c1} and the definition of $\sr$ in place of the Gagliardo--Nirenberg and Sobolev inequalities, and Lemma~\ref{B2poh} for the Pohozaev identity. We give the details, since the proposition is used in every existence proof and the boundedness proof for $q=q_c$ (Step~1, Case~2) differs from \cite[Lemma 3.2]{soave2020normalized}: the interpolation between $L^2$ and $L^{2^*}$ used there is not available for the weights. Throughout the proof we write $A_n:=A(u_n)$, $B_n:=B(u_n)$ and $C_n:=C(u_n)$.
		
		\noindent\textbf{Step 1.} $\{u_n\}$ is bounded in $\mathcal{X}$.
		
		\noindent\textbf{Case 1:} $q<q_c$. Using first $P_{\beta}(u_n)\to0$ and then \eqref{c1}, we obtain
		\begin{align*}
			\e(u_n)&=\Big(\frac{1}{2}-\frac{1}{\t}\Big)A_n-\beta\Big(\frac{1}{q}-\frac{\delta_q}{\t}\Big)B_n+o(1)\\
			&\ge \Big(\frac{1}{2}-\frac{1}{\t}\Big)A_n-\beta \C \rho^{(1-\delta_q)q}\Big(\frac{1}{q}-\frac{\delta_q}{\t}\Big)A_n^{\frac{q\delta_q}{2}}+o(1).
		\end{align*}
		Since $\e(u_n)\le m+1$ for $n$ large, it follows that
		\begin{align*}
			\Big(\frac{1}{2}-\frac{1}{\t}\Big)A_n\le \beta \C \rho^{(1-\delta_q)q}\Big(\frac{1}{q}-\frac{\delta_q}{\t}\Big)A_n^{\frac{q\delta_q}{2}}+m+1+o(1).
		\end{align*}
		This inequality and $q\delta_q<2$ imply that $\{u_n\}$ is bounded in $\mathcal{X}$.
		
		\noindent\textbf{Case 2:} $q=q_c$. We first bound $\{C_n\}$. Since $q\delta_q=2$ and $P_{\beta}(u_n)\to0$, for $n$ large we have
		\begin{align*}
			m+1\ge \e(u_n)=\e(u_n)-\frac12P_\beta(u_n)+o(1)=\frac{d}{N}C_n+o(1).
		\end{align*}
		Thus $\{C_n\}$ is bounded. For $\{A_n\}$, \eqref{c1} and the definition of $\beta^*$ yield $\beta\delta_qB_n\le\frac{\beta}{\beta^*}A_n$, which, combined with $P_\beta(u_n)\to0$, leads to
		\begin{align*}
			\Big(1-\frac{\beta}{\beta^*}\Big)A_n\le C_n+o(1).
		\end{align*}
		The coefficient on the left is positive by the assumption $\beta<\beta^*$, and so $\{u_n\}$ is bounded in $\mathcal{X}$.
		
		\noindent\textbf{Case 3:} $q>q_c$. From $P_\beta(u_n)\to0$ we obtain
		\begin{align*}
			\e(u_n)=\e(u_n)-\frac12P_\beta(u_n)+o(1)=\beta\Big(\frac{\delta_q}{2}-\frac{1}{q}\Big)B_n+\frac{d}{N}C_n+o(1).
		\end{align*}
		Both coefficients are positive, because $q\delta_q>2$, so $\{B_n\}$ and $\{C_n\}$ are bounded. The relation $A_n=\beta\delta_qB_n+C_n+o(1)$ then shows that $\{A_n\}$ is bounded as well.
		
		\noindent\textbf{Step 2.} \textbf{Existence of Lagrange multipliers \boldmath{${\lambda_n\to\lambda}$}.} Proposition~\ref{Y1}, applied to the bounded sequence $\{u_n\}$, provides $u\in \spa$ such that, up to a subsequence, $u_n\rightharpoonup u$ in $\spa$, $u_n\to u$ in $L^q_r(\R^N;|x|^{-bq})$, and $u_n\to u$ a.e. in $\R^N$. In addition, since $\{u_n\}$ is a bounded Palais--Smale sequence for $\e|_{S_{\rho}}$, there is $\{\lambda_n\}\subset\R$ such that, for every $\phi\in\mathcal{X}$,
		\begin{equation}\label{c4}
			\int_{\R^N}\frac{\na u_n\cdot\na \phi}{|x|^{2a}}\dx-\lambda_n\int_{\R^N}\frac{u_n\phi}{|x|^{2a}}\dx-\beta\int_{\R^N}\frac{|u_n|^{q-2}u_n\phi}{|x|^{bq}}\dx-\int_{\R^N}\frac{|u_n|^{\t-2}u_n\phi}{|x|^{b\t}}\dx=o(1)\|\phi\|.
		\end{equation}
		Taking $\phi=u_n$ in \eqref{c4}, we obtain
		\begin{align*}
			\lambda_n\rho^2=A_n-\beta B_n-C_n+o(1).
		\end{align*}
		By Step~1, $\{\lambda_n\}$ is therefore bounded, and, up to a subsequence, $\lambda_n\to\lambda\in \R$. To identify $\lambda$, we insert $A_n-C_n=\beta\delta_qB_n+o(1)$, which follows from $P_\beta(u_n)\to0$, into the previous identity and find
		\begin{align*}
			\lambda\rho^2=\lim_{n\to\infty}\beta(\delta_q-1)B_n=\beta(\delta_q-1)B(u).
		\end{align*}
		It follows from $\delta_q<1$ that $\lambda\le0$, and that $\lambda=0$ if and only if $u\equiv 0$.
		
		\noindent\textbf{Step 3.} {\boldmath$u\not\equiv0.$} Suppose, by contradiction, that $u\equiv0$. As $\{u_n\}$ is bounded in $\mathcal{X}$, we may assume that $A_n\to l\in \R$. Since $P_{\beta}(u_n)\to0$ and $B_n\to0$, it follows that
		\begin{align*}
			C_n= A_n-\beta\delta_qB_n+o(1)\to l.
		\end{align*}
		On the other hand, $\sr\, l^{2/\t}\le l$ by the definition of $\sr$, so either $l=0$ or $l\ge\sr^{N/2d}$. We exclude both possibilities. Suppose first that $l\ge\sr^{N/2d}$. Using $\e(u_n)\to m$ and $P_{\beta}(u_n)\to 0$, we obtain
		\begin{align*}
			m+o(1)=\e(u_n)-\frac1\t P_\beta(u_n)=\frac{d}{N}A_n-\frac{\beta}{q}\Big(1-\frac{q\de_q}{\t}\Big)B_n=\frac{d}{N}l+o(1).
		\end{align*}
		Thus $m=\frac{d}{N}l\ge\frac dN\sr^{N/2d}$, contradicting $m<\frac dN\sr^{N/2d}$. If $l=0$, then $A_n\to0$, $B_n\to0$ and $C_n\to0$, so $\e(u_n)\to 0$, contradicting $m\ne0$. Therefore $u\not\equiv0$.
		
		We now show that $u$ solves the limit problem and lies on $\S$. Passing to the limit in \eqref{c4} and applying Lemma~\ref{B2poh} with $c_0=1$, $\mu_1=\lambda$, $\mu_2=\beta$ and $\mu_3=1$, we obtain
		\begin{equation}\label{c5}
			\e'(u)\phi=\lambda\int_{\R^N}|x|^{-2a}u\phi\dx\quad\text{for all }\phi\in\spa,\qquad\text{and}\qquad P_\beta(u)=0.
		\end{equation}
		Here, the first identity is the limit of \eqref{c4}, and the second is the Pohozaev identity. Taking $\phi=u$ in \eqref{c5} and using $P_\beta(u)=0$ and Step~2, we find
		\begin{align*}
			\lambda\||x|^{-a}u\|_2^2=\beta(\delta_q-1)B(u)=\lambda\rho^2.
		\end{align*}
		By Step~2 and the first part of this step, $\lambda<0$, and therefore $\||x|^{-a}u\|_2=\rho$, that is, $u\in\S$.
		
		\noindent\textbf{Step 4.} \textbf{Conclusion.} Let $v_n:=u_n-u$. Then $v_n\rightharpoonup0$ in $\mathcal{X}$, and hence
		\begin{equation}\label{c2}
			A_n=A(u)+A(v_n)+o(1).
		\end{equation}
		Similarly, by the Brezis--Lieb lemma \cite{brezis1983relation},
		\begin{align*}
			C_n=C(u)+C(v_n)+o(1).
		\end{align*}
		These two expansions, together with $P_{\beta}(u_n)\to0$ and $B_n\to B(u)$, yield
		\begin{align*}
			A(u)+A(v_n)=\beta\delta_qB(u)+C(u)+C(v_n)+o(1).
		\end{align*}
		Since $P_{\beta}(u)=0$, it follows that, up to a subsequence,
		\begin{align*}
			\lim_{n\to\infty}A(v_n)=\lim_{n\to\infty}C(v_n)=:\La.
		\end{align*}
		By the definition of $\sr$, $\sr\,\La^{2/\t}\le\La$, so either $\La\ge\sr^{N/2d}$ or $\La=0$. We consider the two cases separately.
		
		\noindent\textbf{Case 1:} $\La\ge\sr^{N/2d}$. Since $B_n\to B(u)$, we have
		\begin{align*}
			m=\lim_{n\to\infty}\e(u_n)&=\lim_{n\to\infty}\Big(\e(u)+\frac{1}{2}A(v_n)-\frac{1}{\t}C(v_n)\Big)\\
			&=\e(u)+\frac{d}{N}\La\ge \e(u)+\frac{d}{N}\sr^{N/2d}.
		\end{align*}
		Therefore
		\begin{align*}
			\e(u)\le m-\frac{d}{N}\sr^{N/2d},
		\end{align*}
		and alternative (I) holds.
		
		\noindent\textbf{Case 2:} $\La=0$. Then, by \eqref{c2}, $A(u_n-u)\to0$, that is, $u_n\to u$ in $\mathcal{D}_{a,r}^{1,2}(\R^N)$ and in $\cri$. It remains to prove that $u_n\to u$ in $L^2_r(\R^N;|x|^{-2a})$. To this end, we take $\phi=u_n-u$ in \eqref{c4} and \eqref{c5}, which yields
		\begin{align*}
			&\int_{\R^N}|x|^{-2a}|\na (u_n-u)|^2\dx-\int_{\R^N}|x|^{-2a}(\lambda_nu_n-\lambda u)(u_n-u)\dx\\
			&\quad=\beta\int_{\R^N}|x|^{-bq}\left(|u_n|^{q-2}u_n-|u|^{q-2}u\right)\left(u_n-u\right)\dx\\
			&\qquad+\int_{\R^N}|x|^{-b\t}\left(|u_n|^{\t-2}u_n-|u|^{\t-2}u\right)\left(u_n-u\right)\dx+o(1).
		\end{align*}
		As $u_n\to u$ in $\mathcal{D}^{1,2}_{a,r}(\R^N)$, in $\q$ and in $\cri$, all terms except the second one on the left tend to $0$.  For the second term, we write $\lambda_nu_n-\lambda u=\lambda_n(u_n-u)+(\lambda_n-\lambda)u$. Since $\lambda_n\to\lambda<0$, we conclude that
		\begin{align*}
			\lim_{n\to\infty}\||x|^{-a}(u_n-u)\|_2=0.
		\end{align*}
		Thus, alternative (II) holds.\QED
	\end{proof}
	
	\begin{lemma}\label{Lman}
		Let $2<q<\t$ and $\beta>0$. If $\mathcal M^0_{\rho,\beta}=\emptyset$, then $\M$ is a $C^1$ manifold of codimension $2$ in $\mathcal X$, that is, of codimension $1$ in $\S$.
	\end{lemma}
	\begin{proof}
		Let $G(u):=\||x|^{-a}u\|_2^2$. The maps $G$ and $P_\beta$ are of class $C^1$ on $\mathcal X$, and
		\begin{align*}
			\M=\{u\in\mathcal X:\ G(u)=\rho^2,\ P_\beta(u)=0\}.
		\end{align*}
		It therefore suffices to show that $dG(u)$ and $dP_\beta(u)$ are linearly independent for $u\in\M$. Since $dG(u)\ne0$, we suppose, by contradiction, that $dP_\beta(u)=\nu\,dG(u)$ in $\mathcal X^*$ for some $\nu\in\R$. In view of the rotation invariance of both functionals, the principle of symmetric criticality \cite{palais1979principle} shows that this identity also holds in $\mathcal Y^*$, that is, $u$ is a weak solution of
		\begin{align*}
			-2\operatorname{div}(|x|^{-2a}\na u)=2\nu|x|^{-2a}u+\beta q\delta_q|x|^{-bq}|u|^{q-2}u+\t|x|^{-b\t}|u|^{\t-2}u.
		\end{align*}
		This is the auxiliary equation \eqref{auxeq} with $c_0=2$, $\mu_1=2\nu$, $\mu_2=\beta q\delta_q$ and $\mu_3=\t$, to which Lemma~\ref{B2poh} applies, as that lemma is stated for general $c_0$. In particular, the second identity of Lemma~\ref{B2poh} yields
		\begin{align*}
			2A(u)=\beta q\delta_q^2B(u)+\t C(u),
		\end{align*}
		that is, $(\Phi^u_\beta)''(0)=0$. Together with $P_\beta(u)=0$, this means that $u\in\mathcal M^0_{\rho,\beta}=\emptyset$, a contradiction.\QED
	\end{proof}
	
	Finally, we recall the notion of a homotopy-stable family and the minimax principle of Ghoussoub.
	
	\begin{definition}
		\cite{ghoussoub1993duality} Let $Z$ be a topological space, and let $\mathcal B$ be a closed subset of $Z$. A class $\mathcal{F}$ of compact subsets of $Z$ is called a homotopy-stable family with extended boundary $\mathcal B$ if $\xi(\{1\}\times \mathcal A)\in \mathcal{F}$ for every $\mathcal A\in\mathcal{F}$ and every $\xi \in C([0,1]\times Z;Z)$ with $\xi(t,x)=x$ for all $(t,x)\in (\{0\}\times Z)\cup ([0,1]\times \mathcal B)$.
	\end{definition}
	
	\begin{lemma}\label{a13}
		\cite{ghoussoub1993duality} Let $\varphi$ be a $C^1$ function on a complete connected $C^1$-Finsler manifold $Z$ without boundary, and let $\mathcal{F}$ be a homotopy-stable family of compact subsets of $Z$ with extended boundary $\mathcal B$. Set
		\begin{align*}
			m=m(\varphi,\mathcal{F}):=\inf_{\mathcal A\in\mathcal F}\max_{\mathcal A}\varphi,
		\end{align*}
		and let $F$ be a closed subset of $Z$ such that
		\begin{itemize}
			\item[(i)] $(\mathcal A\cap F)\setminus \mathcal B\ne \emptyset$ for every $\mathcal A\in \mathcal{F}$, and
			\item[(ii)] $\displaystyle\sup_{\mathcal B}\varphi\le m\le \inf_{F}\varphi$.
		\end{itemize}
		Then, for every sequence $\{\mathcal A_n\}_{n\in \mathbb N}$ in $\mathcal{F}$ with $\displaystyle\lim_{n\to\infty}\sup_{\mathcal A_n} \varphi=m$, there is a sequence $\{x_n\}$ in $Z$ such that
		\begin{align*}
			\lim_{n\to\infty}\varphi(x_n)=m,\qquad\lim_{n\to\infty}\|d\varphi(x_n)\|=0,\qquad\lim_{n\to\infty}\operatorname{dist}(x_n,F)=0,\qquad\lim_{n\to\infty}\operatorname{dist}(x_n,\mathcal A_n)=0.
		\end{align*}
	\end{lemma}
	
	\subsection{Fiber maps for \texorpdfstring{$2<q<q_c$}{2<q<qc}}\label{ss:fibsub}
	Throughout this subsection, $2<q<q_c$. Lemmas~\ref{C2}--\ref{C31} are weighted versions of \cite[Lemmas 4.1 and 4.2 and (4.2)]{soave2020normalized}, which in turn rest on \cite[Lemmas 5.1--5.5]{soave2020normalized1}. Their proofs follow these references, with \eqref{c1} and the definition of $\sr$ in place of the Gagliardo--Nirenberg and Sobolev inequalities, so we give details only where the weights change the argument. The constants $\beta_1$ and $\beta_2$ are those of \cite[Section 4]{soave2020normalized} for the present problem.
	
	We first bound $\e$ from below by a function of one real variable. By the definition of $\sr$ and \eqref{c1}, for $u\in\S$ we have
	\begin{align*}
		\e(u)\ge\frac{1}{2}\n\|_2^2-\frac{\beta}{q}\C\rho^{(1-\delta_q)q}\n\|_2^{q\delta_q}-\frac{1}{\t\mathcal{S}(a,b)^{\t/2}}\n\|_2^\t.
	\end{align*}
	Accordingly, we define $\mathcal L:\R^+\to\R$ by
	\begin{align*}
		\mathcal L(t):=\frac{1}{2}t^2-\frac{\beta}{q}\C\rho^{(1-\delta_q)q}t^{q\delta_q}-\frac{1}{\t\mathcal{S}(a,b)^{\t/2}}t^{\t}.
	\end{align*}
	If $\beta>0$ and $q\delta_q<2$, then $\mathcal L(0^+)=0^-$ and $\mathcal L(+\infty)=-\infty$.
	
	\begin{lemma}\label{C2}
		For $\beta<\beta_1$, the function $\mathcal L$ has exactly two critical points: a strict local minimum at a negative level and a strict global maximum at a positive level. Moreover, there are $R_0:=R_0(\beta,\rho)$ and $R_1:=R_1(\beta,\rho)$ such that $\mathcal L(R_0)=\mathcal L(R_1)=0$, and $\mathcal L(t)>0$ if and only if $t\in(R_0,R_1)$.
	\end{lemma}
	\begin{proof}
		This is \cite[Lemma 4.1]{soave2020normalized} (see also \cite[Lemma 5.1]{soave2020normalized1}) applied to $\mathcal L$, so we only identify the condition on $\beta$. Let
		\begin{align*}
			\mathcal H(t):=\frac12t^{2-q\delta_q}-\frac{t^{\t-q\delta_q}}{\t\sr^{\t/2}}.
		\end{align*}
		Then $\mathcal L(t)>0$ if and only if $\mathcal H(t)>\frac\beta q\C\rho^{(1-\delta_q)q}$, and
		\begin{align*}
			\max_{t>0}\mathcal H(t)=\frac{\t-2}{2(\t-q\delta_q)}\Big(\frac{\t\sr^{\t/2}(2-q\delta_q)}{2(\t-q\delta_q)}\Big)^{\frac{2-q\delta_q}{\t-2}}.
		\end{align*}
		Consequently, the condition in \cite[Lemma 4.1]{soave2020normalized} becomes $\beta<\beta_1$.\QED
	\end{proof}
	
	\begin{lemma}\label{C12}
		Let $\beta<\beta_1$. Then $\mathcal M^0_{\rho,\beta}=\emptyset$, and $\M$ is a $C^1$ manifold of codimension $1$ in $\S$. For every $u\in\S$, the function $\f$ has exactly two critical points $t_1<t_2$ and two zeros $s_1<s_2$, with $t_1<s_1<t_2<s_2$, and the following hold.
		\begin{itemize}
			\item[(i)] $t_1\star u\in\mathcal M^+_{\rho,\beta}$ and $t_2\star u\in\mathcal M^-_{\rho,\beta}$. If $t\star u\in\M$, then $t=t_1$ or $t=t_2$.
			\item[(ii)] $\||x|^{-a}\na(t\star u)\|_2\le R_0$ for every $t<s_1$, and
			\begin{align*}
				\e(t_1\star u)=\min\big\{\e(t\star u):\ t\in\R,\ \||x|^{-a}\na(t\star u)\|_2\le R_0\big\}<0.
			\end{align*}
			\item[(iii)] $\e(t_2\star u)=\max_{t\in\R}\e(t\star u)>0$, and $\f$ is strictly decreasing on $(t_2,\infty)$.
			\item[(iv)] The maps $u\in\S\mapsto t_1$ and $u\in\S\mapsto t_2$ are of class $C^1$.
			\item[(v)] $\mathcal M^+_{\rho,\beta}\subset\mathcal G_{R_0}$, and
			\begin{align*}
				\sup_{\mathcal M^+_{\rho,\beta}}\e\le0<c_\beta(\rho):=\inf_{\mathcal M^-_{\rho,\beta}}\e.
			\end{align*}
		\end{itemize}
	\end{lemma}
	\begin{proof}We first show that $\mathcal M^0_{\rho,\beta}=\emptyset$, as in \cite[Lemma 5.2]{soave2020normalized1}. Let $u\in\mathcal M^0_{\rho,\beta}$, write $A=A(u)$, $B=B(u)$, $C=C(u)$, and set $\varpi:=q\delta_q\in(0,2)$. Combining $P_\beta(u)=0$ and $(\Phi^u_\beta)''(0)=2A-\beta\delta_q\varpi B-\t C=0$, we obtain
		\begin{align*}
			(2-\varpi)A=(\t-\varpi)C,\qquad(\t-2)A=\beta\delta_q(\t-\varpi)B.
		\end{align*}
		On the other hand, $C\le\sr^{-\t/2}A^{\t/2}$ by the definition of $\sr$, and $B\le\C\rho^{(1-\delta_q)q}A^{\varpi/2}$ by \eqref{c1}. Inserting these bounds into the two identities, we obtain
		\begin{align*}
			\Big(\frac{(2-\varpi)\sr^{\t/2}}{\t-\varpi}\Big)^{\frac{2}{\t-2}}\le A\le\Big(\frac{\beta\delta_q(\t-\varpi)\C\rho^{(1-\delta_q)q}}{\t-2}\Big)^{\frac{2}{2-\varpi}}.
		\end{align*}
		These two bounds are incompatible if $\beta<\beta_M$, where
		\begin{align*}
			\beta_M:=\frac{\t-2}{\delta_q(\t-\varpi)\C\rho^{(1-\delta_q)q}}\Big(\frac{(2-\varpi)\sr^{\t/2}}{\t-\varpi}\Big)^{\frac{2-\varpi}{\t-2}}.
		\end{align*}
		It remains to compare $\beta_1$ with $\beta_M$. To this end, let $x:=\frac\varpi2\in(0,1)$ and $y:=\frac\t2>1$. By the definitions of $\beta_1$ and $\beta_M$,
		\begin{align*}
			\frac{\beta_1}{\beta_M}=\frac\varpi2\Big(\frac\t2\Big)^{\frac{2-\varpi}{\t-2}}=x\,y^{\frac{1-x}{y-1}}<x\,e^{1-x}\le1,
		\end{align*}
		because $\log y<y-1$, and $x\mapsto xe^{1-x}$ is increasing on $(0,1]$. Therefore $\mathcal M^0_{\rho,\beta}=\emptyset$ for $\beta<\beta_1$.
		
		Consequently, $\M$ is a $C^1$ manifold by Lemma~\ref{Lman}.
		
		Parts \textit{(i)}--\textit{(iv)} follow as in \cite[Lemma 4.2]{soave2020normalized} (see also \cite[Lemma 5.3]{soave2020normalized1}), with the function $\mathcal L$ of Lemma~\ref{C2} in place of $h$ there. This replacement is possible because
		\begin{align*}
			\f(t)\ge \mathcal L\big(e^t\||x|^{-a}\na u\|_2\big),\qquad t\in\R.
		\end{align*}
		In \textit{(i)} we also use Remark~\ref{CA1} and $\mathcal M^0_{\rho,\beta}=\emptyset$, and in \textit{(iv)} the implicit function theorem.
		
		To prove \textit{(v)}, let $u\in\mathcal M^+_{\rho,\beta}$. By \textit{(i)}, $t_1=0$, so \textit{(ii)} yields $\||x|^{-a}\na u\|_2\le R_0$ and $\e(u)<0$. Together with $\e(u)\ge \mathcal L(\||x|^{-a}\na u\|_2)$ and $\mathcal L(R_0)=0$, these bounds imply that $\||x|^{-a}\na u\|_2<R_0$. Next, let $u\in\mathcal M^-_{\rho,\beta}$. Then $t_2=0$, and \textit{(iii)} and Lemma~\ref{C2} lead to
		\begin{align*}
			\e(u)=\max_{t\in\R}\f(t)\ge\max_{t\in\R}\mathcal L\big(e^t\||x|^{-a}\na u\|_2\big)=\max_{s>0}\mathcal L(s)>0.
		\end{align*}
		As this bound does not depend on $u$, \textit{(v)} follows.\QED
	\end{proof}
	
	\begin{lemma}\label{C31}
		The infimum satisfies $m_\beta(\rho)=\inf_{\mathcal G_{R_0}}\e\in(-\infty,0)$. Moreover, for $r>0$ small enough,
		\begin{align*}
			m_{\beta}(\rho)=\inf_{\mathcal{M}_{\rho,\beta}}\e=\inf_{\mathcal{M}_{\rho,\beta}^+}\e\qquad\text{and}\qquad m_{\beta}(\rho)<\inf_{\overline{\mathcal{G}_{R_0}}\setminus \mathcal{G}_{R_0-r}}\e.
		\end{align*}
	\end{lemma}
	\begin{proof}We follow \cite[(4.2)]{soave2020normalized} (see also \cite[Lemma 5.5]{soave2020normalized1}), using $\e(u)\ge \mathcal L(\||x|^{-a}\na u\|_2)$, Lemma~\ref{C2} and Lemma~\ref{C12}.\QED
	\end{proof}
	
	\subsection{Fiber maps for \texorpdfstring{$q_c\le q<\t$}{qc<=q<2\#}}\label{ss:fibsup}
	
	Throughout this subsection, $q_c\le q<\t$, and $\beta<\beta^*$ if $q=q_c$, with $\beta^*$ as in \eqref{beta0def}. In contrast with Section~\ref{ss:fibsub}, the fiber map $\f$ now has exactly one critical point, which is a strict maximum point.
	
	Lemmas~\ref{C10} and~\ref{C33} correspond to \cite[Lemmas 5.1 and 5.2]{soave2020normalized} for $q=q_c$, and to \cite[Lemmas 6.1--6.3]{soave2020normalized} for $q>q_c$, while Lemma~\ref{Lmm}\textit{(i)} corresponds to \cite[Lemmas 8.1 and 8.2]{soave2020normalized}. As before, \eqref{c1} and the definition of $\sr$ replace the Gagliardo--Nirenberg and Sobolev inequalities in the proofs.
	
	For $q=q_c$, we have $q\delta_q=2$ and $(1-\delta_q)q=q-2=\frac{4d}{N-2a+2b}$, so \eqref{c1} and the definition of $\beta^*$ yield
	\begin{align}\label{c1crit}
		\beta\delta_qB(u)\le\frac{\beta}{\beta^*}A(u)\qquad\text{for }u\in\S.
	\end{align}
	
	\begin{lemma}\label{C10}
		Let $q_c\le q<\t$, and $\beta<\beta^*$ if $q=q_c$. Then $\mathcal M^0_{\rho,\beta}=\emptyset$, $\M=\mathcal M^-_{\rho,\beta}$, and $\M$ is a $C^1$ manifold of codimension $1$ in $\S$. For every $u\in\S$, there is a unique $t_u\in\R$ with $t_u\star u\in\M$. The point $t_u$ is the unique critical point of $\f$, and it is a strict maximum point at a positive level. Moreover:
		\begin{itemize}
			\item[(i)] $\f$ is concave and strictly decreasing on $(t_u,\infty)$;
			\item[(ii)] $P_\beta(u)<0$ if and only if $t_u<0$; in particular, $t_u=0$ if $u\in\M$;
			\item[(iii)] the map $u\in\S\mapsto t_u\in\R$ is of class $C^1$.
		\end{itemize}
	\end{lemma}
	\begin{proof}
		We first study the fiber map of a fixed function. Let $u\in\S$, and write $A=A(u)$, $B=B(u)$ and $C=C(u)$. Then
		\begin{align*}
			(\f)'(t)=P_\beta(t\star u)=e^{2t}g(t),\qquad g(t):=A-\beta\delta_qB\,e^{(q\delta_q-2)t}-C\,e^{(\t-2)t}.
		\end{align*}
		Since $q\delta_q\ge2$ and $\t>2$, the function $g$ is strictly decreasing, with $g'<0$, and $g(t)\to-\infty$ as $t\to\infty$. On the other hand, as $t\to-\infty$, we have
		\begin{align*}
			g(t)\to A>0\ \text{ if }q>q_c,\qquad g(t)\to A-\beta\delta_qB\ge\Big(1-\frac\beta{\beta^*}\Big)A>0\ \text{ if }q=q_c,
		\end{align*}
		where we used \eqref{c1crit}. Consequently, $g$ has a unique zero $t_u$, with $(\f)'>0$ on $(-\infty,t_u)$ and $(\f)'<0$ on $(t_u,\infty)$. Thus $t_u$ is the unique critical point of $\f$ and a strict maximum point, and $t_u\star u\in\M$ by Remark~\ref{CA1}. The maximum level is positive, since, as $t\to-\infty$,
		\begin{align*}
			e^{-2t}\f(t)\to\frac A2>0\ \text{ if }q>q_c,\qquad e^{-2t}\f(t)\to\frac12(A-\beta\delta_qB)>0\ \text{ if }q=q_c.
		\end{align*}
		
		Next, let $u\in\M$. Then $g(0)=0$, and therefore
		\begin{align*}
			(\Phi^u_\beta)''(0)=2g(0)+g'(0)=g'(0)<0.
		\end{align*}
		Hence $\mathcal M^0_{\rho,\beta}=\emptyset$ and $\M=\mathcal M^-_{\rho,\beta}$. In particular, $\M$ is a $C^1$ manifold by Lemma~\ref{Lman}.
		
		\textit{(i)} For $t>t_u$, we have $g(t)<0$ and $g'(t)<0$, so that $(\f)'(t)<0$ and
		\begin{align*}
			(\f)''(t)=e^{2t}\big(2g(t)+g'(t)\big)<0.
		\end{align*}
		
		\noi\textit{(ii)} This follows from $P_\beta(u)=(\f)'(0)=g(0)$ and the monotonicity of $g$.
		
		\noi\textit{(iii)} The map $(t,u)\mapsto(\f)'(t)$ is of class $C^1$ and $(\f)''(t_u)=e^{2t_u}g'(t_u)<0$, so \textit{(iii)} follows from the implicit function theorem.\QED
	\end{proof}
	
	\begin{lemma}\label{C33}
		Let $q_c\le q<\t$, and $\beta<\beta^*$ if $q=q_c$. Then $m_\beta(\rho)>0$, and there is $k>0$ such that
		\begin{align*}
			0<\sup_{\overline{\mathcal G}_k}\e<m_\beta(\rho),\qquad\text{and}\qquad\e(u)>0,\ P_\beta(u)>0\quad\text{for }u\in\overline{\mathcal G}_k,
		\end{align*}
		where $\mathcal G_k=\{u\in\S:\||x|^{-a}\na u\|_2<k\}$, as in Theorem~\ref{C1}.
	\end{lemma}
	\begin{proof}
		We first prove that $m_\beta(\rho)>0$. By \eqref{c1} and the definition of $\sr$, for $u\in\S$ we have
		\begin{align*}
			\beta\delta_qB(u)+C(u)\le h(A(u)),\qquad h(s):=\beta\delta_q\C\rho^{(1-\delta_q)q}s^{\frac{q\delta_q}2}+\sr^{-\frac\t2}s^{\frac\t2}.
		\end{align*}
		If $q>q_c$, then $q\delta_q>2$, while if $q=q_c$, then the first term of $h$ is $\frac\beta{\beta^*}s$, by the definition of $\beta^*$. In both cases, there are therefore $\theta\in(0,1)$ and $s_0>0$ such that
		\begin{align*}
			h(s)\le\theta s\qquad\text{for }0<s\le s_0.
		\end{align*}
		Now let $u\in\M$. Then $A(u)=\beta\delta_qB(u)+C(u)\le h(A(u))$, so that $A(u)>s_0$, and we use this bound to estimate $\e(u)$ from below.
		
		If $q>q_c$, then $2<q\delta_q<\t$, because $\delta_q<1$, and therefore
		\begin{align*}
			\e(u)=\e(u)-\frac1{q\delta_q}P_\beta(u)=\Big(\frac12-\frac1{q\delta_q}\Big)A(u)+\Big(\frac1{q\delta_q}-\frac1\t\Big)C(u)\ge\Big(\frac12-\frac1{q\delta_q}\Big)s_0.
		\end{align*}
		If $q=q_c$, then \eqref{c1crit} implies $C(u)=A(u)-\beta\delta_qB(u)\ge(1-\frac\beta{\beta^*})A(u)$, and hence
		\begin{align*}
			\e(u)=\e(u)-\frac12P_\beta(u)=\frac dNC(u)\ge\frac dN\Big(1-\frac\beta{\beta^*}\Big)s_0.
		\end{align*}
		In both cases, the lower bound is independent of $u$, so $m_\beta(\rho)>0$.
		
		It remains to choose $k$. Let $u\in\overline{\mathcal G}_k$ with $k^2\le s_0$. Then
		\begin{align*}
			P_\beta(u)\ge A(u)-h(A(u))\ge(1-\theta)A(u)>0.
		\end{align*}
		Since $q\delta_q\ge2$ and $\t>2$, we also have
		\begin{align*}
			\e(u)\ge\frac12A(u)-\frac12\big(\beta\delta_qB(u)+C(u)\big)\ge\frac{1-\theta}2A(u)>0.
		\end{align*}
		Finally, $\e(u)\le\frac12A(u)\le\frac{k^2}2$, which is smaller than $m_\beta(\rho)$ if $k$ is small.\QED
	\end{proof}
	
	\begin{lemma}\label{Lmm}
		Let $q_c\le q<\t$ and $0<\beta<\hat\beta$, with $\hat\beta$ and $\beta_0$ as in \eqref{beta0def}. Then:
		\begin{itemize}
			\item[(i)] $m_\beta(\rho)=\inf_{u\in\S}\max_{t\in\R}\e(t\star u)$, and $\beta\mapsto m_\beta(\rho)$ is nonincreasing on $(0,\hat\beta)$;
			\item[(ii)] $m_\beta(\rho)\le\mathcal T$;
			\item[(iii)] $m_\beta(\rho)<\mathcal T$ for every $\beta\in(\beta_0,\hat\beta)$.
		\end{itemize}
	\end{lemma}
	\begin{proof}
		\textit{(i)} By Lemma~\ref{C10}, for every $u\in\S$, the point $t_u$ is the strict global maximum point of $\f$, $t_u\star u\in\M$, and $t_u=0$ if $u\in\M$. Therefore
		\begin{align*}
			m_\beta(\rho)=\inf_{u\in\S}\e(t_u\star u)=\inf_{u\in\S}\max_{t\in\R}\e(t\star u).
		\end{align*}
		In addition, for fixed $u\in\S$ and $t\in\R$, the map $\beta\mapsto\e(t\star u)$ is nonincreasing, and, by the formula above, so is $\beta\mapsto m_\beta(\rho)$.
		
		\noi\textit{(ii)} We test the formula of \textit{(i)} with normalized cut-off extremals. Let $u_\var$ be the cut-off extremal \eqref{c221}, and let $v_\var:=\rho\,u_\var/\||x|^{-a}u_\var\|_2\in\S$. Dropping the subcritical term and using $\frac12-\frac1\t=\frac dN$ and $\frac{\t}{\t-2}=\frac N{2d}$, we obtain
		\begin{align*}
			\max_{t\in\R}\e(t\star v_\var)\le\max_{t\in\R}\Big(\frac{e^{2t}}2A(v_\var)-\frac{e^{\t t}}\t C(v_\var)\Big)=\frac dN\bigg(\frac{A(u_\var)}{C(u_\var)^{2/\t}}\bigg)^{\frac N{2d}}=\mathcal T+O\big(\var^{\frac{N-2d}{2d}}\big),
		\end{align*}
		by Proposition~\ref{C11}\textit{(i)},~\textit{(ii)}, where we also used that the quotient $A/C^{2/\t}$ is invariant under $u\mapsto cu$. By \textit{(i)}, it follows that $m_\beta(\rho)\le\mathcal T+O(\var^{\frac{N-2d}{2d}})$, and we obtain \textit{(ii)} by letting $\var\to0$.
		
		\noi\textit{(iii)} Let $\beta\in(\beta_0,\hat\beta)$. By the definition of $\beta_0$, there is $\beta'\le\beta$ with $m_{\beta'}(\rho)<\mathcal T$, and the monotonicity in \textit{(i)} gives $m_\beta(\rho)\le m_{\beta'}(\rho)<\mathcal T$.\QED
	\end{proof}
	
	The next lemma provides the Palais--Smale sequence used in the proofs of Theorems~\ref{Tcrit}\textit{(i)} and~\ref{Tsup}. Its proof follows that of Theorem~1.1 in \cite{soave2020normalized} and uses the following properties of the dilations, which follow from the definition of $t\star u$ in Section~\ref{Smain}: for $t\in\R$ and $u,v\in\mathcal X$,
	\begin{equation}\label{dil}
		\begin{gathered}
			A(t\star u)=e^{2t}A(u),\qquad B(t\star u)=e^{q\delta_qt}B(u),\qquad C(t\star u)=e^{\t t}C(u),\\
			\int_{\R^N}\frac{(t\star u)(t\star v)}{|x|^{2a}}\dx=\int_{\R^N}\frac{uv}{|x|^{2a}}\dx.
		\end{gathered}
	\end{equation}
	In addition, the map $(t,u)\mapsto t\star u$ is continuous from $\R\times\mathcal X$ to $\mathcal X$. The proof of \cite[Lemma 3.5]{bartsch2019multiple} applies without change.
	
	\begin{lemma}\label{LPS}
		Let $q_c\le q<\t$ and $\beta>0$, with $\beta<\beta^*$ if $q=q_c$. Let $\mathcal X_+:=\{v\in\mathcal X:\ v\ge0\}$. Then there is a Palais--Smale sequence $\{u_n\}\subset\S$ for $\e|_{\S}$ at level $m_\beta(\rho)$ such that
		\begin{align*}
			P_\beta(u_n)\to0\qquad\text{and}\qquad\operatorname{dist}_{\mathcal X}(u_n,\mathcal X_+)\to0.
		\end{align*}
		In particular, every weak limit in $\mathcal X$ of a subsequence of $\{u_n\}$ is nonnegative.
	\end{lemma}
	\begin{proof}
		We follow \cite[Section 5]{soave2020normalized} and give details only where the weights enter. Let $k>0$ be as in Lemma~\ref{C33}, and define $\tilde\e:\R\times\mathcal X\to\R$ by
		\begin{equation}\label{c40}
			\tilde\e(t,u):=\e(t\star u)=\frac{e^{2t}}{2}A(u)-\frac{\beta e^{q\delta_qt}}{q}B(u)-\frac{e^{\t t}}{\t}C(u).
		\end{equation}
		Then $\tilde\e$ is of class $C^1$, and, since $u\mapsto t\star u$ is linear and bounded on $\mathcal X$, \eqref{dil} gives
		\begin{align}\label{c40d}
			\partial_t\tilde\e(t,u)=P_\beta(t\star u),\qquad\partial_u\tilde\e(t,u)[\zeta]=d\e(t\star u)[t\star\zeta]\qquad(\zeta\in\mathcal X).
		\end{align}
		Let $\e^0:=\{u\in\S:\e(u)\le0\}$, let $\Gamma$ be the set of paths $g=(g_1,g_2)\in C([0,1],\R\times\S)$ with $g(0)\in\{0\}\times\overline{\mathcal G}_k$ and $g(1)\in\{0\}\times\e^0$, and set
		\begin{align}\label{cGamma}
			c_\Gamma:=\inf_{g\in\Gamma}\max_{g([0,1])}\tilde\e.
		\end{align}
		\noindent\textbf{Step 1.} Let $g\in\Gamma$. We claim that
		\begin{align*}
			P_\beta\big(g_1(0)\star g_2(0)\big)>0>P_\beta\big(g_1(1)\star g_2(1)\big).
		\end{align*}
		Here $g_1(0)=0$ and $g_2(0)\in\overline{\mathcal G}_k$, so the first inequality follows from Lemma~\ref{C33}. For the second one, let $u:=g_2(1)$. By Lemma~\ref{C10}, $\f>0$ on $(-\infty,t_u]$, while $\f(0)=\e(u)\le0$. Hence $t_u<0$, and $P_\beta(u)<0$ by Lemma~\ref{C10}\textit{(ii)}. By the intermediate value theorem, $g_1(s_g)\star g_2(s_g)\in\M$ for some $s_g\in(0,1)$, and therefore
		\begin{align*}
			\max_{g([0,1])}\tilde\e\ge\e\big(g_1(s_g)\star g_2(s_g)\big)\ge m_\beta(\rho).
		\end{align*}
		This gives $c_\Gamma\ge m_\beta(\rho)$.
		
		Conversely, let $u\in\M$, so that $t_u=0$ is the maximum point of $\f$. Since $A(s\star u)=e^{2s}A(u)\to0$ as $s\to-\infty$ and $\f(s)\to-\infty$ as $s\to\infty$, there are $s_0<0<s_1$ with $s_0\star u\in\mathcal G_k$ and $\e(s_1\star u)\le0$. The path $g_u(s):=\big(0,((1-s)s_0+ss_1)\star u\big)$ then lies in $\Gamma$, and
		\begin{align*}
			c_\Gamma\le\max_{s\in[0,1]}\tilde\e(g_u(s))=\max_{t\in[s_0,s_1]}\f(t)=\f(0)=\e(u).
		\end{align*}
		Taking the infimum over $u\in\M$, we obtain $c_\Gamma\le m_\beta(\rho)$. Thus $c_\Gamma=m_\beta(\rho)$, and, by Lemma~\ref{C33} and the definition of $\e^0$,
		\begin{align*}
			\sup_{\overline{\mathcal G}_k\cup\e^0}\e<m_\beta(\rho)=c_\Gamma.
		\end{align*}
		\noindent\textbf{Step 2.} By Step~1, exactly as in \cite[Section 5]{soave2020normalized}, $\{g([0,1]):g\in\Gamma\}$ is a homotopy-stable family of compact subsets of $Z:=\R\times\S$ with extended boundary $\{0\}\times(\overline{\mathcal G}_k\cup\e^0)$, and $\{\tilde\e\ge c_\Gamma\}$ satisfies the assumptions of Lemma~\ref{a13}. Since $|t\star u|=t\star|u|$, $\e(|u|)=\e(u)$ and $u\mapsto|u|$ is continuous on $\mathcal X$, replacing $g$ by $(0,g_1\star|g_2|)$ does not change $\max_{g([0,1])}\tilde\e$. Hence there is a minimizing sequence $g_n=(0,g_{2,n})$ with $g_{2,n}\ge0$, and Lemma~\ref{a13}, applied with $\mathcal A_n:=g_n([0,1])\subset\{0\}\times\mathcal X_+$, yields $(t_n,w_n)\in Z$ with
		\begin{align*}
			\tilde\e(t_n,w_n)\to c_\Gamma,\qquad\big\|d(\tilde\e|_Z)(t_n,w_n)\big\|\to0,\qquad |t_n|+\operatorname{dist}_{\mathcal X}(w_n,\mathcal X_+)\to0.
		\end{align*}
		\noindent\textbf{Step 3.} Let $u_n:=t_n\star w_n\in\S$. Then $\e(u_n)\to m_\beta(\rho)$, and testing $d(\tilde\e|_Z)(t_n,w_n)$ with $(1,0)$ gives $P_\beta(u_n)\to0$ by \eqref{c40d}. Let $\psi\in T_{u_n}\S$. By the last identity in \eqref{dil}, $\zeta:=(-t_n)\star\psi\in T_{w_n}\S$ (the weighted form of \cite[Lemma 3.6]{bartsch2019multiple}), and $\|\zeta\|_{\mathcal X}\le e^{|t_n|}\|\psi\|_{\mathcal X}$. Testing with $(0,\zeta)$, we obtain from \eqref{c40d}
		\begin{align*}
			|d\e(u_n)[\psi]|\le\big\|d(\tilde\e|_Z)(t_n,w_n)\big\|\,e^{|t_n|}\|\psi\|_{\mathcal X},
		\end{align*}
		so that $\{u_n\}$ is a Palais--Smale sequence for $\e|_{\S}$ at level $m_\beta(\rho)$. Finally, $v\mapsto t\star v$ is linear, maps $\mathcal X_+$ into itself and has norm at most $e^{|t|}$, so that $\operatorname{dist}_{\mathcal X}(u_n,\mathcal X_+)\le e^{|t_n|}\operatorname{dist}_{\mathcal X}(w_n,\mathcal X_+)\to0$. If $u_{n_j}\rightharpoonup u$ and $v_j\in\mathcal X_+$ with $\|u_{n_j}-v_j\|_{\mathcal X}\to0$, then $v_j\rightharpoonup u$, and $u\in\mathcal X_+$ because $\mathcal X_+$ is convex and closed, hence weakly closed.\QED
	\end{proof}
	\section{Existence of solutions}\label{Sproofs}	
	This section contains the proofs of the existence results. Each proof is preceded by the energy estimate it requires. In addition to $A$, $B$, $C$ from Section~\ref{Sfib}, we write $M(u):=\||x|^{-a}u\|_2^2$.
	\subsection{The case \texorpdfstring{$2<q<q_c$}{2<q<qc}}\label{ss:pfsub}
	\noindent\textbf{Proof of Theorem \ref{C1}\textit{(i)}.} We follow the proof of Theorem~1.1 for $2<q<2+\frac4N$ in \cite{soave2020normalized}. We first construct a nonnegative minimizing sequence in $\mathcal{M}_{\rho,\beta}^+$. Let $\{v_n\}$ be a minimizing sequence for $\inf_{\mathcal{G}_{R_0}}\e$. Since $\e(|u|)=\e(u)$ and $|u|\in\mathcal{G}_{R_0}$ for $u\in\mathcal{G}_{R_0}$, we may take $v_n\ge0$. By Lemma~\ref{C12}, for each $n$ there exists $t_n\in\R$ such that $t_n\star v_n\in\mathcal{M}_{\rho,\beta}^+$, $\||x|^{-a}\na(t_n\star v_n)\|_2\le R_0$, and
	\begin{align*}
		\e(t_n\star v_n)=\min\big\{\e(t\star v_n):\ t\in \R,\ \||x|^{-a}\na(t\star v_n)\|_2\le R_0\big\}\le\e(v_n).
	\end{align*}
	Thus $w_n:=t_n\star v_n\in\mathcal{M}_{\rho,\beta}^+$, $w_n\ge0$, and $\{w_n\}$ is again a minimizing sequence. In addition, $\||x|^{-a}\na w_n\|_2\le R_0-r$ for $n$ large by Lemma~\ref{C31}.
	
	Next, we extract a Palais--Smale sequence from $\{w_n\}$ by applying Ekeland's variational principle \cite{ekeland1974variational} on the complete metric space $\overline{\mathcal G}_{R_0}$ with the distance of $\mathcal X$. With $\varepsilon_n:=\e(w_n)-m_\beta(\rho)\to0$, the principle provides $u_n\in\overline{\mathcal G}_{R_0}$ such that
	\begin{align*}
		\e(u_n)\le\e(w_n),\qquad\|u_n-w_n\|_{\mathcal X}\le\sqrt{\varepsilon_n},\qquad\e(v)\ge\e(u_n)-\sqrt{\varepsilon_n}\|v-u_n\|_{\mathcal X}\ \text{ for }v\in\overline{\mathcal G}_{R_0}.
	\end{align*}
	As a consequence, $\||x|^{-a}\na u_n\|_2<R_0-\frac r2$ for $n$ large. To estimate the constrained derivative at $u_n$, let $h\in\mathcal X$ with $\int_{\R^N}|x|^{-2a}u_nh\dx=0$ and $\|h\|_{\mathcal X}=1$. The curve
	\begin{align*}
		c(\tau):=\frac{\rho\,(u_n+\tau h)}{\||x|^{-a}(u_n+\tau h)\|_2}
	\end{align*}
	lies in $\mathcal G_{R_0}$ for $|\tau|$ small, and $c(\tau)=u_n+\tau h+O(\tau^2)$. Applying the last inequality of Ekeland's principle along this curve, we obtain $|d\e(u_n)h|\le\sqrt{\varepsilon_n}$, that is, $\|d\e|_{\S}(u_n)\|\le\sqrt{\varepsilon_n}\to0$. Thus $\{u_n\}$ is a Palais--Smale sequence for $\e|_{\S}$ at level $m_\beta(\rho)$.
	
	We now show that the Pohozaev functional tends to zero along $\{u_n\}$. By \eqref{c1} and the definition of $\sr$, the three terms of $P_\beta$ are Lipschitz on bounded subsets of $\mathcal X$ (by H\"older's inequality and the elementary bound $||s|^p-|r|^p|\le p(|s|^{p-1}+|r|^{p-1})|s-r|$), hence uniformly continuous on bounded subsets of $\mathcal X$. The sequences $\{u_n\}$ and $\{w_n\}$ are bounded and satisfy $\|u_n-w_n\|_{\mathcal X}\to0$. As $P_\beta(w_n)=0$, it follows that
	\begin{align*}
		P_{\beta}(u_n)=P_{\beta}(w_n)+o(1)=o(1)\qquad\text{as } n\to \infty.
	\end{align*}
	The sequence $\{u_n\}$ need not be nonnegative. However, $\|u_n-w_n\|_{\mathcal X}\to0$. Hence, the sequences $\{u_n\}$ and $\{w_n\}$ have the same weak limit, so that the limit $\tilde u$ obtained below is nonnegative.
	
	We now apply Proposition~\ref{C32}, so that either (I) or (II) holds. Alternative (I) is excluded as in \cite[proof of Theorem 1.1 for $2<q<2+\frac4N$]{soave2020normalized}: if $u_n\rightharpoonup\tilde u$ weakly but not strongly, then $P_\beta(\tilde u)=0$ by Lemma~\ref{B2poh}, and \eqref{c1} together with $\beta<\beta_2$ yields $\e(\tilde u)>-\frac dN\sr^{N/2d}$.
	
	To check this, note that $\||x|^{-a}\tilde u\|_2\le\rho$ by weak lower semicontinuity, and write $A=A(\tilde u)$, $B=B(\tilde u)$ and $\varpi=q\delta_q$. From $P_\beta(\tilde u)=0$, $\frac12-\frac1\t=\frac dN$ and \eqref{c1} we obtain
	\begin{align*}
		\e(\tilde u)=\e(\tilde u)-\frac1\t P_\beta(\tilde u)=\frac dNA-\beta\,\frac{\t-\varpi}{q\t}\,B\ge\frac dNA-\beta\,\frac{\t-\varpi}{q\t}\,\C\rho^{(1-\delta_q)q}A^{\frac\varpi2}.
	\end{align*}
	Since $\varpi<2$, the minimum over $A>0$ of the right-hand side is
	\begin{align*}
		-\frac{2-\varpi}{2}\Big(\frac{N\varpi}{2d}\Big)^{\frac{\varpi}{2-\varpi}}\Big(\beta\,\frac{\t-\varpi}{q\t}\,\C\rho^{(1-\delta_q)q}\Big)^{\frac{2}{2-\varpi}},
	\end{align*}
	and a direct computation shows that it is larger than $-\frac dN\sr^{N/2d}$ if and only if $\beta<\beta_2$.
	
	\noindent This contradicts the bound provided by alternative (I), since $m_\beta(\rho)<0$:
	\begin{align*}
		\e(\tilde u)\le m_\beta(\rho)-\frac dN\sr^{N/2d}<-\frac dN\sr^{N/2d}.
	\end{align*}
	Hence (II) holds: $u_n\to\tilde{u}$ strongly in $\mathcal{X}$, $\e(\tilde{u})=m_{\beta}(\rho)$, and $\tilde{u}$ solves \eqref{eq:P} for some $\tilde{\lambda}<0$. By Lemma~\ref{B2poh}, $\tilde u\in\M$, and Lemma~\ref{C12}\textit{(v)} together with $\e(\tilde u)=m_\beta(\rho)<0$ implies $\tilde u\in\mathcal M^+_{\rho,\beta}\subset\mathcal G_{R_0}$. Thus
	\begin{align*}
		\e(\tilde{u})=m_{\beta}(\rho)=\inf_{\mathcal{G}_{R_0}}\e,\qquad\||x|^{-a}\na \tilde u\|_2<R_0.
	\end{align*}
	Finally, $\tilde u>0$ in $\R^N$ by $\tilde u\ge0$ and Lemma~\ref{B3pos}.\QED
	
	\noi The second solution is obtained as in \cite{jeanjean2022multiple} and \cite[Section 3]{wei2022normalized}. The proof relies on the level estimate of Lemma~\ref{C72}, the analogue of \cite[Proposition 1.16]{jeanjean2022multiple} and \cite[Lemma 3.1 and Remark 3.1]{wei2022normalized}. The rest of the proof of Theorem~\ref{C1}\textit{(ii)} follows \cite[Propositions 1.10 and 1.11]{jeanjean2022multiple}. Throughout, $c_\beta(\rho)$ is as in Lemma~\ref{C12}\textit{(v)}, and $\tilde u$ is the ground state of Theorem~\ref{C1}\textit{(i)}.
	
	\begin{lemma}\label{C72}
		$c_\beta(\rho)<m_\beta(\rho)+\frac{d}{N}\sr^{N/2d}$.
	\end{lemma}
	\begin{proof} We bound $c_\beta(\rho)$ from above using test functions built from $\tilde u$ and $u_\var$. For $t\ge0$, let $\hat{u}_{\var,t}:=\tilde{u}+tu_{\varepsilon}$, with $u_\var$ as in \eqref{c221}, and set
		\begin{align*}
			\tilde{v}_{\var,t}(x):=\tau^{\frac{N-2a-2}{2}}\hat{u}_{\var,t}(\tau x),\qquad\tau:=\frac{\||x|^{-a}\hat{u}_{\var,t}\|_2}{\rho}.
		\end{align*}
		Then $\tilde{v}_{\var,t}\in\S$, and
		\begin{equation}\label{c414}
			A(\tilde{v}_{\var,t})=A(\hat{u}_{\var,t}),\qquad C(\tilde{v}_{\var,t})=C(\hat{u}_{\var,t}),\qquad B(\tilde{v}_{\var,t})=\tau^{(\delta_q-1)q}B(\hat{u}_{\var,t}).
		\end{equation}
		\noindent\textbf{Step 1.} By Lemma~\ref{C12}, for every $t\ge0$ there is a unique $\tau_{\var,t}>0$, depending continuously on $t$, such that $\tau_{\var,t}^{\frac{N-2a}{2}} \tilde{v}_{\var,t}(\tau_{\var,t}x)\in\mathcal{M}_{\rho,\beta}^-$, that is,
		\begin{equation}\label{c77}
			\tau_{\var,t}^{2-q\delta_q}A(\tilde{v}_{\var,t})=\beta \delta_qB(\tilde{v}_{\var,t})+\tau_{\var,t}^{\t-q\delta_q}C(\tilde{v}_{\var,t}).
		\end{equation}
		Since $\tilde{u}\in \mathcal{M}_{\rho,\beta}^+$, we have $\tau_{\var,0}>1$. On the other hand, as $\tilde u,u_\var\ge0$, \eqref{c77}, \eqref{c414} and Proposition~\ref{C11}\textit{(i)},~\textit{(ii)} give, for $\var$ small, $$\tau_{\var,t}^{\t-2}\le \frac{A(\hat u_{\var,t})}{C(\hat u_{\var,t})}\le\frac{2A(\tilde u)+2t^2A(u_\var)}{t^\t C(u_\var)}\le K(1+t^2)t^{-\t}.$$ Thus $\tau_{\var,t}\to0$ as $t\to\infty$, and the intermediate value theorem gives $\tau_{\var,t}=1$ for some $t>0$, that is, $\tilde v_{\var,t}\in\mathcal M^-_{\rho,\beta}$. Therefore
		\begin{equation}\label{c99}
			c_{\beta}(\rho)\le\sup_{t\ge0}\e(\tilde{v}_{\var,t}).
		\end{equation}
		\noindent\textbf{Step 2.} Let $c_*:=\frac{d}{2N}\sr^{N/2d}$. We show that there exists $t_0>1$ such that, for $\var$ small and for $t<t_0^{-1}$ or $t>t_0$,
		\begin{align}\label{c98}
			\e(\tilde{v}_{\var,t})=\frac{1}{2}A(\hat{u}_{\var,t})-\frac{\beta}{q}\tau^{(\delta_q-1)q}B(\hat{u}_{\var,t})-\frac{1}{\t}C(\hat{u}_{\var,t})<m_\beta(\rho)+\frac{d}{N}\sr^{N/2d}-c_*.
		\end{align}
		Since $M(\tilde u)=\rho^2$ and $\tilde u,u_\var\ge0$, we have 
		\begin{align*}
			\tau^2=1+\frac{2t}{\rho^2}\int_{\R^N}|x|^{-2a}\tilde u\,u_\var\dx+\frac{t^2}{\rho^2}M(u_\var)\ge1. 
		\end{align*}
		This gives  $0<\tau^{(\delta_q-1)q}\le1$. For large $t$ we drop the $B$-term and use $\hat u_{\var,t}\ge tu_\var\ge0$. Then \eqref{c414} and Proposition~\ref{C11}\textit{(i)},~\textit{(ii)} yield $$\e(\tilde v_{\var,t})\le\frac12A(\hat u_{\var,t})-\frac1\t C(\hat u_{\var,t})\le K(1+t^2)-\frac{t^\t}{2\t}\sr^{N/2d}. $$ This  is smaller than the right-hand side of \eqref{c98} for $t>t_0$ if $t_0$ is large. For $0\le t\le1$, Proposition~\ref{C11}\textit{(i)}--\textit{(iv)} and H\"older's inequality show that $A(\hat u_{\var,t})$, $B(\hat u_{\var,t})$, $C(\hat u_{\var,t})$ and $\tau$ are Lipschitz functions of $t$, with Lipschitz constants bounded uniformly in small $\var$. Hence $|\e(\tilde v_{\var,t})-\e(\tilde u)|\le Ct$, and \eqref{c98} holds for $t<t_0^{-1}$ if $t_0$ is large, because $\e(\tilde u)=m_\beta(\rho)$.
		
		\noindent\textbf{Step 3.} It remains to treat $t_0^{-1}\le t\le t_0$. By Lemma~\ref{B3pos}, $\tilde u\in C(\R^N)$ and $\tilde u(0)>0$; in particular, $\tilde u$ is bounded on $B_{2R}$. The relation $\alpha\frac{N-2d}{2d}=N-2-2a$ gives
		\begin{align*}
			U_\var(x)\le C\var^{\frac{N-2d}{4d}}|x|^{-(N-2-2a)}.
		\end{align*}
		Since $bq<b\t<2+2a$, we obtain
		\begin{align*}
			\int_{\R^N}|x|^{-2a}\tilde u\,u_\var\dx+\int_{\R^N}|x|^{-bq}\tilde u^{q-1}u_\var\dx&\le C\var^{\frac{N-2d}{4d}}\int_{B_{2R}}\big(|x|^{2-N}+|x|^{-bq-(N-2-2a)}\big)\dx\\
			&=O(\var^{\frac{N-2d}{4d}}).
		\end{align*}
		As in Step~2, $\tau^2=1+\mathcal{O}_{\var,t}$, where
		\begin{align*}
			\mathcal{O}_{\var,t}:=\frac{2t}{\rho^2}\int_{\R^N}|x|^{-2a}\tilde u\,u_\var\dx+\frac{t^2}{\rho^2}M(u_\var)\ge0,
		\end{align*}
		and, by the last estimate and Proposition~\ref{C11}\textit{(iv)}, $\mathcal{O}_{\var,t}\to0$ as $\var\to0$, uniformly for $t_0^{-1}\le t\le t_0$.
		
		To expand $\e(\tilde v_{\var,t})$, we use an elementary inequality: for $r>2$, there exists $c_r>0$ such that, for all $X,Y\ge0$,
		\begin{align*}
			(X+Y)^r\ge \begin{cases} X^r+Y^r+rX^{r-1}Y+c_{r}XY^{r-1}, & Y\ge X,\\ X^r+Y^r+rX^{r-1}Y, & Y<X.\end{cases}
		\end{align*}
		Define
		\begin{align*}
			\Omega_{\var,t}:=\{x\in\R^N:\ tu_\var(x)\ge\tilde u(x)\}.
		\end{align*}
		Recall that $\kappa=N-q(b+\gamma_a)>0$ by \eqref{b0gk} and that $\kappa=(1-\delta_q)q$ by the definition of $\delta_q$, so that $\tau^{q(\delta_q-1)}=(1+\mathcal{O}_{\var,t})^{-\kappa/2}$. Testing the equation of $\tilde u$ with $u_\var$, we also obtain
		\begin{align*}
			\int_{\R^N}|x|^{-2a}\na\tilde u\cdot\na u_\var\dx=\tilde\lambda\int_{\R^N}|x|^{-2a}\tilde u\,u_\var\dx+\beta\int_{\R^N}|x|^{-bq}\tilde u^{q-1}u_\var\dx+\int_{\R^N}|x|^{-b\t}\tilde u^{\t-1}u_\var\dx.
		\end{align*}
		We insert this identity into the expression of $\e(\tilde v_{\var,t})$ given by \eqref{c414}. To bound $B(\hat u_{\var,t})$ and $C(\hat u_{\var,t})$ from below, we apply this inequality with $X=\tilde u$ and $Y=tu_\var$, for $r=q$, where we drop the last term, and for $r=\t$, where the last term vanishes outside $\Omega_{\var,t}$. For the factor $\tau^{q(\delta_q-1)}$ we use
		\begin{align*}
			(1+X)^{-\kappa/2}\ge1-\frac\kappa2X,\qquad 1-(1+X)^{-\kappa/2}\le\frac\kappa2X\qquad(X\ge0).
		\end{align*}
		The terms containing $\int_{\R^N}|x|^{-b\t}\tilde u^{\t-1}u_\var\dx$ then cancel, and we arrive at 
		\begin{equation}\label{c415}
			\begin{split}
				\e(\tilde v_{\var,t})\le{}&\e(\tilde u)+\Big(\frac{t^2}{2}A(u_\var)-\frac{t^{\t}}{\t}C(u_\var)\Big)+t\tilde\lambda\int_{\R^N}|x|^{-2a}\tilde u\,u_\var\dx+\frac{\beta\kappa}{2q}\mathcal{O}_{\var,t}B(\tilde u)\\
				&+\frac{\beta\kappa}{2}\,t\,\mathcal{O}_{\var,t}\int_{\R^N}|x|^{-bq}\tilde u^{q-1}u_\var\dx-\frac{\beta}{q}(1+\mathcal{O}_{\var,t})^{-\kappa/2}t^qB(u_\var)\\
				&-\frac{c_{\t}}{\t}t^{\t-1}\int_{\Omega_{\var,t}}|x|^{-b\t}\tilde u\,u_\var^{\t-1}\dx.
			\end{split}
		\end{equation}
		We now estimate the right-hand side of \eqref{c415} term by term. By Lemma~\ref{B2poh}, $\tilde\lambda\rho^2=\beta(\delta_q-1)B(\tilde u)$, so that $\frac{\beta\kappa}{2q}B(\tilde u)=\frac{|\tilde\lambda|\rho^2}{2}$ and
		\begin{align*}
			t\tilde\lambda\int_{\R^N}|x|^{-2a}\tilde u\,u_\var\dx+\frac{\beta\kappa}{2q}\mathcal{O}_{\var,t}B(\tilde u)=\frac{|\tilde\lambda|}{2}t^2M(u_\var),
		\end{align*}
		that is, the terms linear in $t$ cancel. By Proposition~\ref{C11}\textit{(i)},~\textit{(ii)},
		\begin{align*}
			\frac{t^2}{2}A(u_\var)-\frac{t^{\t}}{\t}C(u_\var)\le\frac{d}{N}\sr^{N/2d}+O(\var^{\frac{N-2d}{2d}}),
		\end{align*}
		and, by the estimates at the beginning of this step,
		\begin{align*}
			\frac{\beta\kappa}{2}\,t\,\mathcal{O}_{\var,t}\int_{\R^N}|x|^{-bq}\tilde u^{q-1}u_\var\dx=O(\var^{\frac{N-2d}{2d}})+O(\var^{\frac{N-2d}{4d}}M(u_\var)).
		\end{align*}
		It remains to bound the integral in the last term of \eqref{c415} from below. Proposition~\ref{C11}\textit{(v)} gives only an upper bound, but the same scaling also gives a lower bound. Let $\var$ be so small that $\var^{1/\alpha}<R$. Then $u_\var=U_\var$ on $B_{\var^{1/\alpha}}$, and, because $U_1$ is radially decreasing, \eqref{Ueps} gives
		\begin{align*}
			u_\var(x)=\var^{-\frac{N-2d}{4d}}\,U_1(\var^{-1/\alpha}x)\ge\var^{-\frac{N-2d}{4d}}\,U_1(e_1)\qquad\text{for }|x|\le\var^{1/\alpha},
		\end{align*}
		where $e_1$ is a unit vector. Recall that $\tilde u$ is continuous and $\tilde u(0)>0$. Hence, for $\var$ small and uniformly in $t\ge t_0^{-1}$, $B_{\var^{1/\alpha}}\subset\Omega_{\var,t}$ and $\tilde u\ge\frac12\tilde u(0)$ on $B_{\var^{1/\alpha}}$. The change of variables $x=\var^{1/\alpha}y$ then yields
		\begin{align*}
			\int_{\Omega_{\var,t}}|x|^{-b\t}\tilde u\,u_\var^{\t-1}\dx\ge\frac{\tilde u(0)}{2}\int_{B_{\var^{1/\alpha}}}|x|^{-b\t}U_\var^{\t-1}\dx=\frac{\tilde u(0)}{2}K_0\,\var^{\frac{N-2d}{4d}},\qquad K_0:=\int_{B_1}|y|^{-b\t}U_1^{\t-1}\dy,
		\end{align*}
		and $0<K_0<\infty$ by the proof of Proposition~\ref{C11}\textit{(v)}.
		
		\noindent\textbf{Step 4.} Inserting the estimates of Step~3 into \eqref{c415} and using $(1+\mathcal{O}_{\var,t})^{-\kappa/2}\ge\frac12$ for $\var$ small, we obtain constants $K_1,c_1,c_2>0$, independent of $\var$ and of $t\in[t_0^{-1},t_0]$, such that
		\begin{equation}\label{c416}
			\e(\tilde v_{\var,t})\le m_\beta(\rho)+\frac{d}{N}\sr^{N/2d}+K_1M(u_\var)-c_1B(u_\var)-c_2\var^{\frac{N-2d}{4d}}+O(\var^{\frac{N-2d}{2d}}).
		\end{equation}
		It remains to show that the sum of the last four terms in \eqref{c416} is negative for $\var$ small. The last term is $o(\var^{\frac{N-2d}{4d}})$, because $\frac{N-2d}{2d}>\frac{N-2d}{4d}$. In each case of Proposition~\ref{C11}\textit{(iii)}, the exponent of $\var$ is at most
		\[
		e_q:=\frac{(2N-q(N-2d))(N-2d)}{4d(N-2-2a)}
		\]
		(see the proof of Proposition~\ref{C11}\textit{(iii)}), so that $B(u_\var)\ge c_3\var^{e_q}$ for some $c_3>0$. We now use Proposition~\ref{C11}\textit{(iv)}. If $a\ge\frac{N-4}{2}$, then $M(u_\var)=O(\var^{\frac{N-2d}{2d}}|\log\var|)=o(\var^{\frac{N-2d}{4d}})$, and the term $-c_2\var^{\frac{N-2d}{4d}}$ dominates. If $a<\frac{N-4}{2}$, then $M(u_\var)\le K_2\var^{\frac{N-2d}{d(N-2-2a)}}$. Since $q>2$ and $d<1$ give $q(N-2d)>2N-4$, we have
		\[
		e_q<\frac{4(N-2d)}{4d(N-2-2a)}=\frac{N-2d}{d(N-2-2a)},
		\]
		and therefore $K_1M(u_\var)<c_1B(u_\var)$ for $\var$ small. In both cases, $\e(\tilde v_{\var,t})<m_\beta(\rho)+\frac{d}{N}\sr^{N/2d}$ for all $t\in[t_0^{-1},t_0]$, provided $\var$ is small. We fix such an $\var$ for which \eqref{c98} also holds. By continuity, the maximum of $t\mapsto\e(\tilde v_{\var,t})$ over $[t_0^{-1},t_0]$ is also below $m_\beta(\rho)+\frac{d}{N}\sr^{N/2d}$, and \eqref{c98} and \eqref{c99} yield
		\[
		c_\beta(\rho)\le\sup_{t\ge0}\e(\tilde v_{\var,t})<m_\beta(\rho)+\frac{d}{N}\sr^{N/2d}.
		\]
	\end{proof}
	\begin{rem}\label{R71}
		The condition $m_{\beta}(\rho)<0$, together with Lemma~\ref{C72}, implies that $c_{\beta}(\rho)<\frac{d}{N}\sr^{N/2d}$.
	\end{rem}
	
	\noindent\textbf{Proof of Theorem \ref{C1} \textit{(ii)}.} We construct a Palais--Smale sequence in $\mathcal{M}^-_{\rho,\beta}$ at the level $c_\beta(\rho)$, and then exclude the loss of compactness with Lemma~\ref{C72}.
	
	\noindent\textbf{Step 1. A Palais--Smale sequence for $J$.} For $u\in\S$, let $t_u:=t_2$ be as in Lemma~\ref{C12}, so that $t_u\star u\in\mathcal{M}^-_{\rho,\beta}$. By Lemma~\ref{C12}\textit{(iv)}, $J(u):=\e(t_u\star u)$ is of class $C^1$ on $\S$, and, since $(\f)'(t_u)=0$,
	\begin{align*}
		dJ(u)[\phi]=d\e(t_u\star u)[t_u\star\phi]\qquad\text{for all }\phi\in T_u\S.
	\end{align*}
	Moreover, $\inf_{\S}J=c_\beta(\rho)$ and $J(|u|)=J(u)$, so there are $v_n\in\mathcal{M}^-_{\rho,\beta}$ with $v_n\ge0$ and $\e(v_n)\to c_\beta(\rho)$. Lemma~\ref{a13}, applied with $Z=\S$, $\varphi=J$, the family of all singletons of $\S$ (with $\mathcal B=\emptyset$), $F=\S$ and $\mathcal A_n=\{v_n\}$, yields $\{w_n\}\subset\S$ with
	\begin{align*}
		J(w_n)\to c_\beta(\rho),\qquad\|dJ|_{\S}(w_n)\|\to0,\qquad\|w_n-v_n\|_{\mathcal X}\to0.
	\end{align*}
	
	\noindent\textbf{Step 2. {\boldmath$t_{w_n}\to0$.}} Since $t_{v_n}=0$, this is a continuity property of $u\mapsto t_u$ along sequences that need not converge. We therefore work in $\R^3$. Let $\varpi:=q\delta_q\in(0,2)$, and, for $t\in\R$ and $y=(y_1,y_2,y_3)\in\R^3$, let
	\begin{align*}
		F(t;y):=e^{2t}y_1-\beta\delta_q\,e^{\varpi t}y_2-e^{\t t}y_3,\qquad Y(u):=\big(A(u),B(u),C(u)\big)\quad(u\in\S),
	\end{align*}
	so that $(\f)'(t)=F(t;Y(u))$ and $(\f)''(t)=\partial_tF(t;Y(u))$. Set $y_n:=Y(v_n)$ and $y_n':=Y(w_n)$.
	
	We first show that $\{y_n\}$ and $\{y_n'\}$ are bounded, with first component bounded away from $0$. For $u\in\M$ we have $$\e(u)=\e(u)-\frac12P_\beta(u)=\frac\beta q\big(\frac\varpi2-1\big)B(u)+\frac dNC(u)\le\frac dNC(u),$$ and $$A(u)=\beta\delta_qB(u)+C(u)\ge C(u).$$ Hence, for $n$ large,
	\begin{align*}
		A(v_n)\ge C(v_n)\ge\frac Nd\,\e(v_n)\ge\frac{N}{2d}c_\beta(\rho)=:c_0>0.
	\end{align*}
	On the other hand, $P_\beta(v_n)=0$, $B\ge0$ and \eqref{c1} give
	\begin{align*}
		\e(v_n)=\e(v_n)-\frac1\t P_\beta(v_n)\ge\frac dNA(v_n)-\frac\beta q\C\rho^{(1-\delta_q)q}A(v_n)^{\frac\varpi2},
	\end{align*}
	so that $\{A(v_n)\}$ is bounded, since $\varpi<2$. Thus $\{v_n\}$ and $\{w_n\}$ are bounded in $\mathcal X$. By \eqref{c1} and the definition of $\sr$, for $u\in\S$,
	\begin{align}\label{c10con}
		B(u)\le\C\rho^{(1-\delta_q)q}A(u)^{\frac\varpi2},\qquad C(u)\le\sr^{-\frac\t2}A(u)^{\frac\t2},
	\end{align}
	so $\{y_n\}$ and $\{y_n'\}$ are bounded in $\R^3$. As $A$, $B$ and $C$ are Lipschitz on bounded subsets of $\mathcal X$, we also have $|y_n-y_n'|\to0$.
	
	Now let $y=(y_1,y_2,y_3)$ be the limit of a subsequence of $\{y_n\}$; then also $y_n'\to y$ along it. Passing to the limit in $F(0;y_n)=0$, $\partial_tF(0;y_n)<0$, $A(v_n)\ge c_0$ and \eqref{c10con} with $u=v_n$, we obtain
	\begin{align*}
		F(0;y)=0,\qquad\partial_tF(0;y)\le0,\qquad y_1\ge c_0,\qquad y_2\le\C\rho^{(1-\delta_q)q}y_1^{\frac\varpi2},\qquad y_3\le\sr^{-\frac\t2}y_1^{\frac\t2}.
	\end{align*}
	If $\partial_tF(0;y)=0$, then $y$ satisfies the two identities and the two inequalities from which the proof of $\mathcal M^0_{\rho,\beta}=\emptyset$ in Lemma~\ref{C12} derives $\beta\ge\beta_M$. This contradicts $\beta<\beta_1<\beta_M$. Hence $\partial_tF(0;y)<0$, and the implicit function theorem gives $\delta>0$, a neighbourhood $U$ of $y$ and a continuous function $\tau:U\to(-\delta,\delta)$ with $\tau(y)=0$, $\partial_tF<0$ on $(-\delta,\delta)\times U$, and
	\begin{align*}
		F(t;z)=0\iff t=\tau(z)\qquad\text{for }(t,z)\in(-\delta,\delta)\times U.
	\end{align*}
	For $n$ large, $\tau_n:=\tau(y_n')$ satisfies $(\Phi^{w_n}_\beta)'(\tau_n)=0>(\Phi^{w_n}_\beta)''(\tau_n)$, that is, $\tau_n\star w_n\in\mathcal M^-_{\rho,\beta}$. By Lemma~\ref{C12}\textit{(i)}, $\tau_n=t_{w_n}$, and so $t_{w_n}=\tau(y_n')\to\tau(y)=0$ along the subsequence. Since every subsequence of $\{y_n\}$ has a convergent subsequence, $t_{w_n}\to0$.
	
	\noindent\textbf{Step 3. A Palais--Smale sequence in {\boldmath${\mathcal{M}^-_{\rho,\beta}}$.}} Set $u_n:=t_{w_n}\star w_n\in\mathcal{M}^-_{\rho,\beta}$. Then $\e(u_n)=J(w_n)\to c_\beta(\rho)$ and $P_\beta(u_n)=0$. If $\psi\in T_{u_n}\S$, then $\phi:=(-t_{w_n})\star\psi\in T_{w_n}\S$ and $\|\phi\|_{\mathcal X}\le e^{|t_{w_n}|}\|\psi\|_{\mathcal X}$, so that
	\begin{align*}
		|d\e(u_n)[\psi]|=|dJ(w_n)[\phi]|\le e^{|t_{w_n}|}\|dJ|_{\S}(w_n)\|\,\|\psi\|_{\mathcal X}.
	\end{align*}
	Thus $\{u_n\}$ is a Palais--Smale sequence for $\e|_{\S}$ at level $c_\beta(\rho)$. As $t_{w_n}\to0$ and $w_n-v_n\to0$, the sequences $\{u_n\}$ and $\{v_n\}$ have the same weak limits, which are nonnegative.
	
	\noindent\textbf{Step 4. Compactness.} By the proof of Theorem~\ref{C1}\textit{(i)}, Lemma~\ref{C12}\textit{(v)}, Lemma~\ref{C72} and Remark~\ref{R71},
	\begin{align*}
		\inf_{\M}\e=\inf_{\mathcal{M}^+_{\rho,\beta}}\e=m_\beta(\rho)<0<c_\beta(\rho)<m_\beta(\rho)+\frac{d}{N}\sr^{N/2d}<\frac{d}{N}\sr^{N/2d}.
	\end{align*}
	Hence Proposition~\ref{C32} applies to $\{u_n\}$. If alternative (I) held, the weak limit $u$ would be a solution in $\S$, hence $u\in\M$ by Lemma~\ref{B2poh}, and
	\begin{align*}
		\e(u)\le c_\beta(\rho)-\frac{d}{N}\sr^{N/2d}<m_\beta(\rho)=\inf_{\M}\e,
	\end{align*}
	a contradiction. Therefore alternative (II) holds: $u_n\to u_0$ in $\mathcal{X}$, $\e(u_0)=c_\beta(\rho)>0$, and $u_0$ solves \eqref{eq:P} for some $\lambda_0<0$. Moreover, $u_0\ge0$, so $u_0>0$ by Lemma~\ref{B3pos}, and $u_0\ne\tilde u$ because $\e(u_0)>0>\e(\tilde u)$.\QED
	
	\subsection{The case \texorpdfstring{$q=q_c$}{q=qc}}\label{ss:pfcrit}
	We follow \cite[Section 5]{soave2020normalized}, with $\beta^*$ in place of the constant in \cite[(5.1)]{soave2020normalized}. The next lemma proves $\beta_0<\beta^*$ and the upper bound in Theorem~\ref{Tcrit}\textit{(ii)}.
	
	\begin{lemma}\label{Lbeta0}
		Let $q=q_c$, and let $C_{a,b,q_c}$ be the best constant in \eqref{c1}.
		\begin{itemize}
			\item[(i)] There exists $w\in\S$, $w\ge0$, with $B(w)=C_{a,b,q_c}^{q_c}\rho^{\frac{4d}{N-2a+2b}}A(w)$.
			\item[(ii)] For such $w$ and $0<\beta<\beta^*$,
			\begin{align*}
				\max_{t\in\R}\e(t\star w)=\frac{d}{N}\Big(1-\frac{\beta}{\beta^*}\Big)^{\frac{N}{2d}}\frac{\||x|^{-a}\na w\|_2^{N/d}}{\||x|^{-b}w\|_{\t}^{\t(N-2d)/2d}}.
			\end{align*}
			\item[(iii)] Let $\omega_w:=\sr^{N/2d}\||x|^{-b}w\|_{\t}^{\t(N-2d)/2d}\||x|^{-a}\na w\|_2^{-N/d}\in(0,1]$. Then $\beta_0\le\beta^*\big(1-\omega_w^{2d/N}\big)<\beta^*$, where $\beta_0$ is defined in \eqref{beta0def}. Moreover, $m_\beta(\rho)\to0$ as $\beta\uparrow\beta^*$.
		\end{itemize}
	\end{lemma}
	\begin{proof}\textit{(i)} We obtain $w$ as a maximizer of the following quotient: for $u\in\mathcal Y\setminus\{0\}$, let
		\begin{align*}
			\mathcal W(u):=\frac{B(u)}{A(u)M(u)^{(q_c-2)/2}}.
		\end{align*}
		Here $q_c\delta_{q_c}=2$ and $q_c-2=\frac{4d}{N-2a+2b}$, so \eqref{c1} reads $\mathcal W\le C_{a,b,q_c}^{q_c}$ on $\mathcal Y\setminus\{0\}$, and $C_{a,b,q_c}^{q_c}=\sup\mathcal W$. By Lemma~\ref{B4sym}\textit{(i)}, $\mathcal W(u^\sharp)\ge\mathcal W(u)$, so we may take a maximizing sequence $\{u_n\}\subset\mathcal X$ with $u_n\ge0$. In view of the scaling properties $\mathcal W(cu)=\mathcal W(u)=\mathcal W(t\star u)$ for $c>0$ and $t\in\R$, we may also assume that $M(u_n)=\rho^2$ and $A(u_n)=1$. Then $\{u_n\}$ is bounded in $\mathcal X$ and, up to a subsequence, $u_n\rightharpoonup w$ in $\mathcal X$. By Proposition~\ref{Y1}, we have $B(u_n)\to B(w)$. As
		\begin{align*}
			B(u_n)=\mathcal W(u_n)\rho^{q_c-2}\to C_{a,b,q_c}^{q_c}\rho^{q_c-2}>0,
		\end{align*}
		it follows that $w\ne0$. In addition, $w\ge0$.  By weak lower semicontinuity, $A(w)\le1$ and $M(w)\le\rho^2$. These bounds and $B(w)=C_{a,b,q_c}^{q_c}\rho^{q_c-2}$ give
		\begin{align*}
			C_{a,b,q_c}^{q_c}\ge\mathcal W(w)=\frac{B(w)}{A(w)M(w)^{(q_c-2)/2}}\ge\frac{C_{a,b,q_c}^{q_c}\rho^{q_c-2}}{\rho^{q_c-2}}=C_{a,b,q_c}^{q_c}.
		\end{align*}
		Since $q_c>2$, equality forces $A(w)=1$ and $M(w)=\rho^2$. Hence $w\in\S$ is an extremal function of \eqref{c1}, and $B(w)=C_{a,b,q_c}^{q_c}\rho^{q_c-2}A(w)$ with $q_c-2=\frac{4d}{N-2a+2b}$. This proves \textit{(i)}.
		
		\noi\textit{(ii)} We compute the maximum explicitly. Write $A=A(w)$, $B=B(w)$ and $C=C(w)$. By \textit{(i)} and the definition of $\beta^*$,
		\begin{align*}
			\frac12A-\frac{\beta}{q_c}B=\frac12A\Big(1-\frac{\beta}{\beta^*}\Big)=:K>0.
		\end{align*}
		Hence, with $z=e^t$, we have $\e(t\star w)=Kz^2-\frac{C}{\t}z^{\t}$. This function of $z>0$ attains its maximum at $z^{\t-2}=\frac{2K}{C}$, and the maximum equals
		\begin{align*}
			\Big(1-\frac{2}{\t}\Big)K\Big(\frac{2K}{C}\Big)^{\frac{2}{\t-2}}.
		\end{align*}
		Since $1-\frac2\t=\frac{2d}{N}$ and $\frac{2}{\t-2}=\frac{N-2d}{2d}$, this is \textit{(ii)}.
		
		\noi\textit{(iii)} By the definition of $\sr$, $C\le\sr^{-\t/2}A^{\t/2}$. Since $\frac{\t}{2}\cdot\frac{N-2d}{2d}=\frac{N}{2d}$, this implies $\omega_w\le1$. Next, by \textit{(ii)} and Lemma~\ref{Lmm}\textit{(i)},
		\begin{align*}
			m_\beta(\rho)\le\frac{d}{N}\sr^{N/2d}\Big(1-\frac{\beta}{\beta^*}\Big)^{N/2d}\omega_w^{-1}.
		\end{align*}
		The right-hand side is smaller than $\frac{d}{N}\sr^{N/2d}$ if $1-\frac{\beta}{\beta^*}<\omega_w^{2d/N}$. By the definition of $\beta_0$, this yields the bound on $\beta_0$. Finally, the right-hand side tends to $0$ as $\beta\uparrow\beta^*$, while $m_\beta(\rho)>0$ by Lemma~\ref{C33}.\QED
	\end{proof}
	
	\noindent\textbf{Proof of Theorem~\ref{Tcrit}.} \textit{(i)} By Lemma~\ref{Lbeta0}\textit{(iii)}, $\beta_0<\beta^*$. Let $\beta\in(\beta_0,\beta^*)$. Let $\{u_n\}$ be the Palais--Smale sequence provided by Lemma~\ref{LPS}.
	
	By Lemmas~\ref{C33} and~\ref{Lmm}\textit{(iii)}, $m_{\beta}(\rho)\in(0,\frac{d}{N}\sr^{N/2d})$, so Proposition~\ref{C32} applies and one of the two alternatives holds. We first exclude alternative (I). Suppose that it holds. Then $u_n\rightharpoonup\tilde{u}$ weakly in $\mathcal{X}$, but not strongly, where $\tilde u\ne 0$ is a solution of \eqref{eq:P} for some $\lambda<0$, and
	\begin{equation}\label{c43}
		\e(\tilde u)\le m_{\beta}(\rho)-\frac{d}{N}\sr^{N/2d}<0.
	\end{equation}
	On the other hand, since $P_{\beta}(\tilde u)=0$ by the Pohozaev identity and $q\delta_q=2$, we have
	\begin{align*}
		\e(\tilde u)=\e(\tilde u)-\frac12P_\beta(\tilde u)=\frac{d}{N}C(\tilde{u})>0,
	\end{align*}
	which contradicts \eqref{c43}. Therefore alternative (II) holds, that is, $u_n\to\tilde{u}$ strongly in $\mathcal{X}$, and $\tilde u$ solves \eqref{eq:P} for some $\lambda<0$. By Lemma~\ref{LPS}, $\tilde u\ge0$, and hence $\tilde u>0$ by Lemma~\ref{B3pos}. This proves \textit{(i)}.
	
	\medskip
	
	For \textit{(ii)} and \textit{(iii)}, let $\beta>0$ and $\Theta(u):=A(u)-\beta\delta_qB(u)$ for $u\in\S$. Here $q\delta_q=2$, so $\e(t\star u)=\frac12e^{2t}\Theta(u)-\frac1\t e^{\t t}C(u)$ and $P_\beta(t\star u)=e^{2t}\Theta(u)-e^{\t t}C(u)$. Thus $t\star u\in\M$ for some $t$ if and only if $\Theta(u)>0$. In this case, the fiber map has a unique critical point, namely its strict maximum point, and the computation in the proof of Lemma~\ref{Lbeta0}\textit{(ii)}, with $2K=\Theta(u)$, yields
	\begin{align*}
		\max_{t\in\R}\e(t\star u)=\frac dN\,\Theta(u)^{\frac N{2d}}C(u)^{-\frac{N-2d}{2d}},\qquad m_\beta(\rho)=\inf\Big\{\max_{t\in\R}\e(t\star u):\ u\in\S,\ \Theta(u)>0\Big\}.
	\end{align*}
	In addition, $\e(u)=\e(u)-\frac12P_\beta(u)=\frac dNC(u)>0$ for $u\in\M$, so $m_\beta(\rho)\ge0$, and the value $0$ is not attained. Finally, \eqref{c1crit} holds for every $\beta>0$, that is, $\Theta(u)\ge(1-\frac\beta{\beta^*})A(u)$ on $\S$, with equality for the function $w$ of Lemma~\ref{Lbeta0}\textit{(i)}.
	
	\noi\textit{(ii)} Let $0<\beta<\beta^*$. By the last inequality, $C(u)\le\sr^{-\frac\t2}A(u)^{\frac\t2}$ and $\frac\t2\cdot\frac{N-2d}{2d}=\frac N{2d}$, the powers of $A(u)$ cancel in the maximum formula, and $\max_{t\in\R}\e(t\star u)\ge\frac dN(1-\frac\beta{\beta^*})^{\frac N{2d}}\sr^{\frac N{2d}}=(1-\frac\beta{\beta^*})^{\frac N{2d}}\mathcal T$ for every $u\in\S$. Taking the infimum yields the lower bound. The upper bound follows from Lemma~\ref{Lbeta0}\textit{(ii)}, the definition of $\omega_w$ and Lemma~\ref{Lmm}\textit{(i)},~\textit{(ii)}.
	
	\noi\textit{(iii)} Let $\beta\ge\beta^*$, so that $\Theta(w)\le0$. For $u\in\S$ we have $M(u)=\rho^2$, and therefore
	\begin{align*}
		\Theta(u)=A(u)\big(1-\beta\delta_q\rho^{q_c-2}\,\mathcal W(u)\big),
	\end{align*}
	with $\mathcal W$ as in the proof of Lemma~\ref{Lbeta0}\textit{(i)}. We first find $v\in\S$ with $\Theta(v)>0$. Let $\phi\in C_c^\infty(\R^N\setminus\{0\})$ be radial with $\phi\not\equiv0$, and set $\phi_k(x):=\phi(x)\cos(k|x|)$. As $k\to\infty$,
	\begin{align*}
		A(\phi_k)\ge ck^2,\qquad B(\phi_k)\le B(\phi),\qquad M(\phi_k)\to\frac12M(\phi)>0,\qquad\text{so that}\qquad\mathcal W(\phi_k)\to0.
	\end{align*}
	Since $\mathcal W(cu)=\mathcal W(u)$ for $c>0$, a positive multiple $v$ of $\phi_k$, with $k$ large, lies in $\S$ and satisfies $\Theta(v)>0$. In particular, $\M\ne\emptyset$.
	
	Next, we join $w$ to $v$ in $\S$. If $v=cw$, then $M(v)=M(w)$ gives $c=\pm1$, and hence $\Theta(v)=\Theta(w)$, which is false. So $w$ and $v$ are linearly independent, $g_\tau:=(1-\tau)w+\tau v\ne0$ for $\tau\in[0,1]$, and
	\begin{align*}
		\tau\mapsto\hat g_\tau:=\frac{\rho\,g_\tau}{\||x|^{-a}g_\tau\|_2}\in\S
	\end{align*}
	is a continuous path in $\mathcal X$. Let $\tau_0:=\max\{\tau\in[0,1]:\Theta(\hat g_\tau)\le0\}$. Then $\tau_0<1$, and, by continuity,
	\begin{align*}
		\Theta(\hat g_{\tau_0})=0<\Theta(\hat g_\tau)\quad\text{for }\tau\in(\tau_0,1],\qquad C(\hat g_\tau)\to C(\hat g_{\tau_0})>0\quad\text{as }\tau\downarrow\tau_0.
	\end{align*}
	
	Finally, for $\tau\in(\tau_0,1]$, the formulas for $m_\beta(\rho)$ and for the maximum of $t\mapsto\e(t\star\hat g_\tau)$ give
	\begin{align*}
		0\le m_\beta(\rho)\le\frac dN\,\Theta(\hat g_\tau)^{\frac N{2d}}C(\hat g_\tau)^{-\frac{N-2d}{2d}}\to0\qquad(\tau\downarrow\tau_0).
	\end{align*}
	Hence $m_\beta(\rho)=0$. This value is not attained: $\e(u)=\frac dNC(u)>0$ for every $u\in\M$.\QED
	
	\subsection{The case \texorpdfstring{$q_c<q<\t$}{qc<q<2\#}}\label{ss:pfsup}
	We follow \cite[Section 6]{soave2020normalized}. The level estimate needed for Theorem~\ref{Tsup} is the bound $\beta_0\le\bar\beta$, proved in the next lemma.
	
	\begin{lemma}\label{Lbar}
		Let $q_c<q<\t$ and $\beta>0$. For every $u\in\S$,
		\begin{align*}
			m_\beta(\rho)\le\frac{q\delta_q-2}{2q\delta_q}\,\||x|^{-a}\na u\|_2^{\frac{2q\delta_q}{q\delta_q-2}}\Big(\beta\delta_q\||x|^{-b}u\|_q^q\Big)^{-\frac{2}{q\delta_q-2}}.
		\end{align*}
		In particular, $m_\beta(\rho)<\mathcal T$ for every $\beta>\bar\beta$, with $\bar\beta$ as in Theorem~\ref{Tsup}, and $\beta_0\le\bar\beta$.
	\end{lemma}
	
	\begin{proof}
		To prove the bound, fix $u\in\S$ and write $A:=A(u)$ and $B:=B(u)$. The critical term is nonnegative, so for every $t\in\R$ we have
		\begin{align*}
			\e(t\star u)\le\frac{e^{2t}}{2}A-\frac{\beta e^{q\delta_qt}}{q}B.
		\end{align*}
		Set $s:=e^{2t}$ and $h(s):=\frac{A}{2}s-\frac{\beta B}{q}s^{\frac{q\delta_q}{2}}$ for $s>0$. Since $q\delta_q>2$, $h$ attains its maximum at
		\begin{align*}
			s_0=\Big(\frac{A}{\beta\delta_qB}\Big)^{\frac{2}{q\delta_q-2}},\qquad\text{and}\qquad h(s_0)=\frac{q\delta_q-2}{2q\delta_q}A^{\frac{q\delta_q}{q\delta_q-2}}(\beta\delta_qB)^{-\frac{2}{q\delta_q-2}}.
		\end{align*}
		By Lemma~\ref{Lmm}\textit{(i)}, $m_\beta(\rho)\le\max_{t\in\R}\e(t\star u)\le h(s_0)$, which proves the bound.
		
		For the second part, note that $h(s_0)<\mathcal T$ if and only if
		\begin{align*}
			\beta>\frac{A^{q\delta_q/2}}{\delta_qB}\left(\frac{N(q\delta_q-2)}{2d\,q\delta_q\,\sr^{N/2d}}\right)^{\frac{q\delta_q-2}{2}}.
		\end{align*}
		
		Let $\mathcal Q(u):=B(u)/(A(u)^{q\delta_q/2}M(u)^{(1-\delta_q)q/2})$ for $u\in\mathcal Y\setminus\{0\}$. By  \eqref{c1}, $\sup_{\mathcal Y\setminus\{0\}}\mathcal Q=\C$. Since $\mathcal Q(u^\sharp)\ge\mathcal Q(u)$ and $u^\sharp\in\mathcal X$ by Lemma~\ref{B4sym}\textit{(i)}, the supremum over $\mathcal X\setminus\{0\}$ is also $\C$, as in the proof of Lemma~\ref{Lbeta0}\textit{(i)}. As $\mathcal Q(cu)=\mathcal Q(u)$ for $c>0$, every $u\in\mathcal X\setminus\{0\}$ can be rescaled into $\S$, and therefore
		\begin{align*}
			\inf_{u\in\S}\frac{A(u)^{q\delta_q/2}}{B(u)}=\frac{1}{\C\rho^{(1-\delta_q)q}}.
		\end{align*}
		Consequently, for every $\beta>\bar\beta$ some $u\in\S$ satisfies the above inequality for $\beta$, and the first part gives $m_\beta(\rho)<\mathcal T$. By the definition of $\beta_0$, this implies $\beta_0\le\bar\beta$.\QED
	\end{proof}
	
	\noindent\textbf{Proof of Theorem~\ref{Tsup}.} By Lemma~\ref{Lbar}, $\beta_0\le\bar\beta$. Let $\beta>\beta_0$. We follow the proof of Theorem~1.1 for $2+\frac4N<q<2^*$ in \cite{soave2020normalized} (see also \cite[Proposition 3.2]{wei2022normalized}). Consider the Palais--Smale sequence $\{u_n\}\subset \S$ for $\e|_{\S}$ at level $m_{\beta}(\rho)$, with $P_{\beta}(u_n)\to0$, given by Lemma~\ref{LPS}. By Lemmas~\ref{C33} and~\ref{Lmm}\textit{(iii)}, $m_{\beta}(\rho)\in(0,\frac{d}{N}\sr^{N/2d})$, so Proposition~\ref{C32} applies to $\{u_n\}$ and one of the alternatives holds.
	
	We exclude alternative (I) as in the case $q=q_c$. Suppose that it holds. Then, up to a subsequence, $u_n\rightharpoonup \tilde{u}$ weakly in $\mathcal{X}$ but not strongly, where $\tilde u\not\equiv 0$ is a solution of \eqref{eq:P} for some $\lambda<0$, and
	\begin{equation}\label{c45}
		\e(\tilde u)\le m_{\beta}(\rho)-\frac{d}{N}\sr^{N/2d}<0.
	\end{equation}
	On the other hand, since $P_{\beta}(\tilde{u})=0$ by the Pohozaev identity and $q\delta_q>2$, we have
	\begin{align*}
		\e(\tilde u)=\e(\tilde u)-\frac12P_\beta(\tilde u)=\frac{\beta}{q}\Big(\frac{q\delta_q}{2}-1\Big)B(\tilde u)+\frac{d}{N}C(\tilde u)>0.
	\end{align*}
	This contradicts \eqref{c45}, so alternative (II) holds and a subsequence of $\{u_n\}$ converges strongly to $\tilde{u}$ in $\mathcal{X}$. By Lemma~\ref{LPS}, $\tilde u\ge0$, and $\tilde u>0$ by Lemma~\ref{B3pos}.\QED
		\section{The threshold \texorpdfstring{$\beta_0$}{beta0} and the limit \texorpdfstring{$\beta\to0$}{beta to 0}}\label{Sbeta0}
	This section contains the proofs of Theorems~\ref{Tbeta0} and~\ref{Tasym}. The main step is the positivity of $\beta_0$ when $q_c<\frac ND$ and $q\le q^*$. For a coupled Schr\"odinger system in dimension $N=3$, the positivity of the corresponding threshold is proved in \cite[Lemma 2.9(c)]{bartsch2023existence}. That proof tests each equation with its own component. It therefore requires two components and does not apply to \eqref{eq:P}.
	
	Instead, we use a blow-up analysis of radial ground states in the Emden--Fowler variable
	\[
	\varphi(s):=e^{\gamma_as}\,u(e^s),\qquad s\in\R,\qquad\text{where }\gamma_a=\frac{N-2-2a}{2}.
	\]
	In this variable, the Caffarelli--Kohn--Nirenberg inequality for radial functions becomes an inequality on the line (see \cite[(2.1)]{catrina2001caffarelli} and \cite[(2.4)]{dolbeault2009symmetry}), and the tails of ground states can be compared with modified Bessel functions.
	
	Throughout this section, $q_c\le q<\t$, and we write $p:=\t$. We use the notation $A$, $B$, $C$ and $M$ of Sections~\ref{Sfib} and~\ref{Sproofs} and the scaling relations \eqref{dil}; in particular, $M(t\star u)=M(u)$. The numbers $\beta_0$ and $\hat\beta$ are defined in \eqref{beta0def}, and $\gamma_a$, $\kappa$, $D$ in \eqref{b0gk}.  Denote by $|B_1|$ the measure of the unit ball.
	
	\subsection{Parameter ranges and the case \texorpdfstring{$\beta_0=0$}{beta0=0}}\label{ss:b0alg}
	
	We first collect the identities between the parameters. By definition, $2\gamma_a=N-2-2a$. Hence $b+\gamma_a=\frac{N-2d}{2}$ and $N-2a+2b=2(1+b+\gamma_a)$, so that $q_c=\frac{N+2}{1+b+\gamma_a}$. From these identities and the definitions of $\delta_q$ and $\kappa$, we obtain
	\begin{align}\label{b0rel}
		p(b+\gamma_a)=N,\qquad q\delta_q=(1+b+\gamma_a)q-N,\qquad 1-\delta_q=\frac{\kappa}{q},
	\end{align}
	and, since $D=b+2\gamma_a$,
	\begin{align}\label{b0rel2}
		\kappa-(q-2)=(1+b+\gamma_a)(q_c-q),\qquad\kappa-q\gamma_a=N-qD.
	\end{align}
	In particular, $\kappa>0$ for $q<p$, $\kappa=q-2$ for $q=q_c$, and $\kappa<q-2$ and $2<q\delta_q<p$ for $q_c<q<p$. In addition, $\gamma_a<1$ if and only if $a>\frac{N-4}{2}$. Next, for $q\in\R$ we set
	\begin{align}\label{b0f}
		f(q):=\kappa(1-\gamma_a)-2\gamma_a=\frac14\Big[\big(2N-q(N-2d)\big)(4+2a-N)-4(N-2-2a)\Big],
	\end{align}
	where $\kappa=N-q(b+\gamma_a)$. The second expression follows from $2\kappa=2N-q(N-2d)$, $2(1-\gamma_a)=4+2a-N$ and $2\gamma_a=N-2-2a$.
	
	\begin{lemma}\label{b0L:alg}
		Let $q^*$ be as in Theorem~\ref{Tbeta0}.
		\begin{itemize}
			\item[\textit{(i)}] $q_c<\frac ND$ if and only if $2b+(N+4)\gamma_a<N$, and $q_c=\frac ND$ if and only if $2b+(N+4)\gamma_a=N$. If $q_c\le\frac ND$, then $\gamma_a<1$, $f$ is strictly decreasing, $f(q^*)=0$ and $q_c\le\frac ND\le q^*<p$. If $q_c<\frac ND$, then $q_c<\frac ND<q^*$.
			\item[(ii)] If $q_c<\frac ND$ and $q_c\le q\le q^*$, then $\gamma_a<1$, and $\kappa\ge q\gamma_a$ or $f(q)\ge0$.
			\item[(iii)] If $q_c<q<p$ and $q>q^*$, then $q>\frac ND$. If moreover $\gamma_a<1$, then $f(q)<0$.
		\end{itemize}
	\end{lemma}
	\begin{proof}\textit{(i)} Since $q_c=\frac{N+2}{1+b+\gamma_a}$ and $D=b+2\gamma_a$,
		\begin{align*}
			(N+2)D-N(1+b+\gamma_a)=2b+(N+4)\gamma_a-N.
		\end{align*}
		The first statement follows. Now let $q_c\le\frac ND$. Then $(N+4)\gamma_a<2b+(N+4)\gamma_a\le N$, so $\gamma_a<1$. The quantity $\kappa$ is affine in $q$ with slope $-(b+\gamma_a)$, and therefore the function $f$ is affine with slope $-(b+\gamma_a)(1-\gamma_a)<0$. At $q=\frac ND$ we have $\kappa=\frac{N(D-b-\gamma_a)}{D}=\frac{N\gamma_a}{D}$, hence
		\begin{align*}
			f\Big(\frac ND\Big)=\frac{\gamma_a}D\big(N(1-\gamma_a)-2D\big)=\frac{\gamma_a}D\big(N-(N+4)\gamma_a-2b\big)\ge0,
		\end{align*}
		with equality if and only if $q_c=\frac ND$. At $q=p$ we have $\kappa=0$ by \eqref{b0rel}, so $f(p)=-2\gamma_a<0$. Consequently, $f$ has exactly one zero, it lies in $[\frac ND,p)$, and it equals
		\begin{align*}
			\frac{1}{b+\gamma_a}\Big(N-\frac{2\gamma_a}{1-\gamma_a}\Big)=p-\frac{4(N-2-2a)}{(N-2d)(4+2a-N)},
		\end{align*}
		because $b+\gamma_a=\frac{N-2d}2$ and $\frac{2\gamma_a}{1-\gamma_a}=\frac{2(N-2-2a)}{4+2a-N}$. If $q_c<\frac ND$, this number is $q^*$, and it is larger than $\frac ND$ because $f(\frac ND)>0$. If $q_c=\frac ND$, then $f(\frac ND)=0$, so the zero is $\frac ND=q_c=q^*$. This proves \textit{(i)}.
		
		\noi\textit{(ii)} By \textit{(i)}, $\gamma_a<1$. If $q\le\frac ND$, then $\kappa-q\gamma_a=N-qD\ge0$ by \eqref{b0rel2}. If $q>\frac ND$, then $f(q)\ge f(q^*)=0$, because $f$ is decreasing and $q\le q^*$.
		
		\noi\textit{(iii)} First let $q_c\le\frac ND$. By \textit{(i)}, $q>q^*\ge\frac ND$, $\gamma_a<1$, and $f(q)<f(q^*)=0$. Now let $q_c>\frac ND$. Then $q>q_c>\frac ND$. If moreover $\gamma_a<1$, then $f$ is strictly decreasing, so it suffices to show that $f(q_c)<0$. At $q=q_c$ we have $\kappa=q_c-2$, hence
		\begin{align*}
			f(q_c)=(q_c-2)(1-\gamma_a)-2\gamma_a=q_c(1-\gamma_a)-2.
		\end{align*}
		Here $q_c(1-\gamma_a)<2$ is equivalent to $(N+2)(1-\gamma_a)<2(1+b+\gamma_a)$, that is, to $2b+(N+4)\gamma_a>N$. By the identity in the proof of \textit{(i)}, this is equivalent to $q_c>\frac ND$, which holds. Therefore $f(q)<f(q_c)<0$.\QED
	\end{proof}
	
	\begin{rem}\label{b0Rrange}
		\textit{(i)} Let $q_c<\frac ND$. By Lemma~\ref{b0L:alg}\textit{(i)}, $q_c<\frac ND<q^*<p$, $f$ is strictly decreasing and $f(q^*)=0$. Thus, for $q>\frac ND$, we have $q\le q^*$ if and only if $f(q)\ge0$. By \eqref{b0f}, the condition $q\le q^*$ is then equivalent to
		\begin{align*}
			\big(2N-q(N-2d)\big)(4+2a-N)\ge4(N-2-2a).
		\end{align*}
		
		\noi\textit{(ii)} Since $b>a$ and $2\gamma_a=N-2-2a$, we have
		\begin{align*}
			2b+(N+4)\gamma_a>\frac{(N+4)(N-2)}{2}-(N+2)a.
		\end{align*}
		The right-hand side is at least $N$ when $a\le\frac{N^2-8}{2(N+2)}$. In that case Lemma~\ref{b0L:alg}\textit{(i)} yields $q_c>\frac ND$ for every $b\in(a,1+a)$. This is the last statement of Theorem~\ref{Tbeta0}\textit{(i)}. Equivalently, if $q_c\le\frac ND$, then $a>\frac{N^2-8}{2(N+2)}$. Note also that $\frac{N^2-8}{2(N+2)}>\frac{N-4}{2}$, because $N^2-8-(N-4)(N+2)=2N>0$.
	\end{rem}
	
	\begin{lemma}\label{Lzero}
		Let $q_c\le q<\t$, and assume that $q_c\ge\frac ND$ or $q>q^*$, with $q^*$ as in Theorem~\ref{Tbeta0}. 
		Then $m_\beta(\rho)<\mathcal T$ for every $\beta\in(0,\hat\beta)$, that is, for every $\beta>0$ if $q>q_c$ and for every $\beta\in(0,\beta^*)$ if $q=q_c$. Hence $\beta_0=0$.
	\end{lemma}
	
	\begin{proof}
		Let $u_\var$ be as in \eqref{c221}. We test $m_\beta(\rho)$ with a multiple of $u_\var$, and we write
		\begin{align*}
			A_\var:=A(u_\var),\qquad B_\var:=B(u_\var),\qquad C_\var:=C(u_\var),\qquad M_\var:=M(u_\var),
		\end{align*}
		and set $\ve:=\rho\,M_\var^{-1/2}u_\var$, so that $\ve\in\S$ and
		\begin{align}\label{b0zv}
			A(\ve)=\frac{\rho^2A_\var}{M_\var},\qquad B(\ve)=\frac{\rho^qB_\var}{M_\var^{q/2}},\qquad C(\ve)=\frac{\rho^pC_\var}{M_\var^{p/2}}.
		\end{align}
		By Lemma~\ref{C10}, the map $\fom$ has a unique critical point $t_\var$, which is its maximum point, and $t_\var\star\ve\in\M$. Consequently, by Lemma~\ref{Lmm}\textit{(i)},
		\begin{align*}
			m_\beta(\rho)\le\e(t_\var\star\ve)=\max_{t\in\R}\fom(t).
		\end{align*}
		It therefore suffices to show that $\max_\R\fom<\mathcal T$ for some $\var>0$.
		If $q_c\ge\frac ND$, then $q^*=q_c$, so the assumption covers every $q\in[q_c,\t)$: the case $q=q_c$ is treated in Case~1, and the case $q>q_c=q^*$ in Case~2. If $q_c<\frac ND$, then $q>q^*>q_c$, and only Case~2 occurs.

		We compare $\fom$ with $\fo$. The map $\fo$ also has a unique critical point, namely its maximum point. Using $\frac{p}{p-2}=\frac N{2d}$ and Proposition~\ref{C11}\textit{(i)},~\textit{(ii)}, we obtain
		\begin{align}\label{b0fo}
			\max_{t\in\R}\fo(t)=\frac dN\Big(\frac{A(\ve)}{C(\ve)^{2/p}}\Big)^{\frac N{2d}}=\frac dN\Big(\frac{A_\var}{C_\var^{2/p}}\Big)^{\frac N{2d}}=\mathcal T+O(\var^{\frac{N-2d}{2d}}).
		\end{align}
		By Proposition~\ref{C11}\textit{(i)},~\textit{(ii)} again, $A_\var$ and $C_\var$ tend to $\sr^{\frac N{2d}}>0$ as $\var\to0$.
		
		\noindent\textbf{Case 1:} $q=q_c$. The assumption reduces to $q_c\ge\frac ND$, because $q^*\ge q_c$. Let $0<\beta<\beta^*=\hat \beta$. Since $q\delta_q=2$ and $\delta_q=\frac2{q_c}$, the identity $P_\beta(t_\var\star\ve)=0$ and \eqref{c1crit} yield
		\begin{align*}
			e^{(p-2)t_\var}=\frac{A(\ve)}{C(\ve)}-\frac{2\beta}{q_c}\frac{B(\ve)}{C(\ve)}\ge\Big(1-\frac{\beta}{\beta^*}\Big)\frac{A(\ve)}{C(\ve)}.
		\end{align*}
		Consequently,
		\begin{align*}
			\max_{t\in\R}\fom(t)=\fo(t_\var)-\frac{\beta}{q_c}e^{2t_\var}B(\ve)\le\max_{t\in\R}\fo(t)-\frac{\beta}{q_c}\Big(1-\frac{\beta}{\beta^*}\Big)^{\frac{2}{p-2}}\Big(\frac{A(\ve)}{C(\ve)}\Big)^{\frac{2}{p-2}}B(\ve).
		\end{align*}
		By \eqref{b0zv},
		\begin{align*}
			\Big(\frac{A(\ve)}{C(\ve)}\Big)^{\frac{2}{p-2}}B(\ve)=\rho^{q_c-2}\Big(\frac{A_\var}{C_\var}\Big)^{\frac{2}{p-2}}\frac{B_\var}{M_\var^{(q_c-2)/2}},\qquad q_c-2=\frac{4d}{N-2a+2b}.
		\end{align*}
		Since $A_\var/C_\var\to1$, \eqref{b0fo} implies, for $\var$ small,
		\begin{align*}
			\max_{t\in\R}\fom(t)\le\mathcal T+O(\var^{\frac{N-2d}{2d}})-c\,\frac{B_\var}{M_\var^{(q_c-2)/2}},
		\end{align*}
		where $c>0$ does not depend on $\var$. Here $q_c\ge\frac ND$, so Proposition~\ref{C79}\textit{(i)} and the first case of Proposition~\ref{C79}\textit{(ii)} show that the last quotient is bounded below by one of the following: a positive constant, $c|\log\var|^{-\frac{2d}{N-2a+2b}}$, $c\var^{e_c}$, or $c\var^{\frac{N-2d}{2d}}|\log\var|$, where
		\begin{align*}
			e_c:=\frac{(N-2d)(4+2a-N)}{(N-2-2a)(N-2a+2b)}.
		\end{align*}
		The third case occurs only if $q_c>\frac ND$ and $a>\frac{N-4}{2}$. Then $e_c<\frac{N-2d}{2d}$. Indeed, $\kappa=q_c-2=\frac{4d}{N-2a+2b}$ at $q=q_c$, so this inequality is equivalent to $f(q_c)<0$, which was shown in the proof of Lemma~\ref{b0L:alg}\textit{(iii)}. Thus in each case the quotient dominates $O(\var^{\frac{N-2d}{2d}})$, and therefore $\max_\R\fom<\mathcal T$ for $\var$ small, for every $\beta\in(0,\beta^*)$.
		
		\noindent\textbf{Case 2:} $q_c<q<\t$.   Clearly, $q>q^*$, as explained above. By Lemma~\ref{Lmm}\textit{(i)}, $\beta\mapsto m_\beta(\rho)$ is nonincreasing, so it suffices to find $\beta_\natural>0$ such that $m_\beta(\rho)<\mathcal T$ for $0<\beta\le\beta_\natural$.  By Lemma~\ref{b0L:alg}\textit{(iii)}, $q>\frac ND$, and $f(q)<0$ if $\gamma_a<1$, that is, if $a>\frac{N-4}{2}$. As in Case~1, we start from the identity $P_\beta(t_\var\star\ve)=0$, which now reads
		\begin{align*}
			e^{(p-2)t_\var}C(\ve)=A(\ve)-\beta\delta_qe^{(q\delta_q-2)t_\var}B(\ve).
		\end{align*}
		Since the last term is nonnegative, it follows that
		\begin{align}\label{c21}
			e^{(p-2)t_\var}\le\frac{A(\ve)}{C(\ve)}.
		\end{align}
		We may insert \eqref{c21} into the last term, because $q\delta_q>2$. Together with \eqref{b0zv}, this leads to the lower bound
		\begin{align*}
			e^{(p-2)t_\var}&\ge\frac{A(\ve)}{C(\ve)}-\beta\delta_q\Big(\frac{A(\ve)}{C(\ve)}\Big)^{\frac{q\delta_q-2}{p-2}}\frac{B(\ve)}{C(\ve)}\\
			&=\frac{A(\ve)}{C(\ve)}\bigg(1-\beta\delta_q\rho^{(1-\delta_q)q}A_\var^{-\frac{p-q\delta_q}{p-2}}C_\var^{-\frac{q\delta_q-2}{p-2}}\,\Lambda_\var\bigg),\qquad\Lambda_\var:=\frac{B_\var}{M_\var^{(1-\delta_q)q/2}}.
		\end{align*}
		Recall that $A_\var$ and $C_\var$ converge to $\sr^{\frac N{2d}}$. Hence there is $c_1\in(0,1]$ such that, for $\var$ small,
		\begin{align*}
			A_\var^{-\frac{p-q\delta_q}{p-2}}C_\var^{-\frac{q\delta_q-2}{p-2}}\le\frac1{c_1},\qquad\frac{A_\var}{C_\var}\ge c_1.
		\end{align*}
		We next estimate $\Lambda_\var$. Since $q>\frac ND$, the third case of Proposition~\ref{C11}\textit{(iii)} applies, and together with Proposition~\ref{C11}\textit{(iv)} it yields
		\begin{align}\label{c23}
			\Lambda_\var\asymp
			\begin{cases}
				1,&\text{if }0<a<\max\{\frac{N-4}{2},0\},\\
				|\log\var|^{-\frac{(1-\delta_q)q}{2}},&\text{if }0<a=\frac{N-4}{2},\\
				\var^{e},&\text{if }\max\{\frac{N-4}{2},0\}<a<\frac{N-2}{2},
			\end{cases}
		\end{align}
		where
		\begin{align*}
			e:=\frac{(N-2d)(2N-Nq+2qd)(4+2a-N)}{8d(N-2-2a)}=\frac{(N-2d)\,\kappa(1-\gamma_a)}{4d\,\gamma_a}.
		\end{align*}
		We now show that the bracket in the lower bound above is at least $\frac12$. In the first case of \eqref{c23}, let $C_*\ge\Lambda_\var$ for $\var$ small. Then
		\begin{align*}
			\beta\delta_q\rho^{(1-\delta_q)q}A_\var^{-\frac{p-q\delta_q}{p-2}}C_\var^{-\frac{q\delta_q-2}{p-2}}\Lambda_\var\le\frac{\beta\delta_q\rho^{(1-\delta_q)q}C_*}{c_1}\le\frac12\qquad\text{if }\ \beta\le\beta_\natural:=\frac{c_1}{2\delta_q\rho^{(1-\delta_q)q}C_*}.
		\end{align*}
		In the other two cases $\Lambda_\var\to0$, and we set $\beta_\natural:=\infty$.   In all cases, therefore, \eqref{b0zv} yields
		\begin{align}\label{c24}
			e^{(p-2)t_\var}\ge\frac{A(\ve)}{2C(\ve)}=\frac{\rho^{2-p}M_\var^{(p-2)/2}A_\var}{2C_\var}\ge c\,\rho^{2-p}M_\var^{(p-2)/2}.
		\end{align}
		We now insert this bound into the energy. By \eqref{b0fo}, \eqref{c24} and \eqref{b0zv},
		\begin{align*}
			\max_{t\in\R}\fom(t)&=\fo(t_\var)-\frac\beta qe^{q\delta_qt_\var}B(\ve)\le\max_{t\in\R}\fo(t)-\frac\beta q\big(e^{(p-2)t_\var}\big)^{\frac{q\delta_q}{p-2}}B(\ve)\\
			&\le\mathcal T+O(\var^{\frac{N-2d}{2d}})-c\,\beta\rho^{(1-\delta_q)q}\Lambda_\var.
		\end{align*}
		It remains to compare $\Lambda_\var$ with the error term. In the first two cases of \eqref{c23}, $\Lambda_\var$ dominates $O(\var^{\frac{N-2d}{2d}})$. In the third case $\gamma_a<1$, so $f(q)<0$, that is, $\kappa(1-\gamma_a)<2\gamma_a$ by \eqref{b0f}. Thus $e<\frac{N-2d}{2d}$, and $\Lambda_\var$ dominates $O(\var^{\frac{N-2d}{2d}})$ in this case as well. Consequently, $\max_\R\fom<\mathcal T$ for $\var$ small.
		
		Thus $m_\beta(\rho)<\mathcal T$ for $0<\beta\le\beta_\natural$, and hence, by the reduction at the beginning of Case~2, for every $\beta>0$.\QED
	\end{proof}
	
	\begin{lemma}\label{Lsol}
		Let $q_c\le q<\t$ and $0<\beta<\hat\beta$.
		\begin{itemize}
			\item[(i)] Every solution $u\in S^{\mathcal Y}_\rho$ of \eqref{eq:P} satisfies $\e(u)\ge m_\beta(\rho)$. If $m_\beta(\rho)$ is not attained, then $\e(u)>m_\beta(\rho)$.
			\item[(ii)] The map $\beta\mapsto m_\beta(\rho)$ is upper semicontinuous on $(0,\hat\beta)$. If $\beta_0>0$, then $m_\beta(\rho)=\mathcal T$ for $0<\beta\le\beta_0$, and $m_\beta(\rho)$ is not attained for $0<\beta<\beta_0$.
		\end{itemize}
	\end{lemma}
	\begin{proof}\textit{(i)} As in the proof of Proposition~\ref{Psym}, we have $P_\beta(u)=0$ and $\e(u)=\max_{t\in\R}\e(t\star u)$. By Lemma~\ref{B4sym}\textit{(i)}, $u^\sharp\in\S$ and $\e(t\star u^\sharp)\le\e(t\star u)$ for every $t\in\R$. Lemma~\ref{Lmm}\textit{(i)} therefore yields
		\begin{align*}
			m_\beta(\rho)\le\max_{t\in\R}\e(t\star u^\sharp)\le\max_{t\in\R}\e(t\star u)=\e(u).
		\end{align*}
		This proves the first statement. For the second, suppose that $\e(u)=m_\beta(\rho)$, and let $t_*$ be the maximum point of $t\mapsto\e(t\star u^\sharp)$. Then $t_*\star u^\sharp\in\M$ by Remark~\ref{CA1}, and $\e(t_*\star u^\sharp)=m_\beta(\rho)$ by the inequalities above, so that $m_\beta(\rho)$ is attained.
		
		\noi\textit{(ii)} For $u\in\S$ and $\beta\in(0,\hat\beta)$, let
		\begin{align*}
			g_u(\beta):=\max_{t\in\R}\e(t\star u),
		\end{align*}
		which is finite by Lemma~\ref{C10}. For fixed $t$, the map $\beta\mapsto\e(t\star u)$ is affine and nonincreasing, so that $g_u$ is convex, finite, and nonincreasing, and therefore continuous on $(0,\hat\beta)$. By Lemma~\ref{Lmm}\textit{(i)},
		\begin{align*}
			m_\beta(\rho)=\inf_{u\in\S}g_u(\beta),
		\end{align*}
		and an infimum of continuous functions is upper semicontinuous. This proves the first statement of \textit{(ii)}.
		
		Now let $\beta_0>0$. If $0<\beta<\beta_0$, then $m_\beta(\rho)\ge\mathcal T$ by the definition of $\beta_0$, and $m_\beta(\rho)\le\mathcal T$ by Lemma~\ref{Lmm}\textit{(ii)}; hence equality holds. At $\beta=\beta_0$, we have $m_{\beta_0}(\rho)\le\mathcal T$ by Lemma~\ref{Lmm}\textit{(ii)}, while upper semicontinuity gives
		\begin{align*}
			m_{\beta_0}(\rho)\ge\limsup_{\beta\uparrow\beta_0}m_\beta(\rho)=\mathcal T.
		\end{align*}
		
		It remains to prove the non-attainment. Let $0<\beta<\beta_0$, and suppose that $m_\beta(\rho)=\e(u)$ for some $u\in\M$. Let $\beta'\in(\beta,\beta_0)$, and let $t'$ be the maximum point of $t\mapsto\mathcal E_{\beta'}(t\star u)$ given by Lemma~\ref{C10}. Since $t=0$ is the maximum point of $t\mapsto\e(t\star u)$ and $B(t'\star u)>0$, it follows from Lemma~\ref{Lmm}\textit{(i)} that
		\begin{align*}
			m_{\beta'}(\rho)\le\mathcal E_{\beta'}(t'\star u)<\e(t'\star u)\le\e(u)=\mathcal T.
		\end{align*}
		Thus $\beta_0\le\beta'$, contradicting the choice of $\beta'$.\QED
	\end{proof}

	\subsection{One-dimensional tools}\label{ss:b0pre}
	We collect the tools on the line used in the rest of this section here. For $f\in H^1(\R)$ we write
	\begin{align*}
		\nr{f}^2:=\int_\R\big(f'^2+\gamma_a^2f^2\big)\ds.
	\end{align*}
	We shall repeatedly use the one-dimensional inequality
	\begin{align}\label{b01DS}
		\|f\|_{L^\infty(\R)}^2\le\|f\|_{L^2(\R)}\|f'\|_{L^2(\R)}\le\gamma_a^{-1}\nr{f}^2,\qquad f\in H^1(\R).
	\end{align}
	The first inequality is (6) in the proof of \cite[Theorem 8.5]{lieb2001analysis}, and the second follows from $\|f\|_2\le\gamma_a^{-1}\nr f$ and $\|f'\|_2\le\nr f$. Recall also from \cite[Theorem 8.5]{lieb2001analysis} that every $f\in H^1(\R)$ is continuous and $f(s)\to0$ as $|s|\to\infty$.
	
	\begin{lemma}[Emden--Fowler variables]\label{b0L:EF}
		Let $u\in\mathcal X\cap C^2(\R^N\setminus\{0\})$, write $u(x)=u(|x|)$, and define
		\begin{align*}
			\varphi(s):=e^{\gamma_a s}u(e^s),\qquad s\in\R.
		\end{align*}
		Then:
		\begin{itemize}
			\item[(i)] $\varphi\in C^2(\R)\cap H^1(\R)$ and $\varphi(s)\to0$ as $s\to\pm\infty$;
			\item[(ii)] the following identities hold:
			\begin{equation}\label{b0EFid}
				\begin{aligned}
					A(u)&=\SN\,\nr{\varphi}^2, & M(u)&=\SN\int_\R e^{2s}\varphi^2\ds,\\
					B(u)&=\SN\int_\R e^{\kappa s}|\varphi|^q\ds, & C(u)&=\SN\int_\R|\varphi|^p\ds;
				\end{aligned}
			\end{equation}
			\item[(iii)] if, in addition, $u$ solves \eqref{eq:P} for some $\lambda\in\R$, then
			\begin{align}\label{b0EFode}
				-\varphi''+\gamma_a^2\varphi=\lambda e^{2s}\varphi+\beta e^{\kappa s}|\varphi|^{q-2}\varphi+|\varphi|^{p-2}\varphi\qquad\text{in }\R.
			\end{align}
		\end{itemize}
		Moreover, if $w\in C^1(\R^N\setminus\{0\})$ is radial and $\varphi_w(s):=e^{\gamma_as}w(e^s)$ belongs to $H^1(\R)$, then
		\begin{align}\label{b0EFgen}
			\int_{\R^N}|x|^{-2a}|\na w|^2\dx=\SN\,\nr{\varphi_w}^2.
		\end{align}
	\end{lemma}
	For functions independent of the angular variable, the identities \eqref{b0EFid} for $A$ and $C$ are proved in \cite[Proposition 2.2 and the identities before it]{catrina2001caffarelli}, with the variable $-s$. Here we also need the decay in \textit{(i)}.
	
	\begin{proof}
		We first prove \eqref{b0EFgen} for $w$ as in the statement. Differentiating $w(t)=t^{-\gamma_a}\varphi_w(\ln t)$, we find
		\begin{align*}
			w'(t)=t^{-\gamma_a-1}\big(\varphi_w'-\gamma_a\varphi_w\big)(\ln t).
		\end{align*}
		Recall that $N-1-2a=2\gamma_a+1$. For $T_1<T_2$, polar coordinates and the substitution $t=e^s$ yield
		\begin{align}\label{b0EFann}
			\frac{1}{\SN}\int_{e^{T_1}<|x|<e^{T_2}}|x|^{-2a}|\na w|^2\dx
			&=\int_{e^{T_1}}^{e^{T_2}}t^{2\gamma_a+1}w'(t)^2\,dt=\int_{T_1}^{T_2}\big(\varphi_w'-\gamma_a\varphi_w\big)^2\ds\notag\\
			&=\int_{T_1}^{T_2}\big(\varphi_w'^2+\gamma_a^2\varphi_w^2\big)\ds-\gamma_a\Big[\varphi_w^2\Big]_{T_1}^{T_2}.
		\end{align}
		Since $\varphi_w\in H^1(\R)$, we have $\varphi_w(s)\to0$ as $|s|\to\infty$, and letting $T_1\to-\infty$ and $T_2\to\infty$ in \eqref{b0EFann}, we obtain \eqref{b0EFgen}.
		
		\noindent\textit{(i)} The function $u$ is radial and of class $C^2$ in $\R^N\setminus\{0\}$, so $\varphi\in C^2(\R)$. By polar coordinates,
		\begin{align}\label{b0EFfin}
			\int_0^\infty t^{2\gamma_a+1}u'(t)^2\,dt=\frac{A(u)}{\SN}<\infty,\qquad\int_0^\infty t^{2\gamma_a+1}u(t)^2\,dt=\frac{M(u)}{\SN}<\infty.
		\end{align}
		
		\noindent\textbf{Case 1:} $s\to+\infty$.
		As shown in the proof of \eqref{Y1pt}, $u(t)\to0$ as $t\to\infty$.
		Consequently, \eqref{Y1hol} with $r_1=r$ and $r_2\to\infty$, together with $2a+1-N=-2\gamma_a-1$ and \eqref{b0EFfin}, yields
		\begin{align*}
			r^{\gamma_a}|u(r)|\le\frac{1}{\sqrt{2\gamma_a}}\Big(\int_r^\infty t^{2\gamma_a+1}u'^2dt\Big)^{\frac12}\to0\quad\text{as }r\to\infty,
		\end{align*}
		that is, $\varphi(s)\to0$ as $s\to\infty$.
		
		\noindent\textbf{Case 2:} $s\to-\infty$. Fix $R>0$. For $0<r<R$, the fundamental theorem of calculus and the Cauchy--Schwarz inequality give
		\begin{align}\label{b0EFcs}
			|u(R)-u(r)|\le\int_r^R|u'|\,dt&\le\Big(\int_r^Rt^{-2\gamma_a-1}dt\Big)^{\frac12}\Big(\int_r^Rt^{2\gamma_a+1}u'^2dt\Big)^{\frac12}\notag\\
			&\le\frac{r^{-\gamma_a}}{\sqrt{2\gamma_a}}\Big(\int_r^Rt^{2\gamma_a+1}u'^2dt\Big)^{\frac12},
		\end{align}
		and hence
		\begin{align*}
			r^{\gamma_a}|u(r)|\le r^{\gamma_a}|u(R)|+\frac{1}{\sqrt{2\gamma_a}}\Big(\int_0^Rt^{2\gamma_a+1}u'^2dt\Big)^{\frac12}.
		\end{align*}
		Letting $r\to0$, we obtain
		\begin{align*}
			\limsup_{r\to0}r^{\gamma_a}|u(r)|\le\frac{1}{\sqrt{2\gamma_a}}\Big(\int_0^Rt^{2\gamma_a+1}u'^2dt\Big)^{\frac12}.
		\end{align*}
		By \eqref{b0EFfin}, the right-hand side tends to $0$ as $R\to0$, and therefore $\varphi(s)\to0$ as $s\to-\infty$.
		
		\medskip
		Now take $w=u$ in \eqref{b0EFann} and let $T_1\to-\infty$ and $T_2\to\infty$. The boundary term tends to $0$ by Cases 1 and 2, and monotone convergence gives
		\begin{align*}
			\nr{\varphi}^2=\frac{A(u)}{\SN}<\infty.
		\end{align*}
		Thus $\varphi\in H^1(\R)$, which proves \textit{(i)} and the identity for $A(u)$ in \textit{(ii)}.
		
		\noindent\textit{(ii)} It remains to prove the identities for $M$, $B$ and $C$. By polar coordinates, the substitution $t=e^s$ and $u(t)=t^{-\gamma_a}\varphi(\ln t)$,
		\begin{align*}
			M(u)&=\SN\int_0^\infty t^{N-1-2a-2\gamma_a}\varphi(\ln t)^2\,dt=\SN\int_\R e^{(N-2a-2\gamma_a)s}\varphi^2\ds,\\
			B(u)&=\SN\int_0^\infty t^{N-1-q(b+\gamma_a)}|\varphi(\ln t)|^q\,dt=\SN\int_\R e^{(N-q(b+\gamma_a))s}|\varphi|^q\ds,\\
			C(u)&=\SN\int_0^\infty t^{N-1-p(b+\gamma_a)}|\varphi(\ln t)|^p\,dt=\SN\int_\R e^{(N-p(b+\gamma_a))s}|\varphi|^p\ds.
		\end{align*}
		The relation $2\gamma_a=N-2-2a$ gives $N-2a-2\gamma_a=2$. Moreover, $N-q(b+\gamma_a)=\kappa$ by \eqref{b0gk}, and $N-p(b+\gamma_a)=0$ by \eqref{b0rel}. This proves \textit{(ii)}.
		
		\noindent\textit{(iii)} In radial coordinates, $\operatorname{div}(|x|^{-2a}\na u)=t^{1-N}\big(t^{2\gamma_a+1}u'\big)'$. Moreover, $u'(t)=t^{-\gamma_a-1}\big(\varphi'-\gamma_a\varphi\big)(\ln t)$, so that $t^{2\gamma_a+1}u'(t)=t^{\gamma_a}\big(\varphi'-\gamma_a\varphi\big)(\ln t)$.  Differentiating once more, we obtain
		\begin{align*}
			-\operatorname{div}(|x|^{-2a}\na u)=-t^{\gamma_a-N}\big(\varphi''-\gamma_a^2\varphi\big)(\ln t).
		\end{align*}
		We multiply \eqref{eq:P} by $t^{N-\gamma_a}$ and write $u(t)=t^{-\gamma_a}\varphi(\ln t)$. With $s=\ln t$, this gives
		\begin{align*}
			-\varphi''+\gamma_a^2\varphi=\lambda t^{N-2a-2\gamma_a}\varphi+\beta t^{N-q(b+\gamma_a)}|\varphi|^{q-2}\varphi+t^{N-p(b+\gamma_a)}|\varphi|^{p-2}\varphi .
		\end{align*}
		By the exponents computed in \textit{(ii)}, $t^{N-2a-2\gamma_a}=e^{2s}$, $t^{N-q(b+\gamma_a)}=e^{\kappa s}$ and $t^{N-p(b+\gamma_a)}=1$. This is \eqref{b0EFode}.\QED
	\end{proof}
	
	\begin{lemma}\label{b0L:comp}
		Let $s_0\in\R$, $V\in C([s_0,\infty))$ with $V>0$, and $w\in C([s_0,\infty))\cap C^2((s_0,\infty))$ such that
		\begin{align*}
			-w''+Vw\ge0\ \text{ in }(s_0,\infty),\qquad w(s_0)\ge0,\qquad\liminf_{s\to\infty}w(s)\ge0.
		\end{align*}
		Then $w\ge0$ in $[s_0,\infty)$.
	\end{lemma}
	\begin{proof}	
		We argue by contradiction. Suppose that $w(\bar s)=-m$ for some $\bar s>s_0$ and some $m>0$. By the assumption $\liminf_{s\to\infty}w(s)\ge0$, there is $R>\bar s$ such that $w>-\frac m2$ on $[R,\infty)$. The continuous function $w$ attains its minimum over $[s_0,R]$ at some point $s^*$, and
		\[
		w(s^*)\le w(\bar s)=-m<-\frac m2<\min\{w(s_0),w(R)\},
		\]
		because $w(s_0)\ge0$ and $w(R)>-\frac m2$. Hence $s^*\in(s_0,R)$ is an interior minimum point, so that $w''(s^*)\ge0$. Since $V(s^*)>0$ and $w(s^*)<0$, we obtain
		\[
		-w''(s^*)+V(s^*)w(s^*)\le V(s^*)w(s^*)<0,
		\]
		which contradicts the differential inequality at $s^*$.\QED
	\end{proof}
	
	The comparison functions below are built from modified Bessel functions. For $\nu>0$, let $K_\nu$ denote the modified Bessel function of the second kind. We use the following facts from \cite[Chapter 10]{dlmf}.
	\begin{itemize}
		\item[(K1)] $K_\nu$ is real-analytic on $(0,\infty)$ and solves $z^2K_\nu''+zK_\nu'-(z^2+\nu^2)K_\nu=0$ \cite[(10.25.1), \textsection\,10.25(ii)]{dlmf}. Moreover, $K_\nu(x)=\int_0^\infty e^{-x\cosh t}\cosh(\nu t)\,dt>0$ for every real $\nu$ and $x>0$ \cite[(10.32.9)]{dlmf}.
		\item[(K2)] $\frac{d}{dz}\big(z^\nu K_\nu(z)\big)=-z^\nu K_{\nu-1}(z)$. This is \cite[(10.29.4)]{dlmf} with $k=1$ and $\mathscr Z_\nu(z)=e^{\nu\pi i}K_\nu(z)$ (see \cite[\textsection\,10.25(ii)]{dlmf}).
		\item[(K3)] $K_\nu(z)\sim\frac12\Gamma(\nu)(\frac12z)^{-\nu}$ as $z\to0^+$ \cite[(10.30.2)]{dlmf}, that is, $\lim_{z\to0^+}z^\nu K_\nu(z)=2^{\nu-1}\Gamma(\nu)$.
		\item[(K4)] $K_\nu(z)\sim\sqrt{\pi/(2z)}\,e^{-z}$ as $z\to\infty$ \cite[(10.25.3)]{dlmf}. Hence there is $c_\nu>0$ such that $K_\nu(z)\le c_\nu z^{-1/2}e^{-z}$ for $z\ge1$.
	\end{itemize}
	
	\begin{lemma}\label{b0L:Bes}
		Let $\nu>0$ and $\ell\in\R$, and set
		\begin{align*}
			\Psi_\nu(s):=K_\nu(e^{s-\ell}),\quad s\in\R,\qquad G_\nu(z):=z^\nu K_\nu(z),\quad z>0,\qquad g_\nu:=2^{\nu-1}\Gamma(\nu).
		\end{align*}
		\begin{itemize}
			\item[(i)] $\Psi_\nu\in C^\infty(\R)$, $\Psi_\nu>0$, $\Psi_\nu(s)\to0$ as $s\to\infty$, and
			\begin{align}\label{b0Besode}
				-\Psi_\nu''+\big(\nu^2+e^{-2\ell}e^{2s}\big)\Psi_\nu=0\qquad\text{in }\R.
			\end{align}
			\item[(ii)] $G_\nu$ is decreasing on $(0,\infty)$, $0<G_\nu\le g_\nu$, and $G_\nu(z)\le c_\nu z^{\nu-1/2}e^{-z}$ for $z\ge1$.
			\item[(iii)] Let $s_0\le s$, $z:=e^{s-\ell}$ and $z_0:=e^{s_0-\ell}$. Then
			\begin{align}
				\frac{\Psi_\nu(s)}{\Psi_\nu(s_0)}&\ge\frac{K_\nu(1)}{g_\nu}\,e^{-\nu(s-s_0)}&&\text{if }s\le\ell,\label{b0Beslow}\\
				\frac{\Psi_\nu(s)}{\Psi_\nu(s_0)}&\le\frac{1}{e^{-\nu}K_\nu(e^{-1})}\,e^{-\nu(s-s_0)}\,G_\nu(z)&&\text{if }s_0\le\ell-1.\label{b0Besup}
			\end{align}
			\item[(iv)] For $\sigma\le\ell-1$ and $s\ge\sigma$,
			\begin{align*}
				\Psi_\nu(s)\le\hat C_\nu e^{-\nu(s-\sigma)}\Psi_\nu(\sigma),\qquad\hat C_\nu:=\frac{g_\nu}{e^{-\nu}K_\nu(e^{-1})}.
			\end{align*}
			\item[(v)] $K_\nu$ is decreasing on $(0,\infty)$, and $K_\nu(z)\ge K_\nu(1)z^{-\nu}$ for $0<z\le1$. If $0<\nu<1$, then $J_\nu:=\int_0^\infty zK_\nu(z)^2\,dz<\infty$.
		\end{itemize}
	\end{lemma}
	\begin{proof}\textit{(i)} With $z=e^{s-\ell}$ we have
		\begin{align*}
			\frac{d}{ds}=z\frac{d}{dz},\qquad\frac{d^2}{ds^2}=z^2\frac{d^2}{dz^2}+z\frac{d}{dz}.
		\end{align*}
		Property \textit{(i)} therefore follows from (K1) and (K4).
		
		\noi\textit{(ii)} By (K2) and (K1),
		\begin{align*}
			G_\nu'(z)=-z^\nu K_{\nu-1}(z)<0,
		\end{align*}
		and $G_\nu(z)\to g_\nu$ as $z\to0^+$ by (K3). Thus $G_\nu$ is decreasing and $0<G_\nu\le g_\nu$. The bound for $z\ge1$ follows from (K4).
		
		\noi\textit{(iii)} Writing $K_\nu(z)=z^{-\nu}G_\nu(z)$ and using $\frac{z_0}{z}=e^{-(s-s_0)}$, we obtain
		\begin{align*}
			\frac{\Psi_\nu(s)}{\Psi_\nu(s_0)}=e^{-\nu(s-s_0)}\frac{G_\nu(z)}{G_\nu(z_0)}.
		\end{align*}
		If $s\le\ell$, then $z\le1$, and by \textit{(ii)} we have $G_\nu(z)\ge G_\nu(1)=K_\nu(1)$ and $G_\nu(z_0)\le g_\nu$, which proves \eqref{b0Beslow}. If $s_0\le\ell-1$, then $z_0\le e^{-1}$, so $G_\nu(z_0)\ge G_\nu(e^{-1})=e^{-\nu}K_\nu(e^{-1})$, which proves \eqref{b0Besup}.
		
		\noi\textit{(iv)} This follows from \eqref{b0Besup}, with $\sigma$ in place of $s_0$, and from $G_\nu\le g_\nu$.
		
		\noi\textit{(v)} The function $K_\nu=z^{-\nu}G_\nu$ is a product of positive decreasing functions, so it is decreasing. For $0<z\le1$ we have $K_\nu(z)\ge z^{-\nu}G_\nu(1)=K_\nu(1)z^{-\nu}$. If $0<\nu<1$, then
		\begin{align*}
			zK_\nu(z)^2\le g_\nu^2z^{1-2\nu}\ \text{ on }(0,1),\qquad zK_\nu(z)^2\le c_\nu^2e^{-2z}\ \text{ on }(1,\infty),
		\end{align*}
		so $J_\nu<\infty$.\QED
	\end{proof}
	
	For $f\in H^1(\R)\setminus\{0\}$ let
	\begin{align*}
		Q(f):=\frac{\nr{f}^2}{\big(\int_\R|f|^p\ds\big)^{2/p}},\qquad\Sigma:=\inf_{f\in H^1(\R)\setminus\{0\}}Q(f).
	\end{align*}
	By \eqref{b01DS},
	\begin{align*}
		\int_\R|f|^p\ds\le\|f\|_\infty^{p-2}\|f\|_2^2\le C\nr f^p,\qquad C:=\gamma_a^{-(p+2)/2}.
	\end{align*}
	Hence $\Sigma\ge C^{-2/p}>0$. Applied to $f-g$, the same bound yields $\big|\|f\|_p-\|g\|_p\big|\le C^{1/p}\nr{f-g}$, so that $Q$ is continuous on $H^1(\R)\setminus\{0\}$. The exponent $p=\t$ is not critical on the line, since $H^1(\R)\subset L^r(\R)$ for every $r\in[2,\infty]$ by \eqref{b01DS}. Rather, it is the exponent for which the weight disappears from the identity for $C(u)$ in \eqref{b0EFid}, because $N-\t(b+\gamma_a)=0$. Consequently, $Q$ is invariant under the translations $f\mapsto f(\cdot+y)$, which correspond to the dilations $u\mapsto\sigma^{\gamma_a}u(\sigma\,\cdot)$, $\sigma=e^y$, in $\R^N$. These translations account for the shifts $y_n$ in Lemma~\ref{b0L:CC}.
	
	The constant $\Sigma$ is related to $\sr$ as follows. For functions on $\R\times\mathbb S^{N-1}$ that do not depend on the angular variable, the quotient $F_{a,b}$ of \cite[(2.2)]{catrina2001caffarelli} equals $\SN^{1-2/p}Q$. Therefore $\SN^{1-2/p}\Sigma$ is the radial constant $R(a,b)$ of \cite[(2.6)]{catrina2001caffarelli}, which equals $\sr$ for $a>0$ by \cite[Remark 2.7]{catrina2001caffarelli} (see also \cite[(2.14), (2.15)]{dolbeault2009symmetry}). We need only the following inequality. The equality is recovered in Section~\ref{ss:b0sup}.
	
	\begin{lemma}\label{b0L:Sigma}
		$\Sigma\ge\SN^{\frac2p-1}\sr$.
	\end{lemma}
	\begin{proof}
		First let $f\in C_c^\infty(\R)\setminus\{0\}$ and set $u(x):=|x|^{-\gamma_a}f(\ln|x|)$. Then $u\in C_c^\infty(\R^N\setminus\{0\})$ is radial and $e^{\gamma_as}u(e^s)=f(s)$. By \eqref{b0EFgen}, and as in the proof of Lemma~\ref{b0L:EF}\textit{(ii)},
		\begin{align*}
			A(u)=\SN\nr f^2,\qquad C(u)=\SN\int_\R|f|^p\ds.
		\end{align*}
		Therefore
		\begin{align*}
			\sr\le\frac{A(u)}{C(u)^{2/p}}=\SN^{1-\frac2p}Q(f).
		\end{align*}
		Since $C_c^\infty(\R)$ is dense in $H^1(\R)$ and $Q$ is continuous, the inequality extends to every $f\in H^1(\R)\setminus\{0\}$.\QED
	\end{proof}
	
	\begin{lemma}\label{b0L:CC}
		Let $\{f_n\}\subset H^1(\R)$ with $f_n\ge0$, and assume that there is $L>0$ with
		\begin{align*}
			\nr{f_n}^2\to L,\qquad\int_\R f_n^p\ds\to L,\qquad L^{1-\frac2p}=\Sigma.
		\end{align*}
		Then there are $y_n\in\R$ and $\psi\in H^1(\R)$ such that, up to a subsequence, $f_n(\cdot+y_n)\to\psi$ in $H^1(\R)$ and in $L^\infty(\R)$. Moreover, $\psi\in C^2(\R)$, $\psi>0$ in $\R$, $\psi(\sigma)\to0$ as $|\sigma|\to\infty$,
		\begin{align}\label{b0psiode}
			-\psi''+\gamma_a^2\psi=\psi^{p-1}\quad\text{in }\R,
		\end{align}
		and
		\begin{align}\label{b0psilow}
			\psi(\sigma)\ge\psi(0)e^{-\gamma_a\sigma}\qquad\text{for }\sigma\ge0.
		\end{align}
	\end{lemma}
	\begin{proof}
		We first exclude vanishing. For $y\in\R$, let $I_y:=(y-1,y+1)$. For $f\in H^1(\R)$ and $x,x'\in I_y$ we have
		\begin{align*}
			f(x)^2=f(x')^2+2\int_{x'}^xff'.
		\end{align*}
		Averaging over $x'\in I_y$, we obtain
		\begin{align*}
			f(x)^2\le\frac12\|f\|_{L^2(I_y)}^2+2\|f\|_{L^2(I_y)}\|f'\|_{L^2(I_y)},\qquad x\in I_y.
		\end{align*}
		Now suppose that $\sup_y\|f_n\|_{L^2(I_y)}\to0$ along a subsequence. Since $\{f_n\}$ is bounded in $H^1(\R)$, this estimate yields $\|f_n\|_\infty\to0$, and then
		\begin{align*}
			\int_\R f_n^p\ds\le\|f_n\|_\infty^{p-2}\gamma_a^{-2}\nr{f_n}^2\to0,
		\end{align*}
		which contradicts $L>0$. Hence there are $c>0$ and $y_n\in\R$ such that $\int_{I_{y_n}}f_n^2\ge c$ for $n$ large.
		
		Next, we prove strong convergence. Let $\psi_n:=f_n(\cdot+y_n)$. Up to a subsequence, $\psi_n\rightharpoonup\psi$ in $H^1(\R)$ and $\psi_n\to\psi$ uniformly on compact intervals. Thus $\psi\ge0$ and $\int_{-1}^1\psi^2\ge c$. Let $w_n:=\psi_n-\psi$. Then
		\begin{align*}
			\nr{\psi_n}^2=\nr\psi^2+\nr{w_n}^2+o(1),\qquad\int_\R\psi_n^p=\int_\R\psi^p+\int_\R|w_n|^p+o(1),
		\end{align*}
		where the second identity is a consequence of the Brezis--Lieb lemma \cite{brezis1983relation}. Up to a subsequence, $\nr{w_n}^2\to\omega$ and $\int_\R|w_n|^p\to L_2$. Set $L_1:=\int_\R\psi^p>0$. By the definition of $\Sigma$ and the subadditivity of $t\mapsto t^{2/p}$,
		\begin{align*}
			L=\nr\psi^2+\omega\ge\Sigma\big(L_1^{2/p}+L_2^{2/p}\big)\ge\Sigma(L_1+L_2)^{2/p}=\Sigma L^{2/p}=L.
		\end{align*}
		Equality therefore holds throughout. Since $t\mapsto t^{2/p}$ is strictly concave and $L_1>0$, this forces $L_2=0$. Then $L_1=L$, $\nr\psi^2\ge\Sigma L^{2/p}=L$, and $\omega=0$. Thus $\psi_n\to\psi$ in $H^1(\R)$, and in $L^\infty(\R)$ by \eqref{b01DS}. In addition, $\nr\psi^2=\int_\R\psi^p=L$, so that $\psi$ minimizes $Q$.
		
		The Euler--Lagrange equation of $Q$ at $\psi$, together with $\nr\psi^2=\int_\R\psi^p$, is the weak form of \eqref{b0psiode}. By the continuity of its right-hand side, $\psi\in C^2(\R)$. If $\psi(\sigma_1)=0$ for some $\sigma_1$, then $\psi'(\sigma_1)=0$ because $\psi\ge0$, and, since $p>2$, uniqueness for the initial value problem yields $\psi\equiv0$, which is impossible. Hence $\psi>0$.
		
		Finally, to prove the lower bound, we determine the limit explicitly. By \eqref{b0psiode}, $\psi''\in L^2(\R)$, so $\psi'\in H^1(\R)$ and $\psi'(\sigma)\to0$ as $|\sigma|\to\infty$. Multiplying \eqref{b0psiode} by $\psi'$ and integrating, we obtain
		\begin{align*}
			\psi'^2=\gamma_a^2\psi^2-\frac2p\psi^p.
		\end{align*}
		Evaluating this identity at a maximum point $\sigma_0$ of $\psi$ gives $\psi(\sigma_0)^{p-2}=\frac{p\gamma_a^2}{2}$. By uniqueness, $\psi=\psi_0(\cdot-\sigma_0)$, where
		\begin{align*}
			\psi_0(\sigma):=\Big(\frac{p\gamma_a^2}{2}\Big)^{\frac1{p-2}}\operatorname{sech}^{\frac2{p-2}}\Big(\frac{(p-2)\gamma_a\sigma}{2}\Big)
		\end{align*}
		is the function in \cite[(2.8)]{catrina2001caffarelli} with $\lambda=\gamma_a$; see also \cite[(2.14)]{dolbeault2009symmetry}. Since $\cosh(x+y)\le e^{|x|}\cosh y$, we obtain
		\begin{align*}
			\psi(\sigma)\ge e^{-\gamma_a|\sigma|}\psi(0),\qquad\sigma\in\R,
		\end{align*}
		which contains \eqref{b0psilow}.\QED
	\end{proof}
	
	\begin{lemma}\label{b0L:Strauss}
		Let $v\in L^2(\R^N)$ be nonnegative, radially symmetric, and radially nonincreasing. Then $v(r)\le|B_1|^{-1/2}\|v\|_{L^2(\R^N)}\,r^{-N/2}$ for every $r>0$.
	\end{lemma}
	\begin{proof}
		Since $v(y)\ge v(r)$ for $|y|<r$, we have
		\begin{align*}
			|B_1|r^Nv(r)^2\le\int_{B_r}v^2\dy\le\|v\|_2^2.
		\end{align*}
		This proves the lemma.\QED
	\end{proof}
	
	\subsection{The mass-supercritical case}\label{ss:b0sup}
	\begin{prop}\label{Pb0sup}
		Let $q_c<q<\t$, and assume that $q_c<\frac ND$ and $q\le q^*$. Then $\beta_0>0$.
	\end{prop}
	
	\begin{proof}
		By Lemma~\ref{b0L:alg}\textit{(ii)},
		\begin{align}\label{b0H}
			\gamma_a<1,\qquad\text{and}\qquad f(q)\ge0\ \text{ whenever }\ \kappa<q\gamma_a.
		\end{align}
		By \eqref{b0rel2}, $\kappa-q\gamma_a=N-qD$, so that $\kappa<q\gamma_a$ if and only if $q>\frac ND$.
		
		Suppose, by contradiction, that $\beta_0=0$, and let $\beta_n\in(0,1)$ with $\beta_n\to0$. Since $\beta_n>\beta_0$, Theorem~\ref{Tsup} provides a ground state $u_n\in\mathcal M_{\rho,\beta_n}$ with $\mathcal E_{\beta_n}(u_n)=m_{\beta_n}(\rho)$, which solves \eqref{eq:P} with $\beta=\beta_n$ and some $\lambda=\lambda_n<0$. 
		For the same reason, case \textit{(iii)} of Proposition~\ref{Psym} applies: $u_n>0$ is a ground state in $S^{\mathcal Y}_\rho$, and $v_n:=|x|^{-a}u_n$ is radially nonincreasing, because $v_n$ coincides with its Schwarz symmetrization (Step~3 of the proof of Proposition~\ref{Psym}).  In addition, $u_n\in C^2(\R^N\setminus\{0\})$ by Lemma~\ref{B1reg}. We write $A_n:=A(u_n)$ and $B_n:=B(u_n)$.  By Lemma~\ref{B2poh} and \eqref{b0rel},
		\begin{align}\label{b0Poh}
			P_{\beta_n}(u_n)=0,\qquad|\lambda_n|\rho^2=\beta_n(1-\delta_q)B_n=\frac{\kappa}{q}\beta_nB_n.
		\end{align}
		Steps~1--6 below provide estimates for $u_n$ in Emden--Fowler variables, and Step~7 shows that these estimates are incompatible with $\beta_n\to0$. As in the rest of the paper, $C$ and $c$ denote positive constants, here independent of $n$, which may change from line to line, while the numbered constants $C_0,~C_1,\dots$ are fixed once defined.
		
		\noindent\textbf{Step 1. Energy estimates.} We first show that there is $C_1>0$ such that
		\begin{align}\label{b0mlow}
			m_\beta(\rho)\ge\mathcal T-C_1\beta\qquad\text{for every }\beta>0.
		\end{align}
		Let $u\in\M$, and choose $\bar t\in\R$ with $A(\bar t\star u)=\sr^{\frac N{2d}}$. By Lemma~\ref{C10}, $t=0$ is the maximum point of $t\mapsto\e(t\star u)$, so $\e(u)\ge\e(\bar t\star u)$. By the definition of $\sr$ and by \eqref{c1},
		\begin{align*}
			C(\bar t\star u)\le\sr^{-\frac p2}A(\bar t\star u)^{\frac p2}=\sr^{\frac N{2d}},\qquad B(\bar t\star u)\le C_{a,b,q}^{\,q}\rho^{(1-\delta_q)q}\sr^{\frac{Nq\delta_q}{4d}}=:qC_1.
		\end{align*}
		It implies
		\begin{align*}
			\e(u)\ge\e(\bar t\star u)\ge\frac12\sr^{\frac N{2d}}-C_1\beta-\frac1p\sr^{\frac N{2d}}=\mathcal T-C_1\beta,
		\end{align*}
		and \eqref{b0mlow} follows. In particular, by \eqref{b0mlow} and Lemma~\ref{Lmm}\textit{(ii)},
		\begin{align}\label{b0Elim}
			\mathcal T-C_1\beta_n\le\mathcal E_{\beta_n}(u_n)\le\mathcal T.
		\end{align}
		We now deduce the limits of the three terms of the energy. We have $\beta_nB_n=\delta_q^{-1}(A_n-C(u_n))$ by $P_{\beta_n}(u_n)=0$, and hence
		\begin{align*}
			\mathcal E_{\beta_n}(u_n)=\Big(\frac12-\frac1{q\delta_q}\Big)A_n+\Big(\frac1{q\delta_q}-\frac1p\Big)C(u_n).
		\end{align*}
		Both coefficients are positive, because $2<q\delta_q<p$ by \eqref{b0rel}, so that $A_n$ and $C(u_n)$ are bounded. By \eqref{c1}, $B_n$ is then bounded as well, and $\beta_nB_n\to0$. Consequently,
		\begin{align*}
			A_n-C(u_n)=\beta_n\delta_qB_n\to0,\qquad \mathcal E_{\beta_n}(u_n)=\frac dNA_n-\frac{\beta_n}{q}B_n+\frac1p\big(A_n-C(u_n)\big)=\frac dNA_n+o(1).
		\end{align*}
		Now \eqref{b0Elim} and \eqref{b0Poh} yield
		\begin{align}\label{b0lim}
			A_n\to\sr^{\frac N{2d}},\qquad C(u_n)\to\sr^{\frac N{2d}},\qquad\beta_nB_n\to0,\qquad\lambda_n\to0.
		\end{align}
		
		\noindent\textbf{Step 2. The profile.} We now pass to Emden--Fowler variables and set $\varphi_n(s):=e^{\gamma_a s}u_n(e^s)$. By Lemma~\ref{b0L:EF}, $\varphi_n\in C^2(\R)\cap H^1(\R)$ is positive, tends to $0$ at $\pm\infty$, and satisfies
		\begin{align}\label{b0odeN}
			-\varphi_n''+\gamma_a^2\varphi_n=\lambda_ne^{2s}\varphi_n+\beta_ne^{\kappa s}\varphi_n^{q-1}+\varphi_n^{p-1}\quad\text{in }\R,
		\end{align}
		and
		\begin{align}\label{b0idN}
			\nr{\varphi_n}^2=\frac{A_n}{\SN},\qquad\int_\R\varphi_n^p\ds=\frac{C_n}{\SN},\qquad\int_\R e^{2s}\varphi_n^2\ds=\frac{\rho^2}{\SN},\qquad\int_\R e^{\kappa s}\varphi_n^q\ds=\frac{B_n}{\SN}.
		\end{align}
		By \eqref{b01DS} and \eqref{b0lim}, there is $C_0>0$ such that
		\begin{align}\label{b0Linf}
			\|\varphi_n\|_\infty^2\le\gamma_a^{-1}\nr{\varphi_n}^2\le C_0.
		\end{align}
		Next, we show that $\{\varphi_n\}$ is a minimizing sequence for $\Sigma$. Let $\bar L:=\sr^{\frac N{2d}}/\SN$. By \eqref{b0lim}, $\nr{\varphi_n}^2\to\bar L$ and $\int_\R\varphi_n^p\to\bar L$. It follows that
		\begin{align*}
			Q(\varphi_n)\to\bar L^{1-2/p}=\SN^{\frac2p-1}\sr.
		\end{align*}
		On the other hand, $Q(\varphi_n)\ge\Sigma$ for every $n$, by the definition of $\Sigma$. Together with Lemma~\ref{b0L:Sigma}, this gives
		\begin{align*}
			\bar L^{1-2/p}=\lim_{n\to\infty}Q(\varphi_n)\ge\Sigma\ge\SN^{\frac2p-1}\sr=\bar L^{1-2/p}.
		\end{align*}
		Hence $\Sigma=\bar L^{1-2/p}$, that is, $\{\varphi_n\}$ is a minimizing sequence for $\Sigma$. Since, in addition, $\varphi_n>0$, Lemma~\ref{b0L:CC} applies with $f_n=\varphi_n$ and $L=\bar L$. It provides $y_n\in\R$ and a positive solution $\psi\in H^1(\R)$ of \eqref{b0psiode} satisfying \eqref{b0psilow} such that, up to a subsequence,
		\begin{align}\label{b0prof}
			\psi_n:=\varphi_n(\cdot+y_n)\to\psi\quad\text{in }H^1(\R)\text{ and in }L^\infty(\R).
		\end{align}
		We claim that $y_n\to-\infty$. Otherwise, $y_n\ge-Y$ along a subsequence, for some $Y>0$. For every $R>0$, by \eqref{b0idN}, \eqref{b0prof} and \eqref{b0psilow},
		\begin{align*}
			\frac{\rho^2}{\SN}\ge e^{2y_n}\int_0^Re^{2\sigma}\psi_n(\sigma)^2d\sigma\ge e^{-2Y}\Big(\psi(0)^2\frac{e^{(2-2\gamma_a)R}-1}{2-2\gamma_a}-o(1)\Big).
		\end{align*}
		Letting $n\to\infty$ with $R$ fixed, we obtain
		\begin{align*}
			\frac{\rho^2}{\SN}\ge e^{-2Y}\psi(0)^2\,\frac{e^{(2-2\gamma_a)R}-1}{2-2\gamma_a}\qquad\text{for every }R>0.
		\end{align*}
		Here $\gamma_a<1$ by \eqref{b0H}, so the right-hand side tends to $\infty$ as $R\to\infty$, which is impossible. Therefore $y_n\to-\infty$.
		
		\noindent\textbf{Step 3. The subcritical coefficient.} Since $\gamma_a+a=\frac{N-2}{2}$, we have $\varphi_n(s)=e^{\frac{(N-2)s}2}v_n(e^s)$, and $\|v_n\|_{L^2(\R^N)}^2=M(u_n)=\rho^2$. By Lemma~\ref{b0L:Strauss},
		\begin{align}\label{b0Str}
			\varphi_n(s)\le C_2e^{-s},\qquad s\in\R,\qquad C_2:=|B_1|^{-1/2}\rho.
		\end{align}
		With this bound, we show that the subcritical term $\beta_ne^{\kappa s}\varphi_n^{q-1}$ in \eqref{b0odeN}, written as $W_n\varphi_n$, has a uniformly small coefficient. Here
		\begin{align*}
			W_n(s):=\beta_ne^{\kappa s}\varphi_n(s)^{q-2}.
		\end{align*}
		For $s\le0$, \eqref{b0Linf} and $\kappa>0$ give $W_n(s)\le\beta_nC_0^{(q-2)/2}$. For $s\ge0$, \eqref{b0Str} and $\kappa<q-2$ (see \eqref{b0rel2}) yield
		\begin{align*}
			W_n(s)\le\beta_nC_2^{q-2}e^{(\kappa-q+2)s}\le\beta_nC_2^{q-2}.
		\end{align*}
		Consequently,
		\begin{align}\label{b0Wsmall}
			\sup_{s\in\R}W_n(s)\le C_3\beta_n,\qquad C_3:=C_0^{(q-2)/2}+C_2^{q-2}.
		\end{align}
		
		\noindent\textbf{Step 4. Choice of parameters.} We first fix an exponent $\nu$, distinguishing two cases.
		
		\emph{Case 1: $\kappa<q\gamma_a$, that is, $q>\frac ND$.} By \eqref{b0H}, $f(q)\ge0$, that is, $\kappa(1-\gamma_a)\ge2\gamma_a$.
		We fix $\nu\in\big(\frac\kappa q,\gamma_a\big)$.
		
		\emph{Case 2: $\kappa\ge q\gamma_a$, that is, $q\le\frac ND$.} We can fix $\nu\in\big(\frac{\gamma_a(2+\kappa)}q,\gamma_a\big)$, because $2+\kappa<q$ by \eqref{b0rel2}. Then $q\nu<\kappa$.
		
		In both cases $0<\nu<\gamma_a$, and we set $\eta:=\gamma_a^2-\nu^2>0$. Next, we fix a point in the tail of the profile: since $\psi(\sigma)\to0$ as $\sigma\to\infty$, we can fix $\sigma_*\ge1$ such that
		\begin{align*}
			\sup_{\sigma\ge \sigma_*}\psi(\sigma)\le\frac12\Big(\frac\eta2\Big)^{\frac1{p-2}}.
		\end{align*}
		We also set
		\begin{align*}
			\ell_n:=-\tfrac12\ln|\lambda_n|,\qquad\text{so that}\qquad|\lambda_n|=e^{-2\ell_n}\quad\text{and}\quad\ell_n\to\infty.
		\end{align*}
		By \eqref{b0prof}, \eqref{b0Wsmall} and Step~2, there is $n_0$ such that, for $n\ge n_0$,
		\begin{align*}
			\|\psi_n-\psi\|_\infty\le\min\Big\{\frac12\Big(\frac\eta2\Big)^{\frac1{p-2}},\frac{\psi(\sigma_*)}2\Big\},\qquad C_3\beta_n\le\frac\eta2,\qquad s_0:=y_n+\sigma_*\le\ell_n-1.
		\end{align*}
		From now on, $n\ge n_0$. The choice of $n_0$ and \eqref{b0Linf}, together with $\varphi_n(s_0)=\psi_n(\sigma_*)$, yield
		\begin{align}\label{b0s0}
			\frac{\psi(\sigma_*)}2\le\varphi_n(s_0)\le C_0^{1/2}.
		\end{align}
		Similarly, for $s\ge s_0$ we have $\varphi_n(s)=\psi_n(s-y_n)\le(\eta/2)^{1/(p-2)}$, which together with \eqref{b0Wsmall} implies
		\begin{align}\label{b0cn}
			c_n(s):=\varphi_n(s)^{p-2}+W_n(s)\le\eta\qquad\text{for }s\ge s_0.
		\end{align}
		Finally, substituting $\lambda_n=-e^{-2\ell_n}$ into \eqref{b0odeN}, we obtain
		\begin{align}\label{b0odeN2}
			-\varphi_n''+\big(\gamma_a^2+e^{2(s-\ell_n)}\big)\varphi_n=c_n(s)\varphi_n\quad\text{in }\R.
		\end{align}
		
		\noindent\textbf{Step 5. Two-sided bounds for $\varphi_n$.} Both bounds follow by comparison with Bessel functions. For the lower bound, let $\Psi_{\gamma_a}(s):=K_{\gamma_a}(e^{s-\ell_n})$ and
		\begin{align*}
			w:=\varphi_n-\frac{\varphi_n(s_0)}{\Psi_{\gamma_a}(s_0)}\Psi_{\gamma_a}.
		\end{align*}
		By \eqref{b0odeN2}, \eqref{b0Besode} and $c_n\ge0$,
		\begin{align*}
			-w''+\big(\gamma_a^2+e^{2(s-\ell_n)}\big)w=c_n\varphi_n\ge0\quad\text{in }(s_0,\infty),\qquad w(s_0)=0,
		\end{align*}
		and $w(s)\to0$ as $s\to\infty$. Hence $w\ge0$ by Lemma~\ref{b0L:comp}, and \eqref{b0Beslow} and \eqref{b0s0} then give
		\begin{align}\label{b0lowb}
			\varphi_n(s)\ge c_4e^{-\gamma_a(s-s_0)}\quad\text{for }s_0\le s\le\ell_n,\qquad c_4:=\frac{\psi(\sigma_*)K_{\gamma_a}(1)}{2^{\gamma_a}\Gamma(\gamma_a)}.
		\end{align}
		Combined with the mass constraint, this lower bound leads to an upper bound for $\ell_n$. Indeed, since $s_0\le\ell_n-1$ and $2-2\gamma_a>0$, \eqref{b0idN} and \eqref{b0lowb} imply
		\begin{align*}
			\frac{\rho^2}{\SN}\ge\int_{\ell_n-1}^{\ell_n}e^{2s}\varphi_n^2\ds\ge c_4^2e^{2\gamma_a s_0}\int_{\ell_n-1}^{\ell_n}e^{(2-2\gamma_a)s}\ds\ge c_4^2e^{2\gamma_a s_0}e^{(2-2\gamma_a)(\ell_n-1)},
		\end{align*}
		that is,
		\begin{align}\label{b0ellb}
			e^{(1-\gamma_a)\ell_n}\le C_5e^{-\gamma_a s_0},\qquad\text{equivalently}\qquad|\lambda_n|\ge C_5^{-\frac2{1-\gamma_a}}e^{\frac{2\gamma_a s_0}{1-\gamma_a}}.
		\end{align}
		For the upper bound, \eqref{b0odeN2}, \eqref{b0cn} and $\nu^2=\gamma_a^2-\eta$ give
		\begin{align*}
			-\varphi_n''+\big(\nu^2+e^{2(s-\ell_n)}\big)\varphi_n=(c_n-\eta)\varphi_n\le0\quad\text{in }(s_0,\infty).
		\end{align*}
		Let $\Psi_\nu(s):=K_\nu(e^{s-\ell_n})$. 
		Since $\Psi_\nu$ solves $-\Psi''+(\nu^2+e^{2(s-\ell_n)})\Psi=0$, the function $\tilde w:=\frac{\varphi_n(s_0)}{\Psi_\nu(s_0)}\Psi_\nu-\varphi_n$ satisfies $-\tilde w''+(\nu^2+e^{2(s-\ell_n)})\tilde w\ge0$ in $(s_0,\infty)$, $\tilde w(s_0)=0$ and $\tilde w(s)\to0$ as $s\to\infty$. Hence $\tilde w\ge0$ by Lemma~\ref{b0L:comp}, that is,
		\begin{align*}
			\varphi_n\le\frac{\varphi_n(s_0)}{\Psi_\nu(s_0)}\Psi_\nu\quad\text{on }[s_0,\infty).
		\end{align*}
		Now \eqref{b0Besup} and \eqref{b0s0} yield
		\begin{align}\label{b0upb}
			\varphi_n(s)\le C_6\,e^{-\nu(s-s_0)}\,G_\nu\big(e^{s-\ell_n}\big)\quad\text{for }s\ge s_0,\qquad C_6:=\frac{C_0^{1/2}}{e^{-\nu}K_\nu(e^{-1})}.
		\end{align}
		
		\noindent\textbf{Step 6. Estimate of $B_n$.} By \eqref{b0idN}, $B_n/\SN=I_1+I_2$, where
		\begin{align*}
			I_1:=\int_{-\infty}^{s_0}e^{\kappa s}\varphi_n^q\ds,\qquad I_2:=\int_{s_0}^\infty e^{\kappa s}\varphi_n^q\ds.
		\end{align*}
		By \eqref{b0Linf} and \eqref{b0upb},
		\begin{align*}
			I_1\le\frac{C_0^{q/2}}\kappa e^{\kappa s_0},\qquad I_2\le C_6^q\,e^{q\nu s_0}\int_{s_0}^\infty e^{(\kappa-q\nu)s}G_\nu\big(e^{s-\ell_n}\big)^q\ds.
		\end{align*}
		In Case~1, $\kappa-q\nu<0$, and $I_2\le Ce^{\kappa s_0}$ follows from $G_\nu\le g_\nu$. Therefore
		\begin{align}\label{b0BI}
			B_n\le C_7\,e^{\kappa s_0}.
		\end{align}
		In Case~2, $\kappa-q\nu>0$, so we split the last integral at $\ell_n$: on $[s_0,\ell_n]$ we use $G_\nu\le g_\nu$, and on $[\ell_n,\infty)$ we put $s=\ell_n+\tau$ and use the bound $G_\nu(e^\tau)\le c_\nu e^{(\nu-\frac12)\tau}e^{-e^\tau}$ of Lemma~\ref{b0L:Bes}\textit{(ii)}. This leads to
		\begin{align*}
			\int_{s_0}^\infty e^{(\kappa-q\nu)s}G_\nu\big(e^{s-\ell_n}\big)^q\ds\le e^{(\kappa-q\nu)\ell_n}\Big(\frac{g_\nu^q}{\kappa-q\nu}+c_\nu^q\int_0^\infty e^{(\kappa-q/2)\tau}e^{-qe^\tau}d\tau\Big)=Ce^{(\kappa-q\nu)\ell_n}.
		\end{align*}
		For $I_1$, the inequality $s_0\le\ell_n$ implies
		\begin{align*}
			e^{\kappa s_0}=e^{q\nu s_0}e^{(\kappa-q\nu)s_0}\le e^{q\nu s_0}e^{(\kappa-q\nu)\ell_n}.
		\end{align*}
		Therefore
		\begin{align}\label{b0BII}
			B_n\le C_8\,e^{q\nu s_0}e^{(\kappa-q\nu)\ell_n}.
		\end{align}
		
		\noindent\textbf{Step 7: Conclusion.} We combine the bounds for $B_n$ with the lower bound for $|\lambda_n|$ from Step~5. By \eqref{b0Poh},
		\begin{align}\label{b0betan}
			\beta_n=\frac{q\rho^2}{\kappa}\frac{|\lambda_n|}{B_n},
		\end{align}
		and $s_0=y_n+\sigma_*\to-\infty$ by Step~2.
		
		In Case~1, \eqref{b0ellb}, \eqref{b0BI} and \eqref{b0betan} yield
		\begin{align*}
			\beta_n\ge c\,|\lambda_n|e^{-\kappa s_0}\ge c\,\exp\Big(s_0\,\frac{2\gamma_a-\kappa(1-\gamma_a)}{1-\gamma_a}\Big)\ge c
		\end{align*}
		for $n$ large, because $2\gamma_a-\kappa(1-\gamma_a)\le0$ and $s_0\to-\infty$.
		
		In Case~2, \eqref{b0BII} and \eqref{b0betan} imply
		\begin{align*}
			\beta_n\ge c\,e^{-q\nu s_0}e^{-(2+\kappa-q\nu)\ell_n}.
		\end{align*}
		By \eqref{b0ellb}, $e^{-\ell_n}\ge c\,e^{\gamma_as_0/(1-\gamma_a)}$. Since $2+\kappa-q\nu>0$, we obtain
		\begin{align*}
			\beta_n\ge c\,\exp\Big(-q\nu s_0+\frac{(2+\kappa-q\nu)\gamma_a}{1-\gamma_a}s_0\Big)=c\,\exp\Big(s_0\,\frac{\gamma_a(2+\kappa)-q\nu}{1-\gamma_a}\Big)\to\infty,
		\end{align*}
		because $\gamma_a(2+\kappa)-q\nu<0$ by the choice of $\nu$.
		
		In both cases this contradicts $\beta_n\to0$, and therefore $\beta_0>0$.\QED
	\end{proof}
	
	\begin{rem}\label{b0R1}
		The hypotheses of Proposition~\ref{Pb0sup} are used as follows.
		\begin{itemize}
			\item The conditions $q_c<\frac ND$ and $q\le q^*$ enter only through \eqref{b0H}: they give $\gamma_a<1$, and $f(q)\ge0$ in Case~1, which is used in Step~7.
			\item The condition $\gamma_a<1$, that is, $a>\frac{N-4}{2}$, enters in Step~2 and in \eqref{b0ellb}. There the tail of the profile carries the mass.
			\item The condition $q>q_c$ enters in Step~1, through $q\delta_q>2$, and in the choice of $\nu$ in Case~2, through $2+\kappa<q$. Estimate \eqref{b0Wsmall} needs only $q\ge q_c$, that is, $\kappa\le q-2$. The case $q=q_c$ is treated in Proposition~\ref{Pb0crit}.
			\item The two cases of Step~4 correspond to the position of the $q$-norm of the profile: it is carried by the core in Case~1 ($q>\frac ND$) and by the tail in Case~2 ($q\le\frac ND$).
		\end{itemize}
	\end{rem}
	\subsection{The mass-critical case}\label{ss:b0crit}
	
	\begin{prop}\label{Pb0crit}
		Let $q=q_c<\frac ND$. Then $\beta_0>0$.
	\end{prop}
	
	\begin{proof}
		By \eqref{b0rel2} and Lemma~\ref{b0L:alg}\textit{(i)},
		\begin{align}\label{b1alg}
			\kappa=q-2,\qquad\kappa-q\gamma_a>0,\qquad0<\gamma_a<1,
		\end{align}
		and $\beta_0<\beta^*$ by Lemma~\ref{Lbeta0}\textit{(iii)}.
		
		Suppose, by contradiction, that $\beta_0=0$. We follow the proof of Proposition~\ref{Pb0sup} and indicate only the changes in Steps~1--5. The main new ingredient is Step~6, which provides the decay of $\varphi_n$ with the exact rate $\gamma_a$. Steps~7 and~8 replace Steps~6 and~7 of the proof of Proposition~\ref{Pb0sup}.
		
		\noindent\textbf{Steps 1--5.} Let $\beta_n\in(0,\frac{\beta^*}2)$ with $\beta_n\to0$, and let $u_n$ be a ground state given by Theorem~\ref{Tcrit}\textit{(i)}, normalized as in the proof of Proposition~\ref{Pb0sup}, so that \eqref{b0Poh} holds. Estimate \eqref{b0mlow} holds for $0<\beta<\beta^*$, with the same proof, but the rest of Step~1 must be changed. Now $q\delta_q=2$, and therefore every $u\in\M$ satisfies
		\begin{align*}
			\e(u)=\e(u)-\frac12P_\beta(u)=\frac dNC(u).
		\end{align*}
		Consequently, $\mathcal T-C_1\beta_n\le\frac dNC(u_n)\le\mathcal T$ by \eqref{b0mlow} and Lemma~\ref{Lmm}\textit{(ii)}, and $C(u_n)\to\sr^{\frac N{2d}}$. To compare $A_n$ with $C(u_n)$, we use $\beta^*$. Since $(1-\delta_q)q=q-2=\frac{4d}{N-2a+2b}$, \eqref{c1} and the definition of $\beta^*$ give $\beta_n\delta_qB_n\le\frac{\beta_n}{\beta^*}A_n$, and the identity $P_{\beta_n}(u_n)=0$ then yields
		\begin{align*}
			C(u_n)\le A_n\le\Big(1-\frac{\beta_n}{\beta^*}\Big)^{-1}C(u_n).
		\end{align*}
		Thus \eqref{b0lim} holds. Step~2 uses only \eqref{b0lim} and $\gamma_a<1$. Therefore \eqref{b0odeN}--\eqref{b0prof} hold, and $y_n\to-\infty$. In Step~3, $\kappa-q+2=0$ by \eqref{b1alg}, so \eqref{b0Wsmall} remains valid. In Step~4 we take $\nu:=\frac{\gamma_a}2$ and $\eta:=\gamma_a^2-\nu^2$. Then \eqref{b0s0}--\eqref{b0odeN2} hold, and Step~5 yields \eqref{b0lowb}, \eqref{b0ellb} and \eqref{b0upb}. We also need the upper bound from a general starting point: the comparison argument for the upper bound in Step~5, started at a point $\sigma\ge s_0$ instead of $s_0$, gives
		\begin{align}\label{b1compsig}
			\varphi_n(s)\le\varphi_n(\sigma)\frac{\Psi_\nu(s)}{\Psi_\nu(\sigma)}\qquad(s\ge\sigma\ge s_0),\qquad\Psi_\nu(s)=K_\nu(e^{s-\ell_n}).
		\end{align}
		Therefore, by Lemma~\ref{b0L:Bes}\textit{(iv)},
		\begin{align}\label{b1decnu}
			\varphi_n(t)\le\hat C_\nu e^{-\nu(t-s)}\varphi_n(s)\qquad(s_0\le s\le\ell_n-1,\ t\ge s).
		\end{align}
		Finally, we fix the constants for Step~6. Let $C_9:=1+2J_\nu K_\nu(1)^{-2}$, where $J_\nu<\infty$ is given by Lemma~\ref{b0L:Bes}\textit{(v)}, and fix $L\ge1$ such that
		\begin{align}\label{b1theta}
			\theta:=\big(\gamma_a^2+C_9e^{-2\nu L}\big)^{1/2}<1;
		\end{align}
		this is possible because $\gamma_a<1$. Enlarging $n_0$, we may assume that $s_0\le\ell_n-L$ for $n\ge n_0$.
		
		\noindent\textbf{Step 6. Decay with the exact rate $\gamma_a$.} 
		We drop the index $n$ where no confusion can arise and write $\ell=\ell_n$. The proof has three parts: an upper bound for the function $\upsilon$ defined in \eqref{b1ups} below, the integrability of $W_n$, and a lower bound for $\upsilon$. By \eqref{b0odeN2} and \eqref{b0cn},
		\begin{align*}
			\varphi''\ge(\gamma_a^2-\eta)\varphi=\nu^2\varphi>0\quad\text{on }[s_0,\infty).
		\end{align*}
		Thus $\varphi$ is convex on $[s_0,\infty)$ and tends to $0$ at infinity, so that $\varphi'<0$ there and $\varphi'(s)\to0$ as $s\to\infty$. Let
		\begin{align}\label{b1ups}
			\upsilon:=-\frac{\varphi'}{\varphi}>0,\qquad H:=\varphi'^2-\gamma_a^2\varphi^2\qquad\text{on }[s_0,\infty).
		\end{align}
		Then $\upsilon^2=\gamma_a^2+H/\varphi^2$, and $H(s)\to0$ as $s\to\infty$. By \eqref{b0odeN2},
		\begin{align*}
			H'=2\varphi'(\varphi''-\gamma_a^2\varphi)=2\varphi\varphi'\big(e^{2(s-\ell)}-c_n\big),
		\end{align*}
		and therefore
		\begin{align}\label{b1Hrep}
			H(s)=\int_s^\infty\big(-(\varphi^2)'(t)\big)\big(e^{2(t-\ell)}-c_n(t)\big)\,dt,\qquad s\ge s_0.
		\end{align}
		Here $-(\varphi^2)'=2\upsilon\varphi^2>0$, so, by \eqref{b0cn}, the negative part of the integrand has integral at most $\eta\varphi(s)^2$.
		
		We first bound $\upsilon$ from above. Let $s_0\le s\le\ell-L$ and $z_s:=e^{s-\ell}\le e^{-L}$. Since $c_n\ge0$, \eqref{b1Hrep} and an integration by parts give
		\begin{align*}
			H(s)\le\int_s^\infty\big(-(\varphi^2)'\big)e^{2(t-\ell)}dt=\varphi(s)^2e^{2(s-\ell)}+2\int_s^\infty\varphi(t)^2e^{2(t-\ell)}dt.
		\end{align*}
		Here, the boundary term at infinity vanishes by \eqref{b1compsig} and (K4). For the last integral, by \eqref{b1compsig}, the substitution $z=e^{t-\ell}$ and Lemma~\ref{b0L:Bes}\textit{(v)},
		\begin{align*}
			2\int_s^\infty\varphi(t)^2e^{2(t-\ell)}dt\le\frac{2\varphi(s)^2}{K_\nu(z_s)^2}\int_{z_s}^\infty zK_\nu(z)^2dz\le\frac{2J_\nu}{K_\nu(1)^2}z_s^{2\nu}\varphi(s)^2.
		\end{align*}
		Using $e^{2(s-\ell)}\le e^{-2\nu L}$ and $z_s^{2\nu}\le e^{-2\nu L}$, we obtain
		\begin{align*}
			H(s)\le C_9e^{-2\nu L}\varphi(s)^2,
		\end{align*}
		and \eqref{b1theta} yields
		\begin{align}\label{b1ubound}
			0<\upsilon(s)\le\theta<1\qquad\text{for }s_0\le s\le\ell-L.
		\end{align}
		
		Next we show that $W_n$ is integrable, with a bound of order $\beta_n$. For $s_0\le s\le t\le\ell-L$, \eqref{b1ubound} implies
		\begin{align*}
			\frac{\varphi(s)}{\varphi(t)}=\exp\int_s^t\upsilon\le e^{\theta(t-s)}.
		\end{align*}
		Since $\kappa=q-2$, it follows that
		\begin{align*}
			W_n(s)=W_n(t)\,e^{\kappa(s-t)}\Big(\frac{\varphi(s)}{\varphi(t)}\Big)^{q-2}\le W_n(t)\,e^{-(q-2)(1-\theta)(t-s)}.
		\end{align*}
		Taking $t=\ell-L$ and using \eqref{b0Wsmall}, we obtain
		\begin{align*}
			\int_{s_0}^{\ell-L}W_n\ds\le\frac{C_3\beta_n}{(q-2)(1-\theta)}.
		\end{align*}
		In addition, $\int_{\ell-L}^\ell W_n\ds\le LC_3\beta_n$. By \eqref{b1compsig} with $\sigma=\ell$, and by (K4),
		\begin{align*}
			\int_\ell^\infty W_n\ds\le W_n(\ell)\int_0^\infty e^{\kappa\tau}\Big(\frac{K_\nu(e^\tau)}{K_\nu(1)}\Big)^{q-2}d\tau\le C\beta_n.
		\end{align*}
		Adding the three bounds, we obtain
		\begin{align}\label{b1Wint}
			\int_{s_0}^\infty W_n\ds\le C_{10}\beta_n.
		\end{align}
		Furthermore, by \eqref{b0s0} and \eqref{b1decnu},
		\begin{align*}
			\int_{s_0}^\infty\varphi^{p-2}\ds\le\frac{(\hat C_\nu C_0^{1/2})^{p-2}}{(p-2)\nu}.
		\end{align*}
		Combined with \eqref{b1Wint}, this yields
		\begin{align}\label{b1cint}
			\int_{s_0}^\infty c_n\ds\le C_{11}.
		\end{align}
		
		We now bound $\upsilon$ from below. By \eqref{b1Hrep}, $H(s)\ge-\varphi(s)^2\chi(s)$, where
		\begin{align*}
			\chi(s):=\varphi(s)^{-2}\int_s^\infty\big(-(\varphi^2)'\big)c_n\,dt\ge0.
		\end{align*}
		Hence $\upsilon^2\ge\gamma_a^2-\chi$, and therefore
		\begin{align*}
			\upsilon\ge\gamma_a-\frac{\chi}{\gamma_a}.
		\end{align*}
		Indeed, if $\chi\le\gamma_a^2$, this follows from $(\gamma_a-\chi/\gamma_a)^2\le\gamma_a^2-\chi$. Otherwise, the right-hand side is negative. To estimate $\chi(s)$ for $s_0\le s\le\ell-L$, we split the integral at $\ell-L$: on $[s,\ell-L]$ we use $-(\varphi^2)'=2\upsilon\varphi^2\le2\theta\varphi^2$ and \eqref{b1decnu}, and on $[\ell-L,\infty)$ we use $c_n\le\eta$ and
		\begin{align*}
			\int_{\ell-L}^\infty\big(-(\varphi^2)'\big)dt=\varphi(\ell-L)^2\le\hat C_\nu^2e^{-2\nu(\ell-L-s)}\varphi(s)^2.
		\end{align*}
		This yields
		\begin{align*}
			\chi(s)\le2\theta\hat C_\nu^2\int_s^{\ell-L}e^{-2\nu(t-s)}c_n(t)\,dt+\eta\hat C_\nu^2e^{-2\nu(\ell-L-s)}.
		\end{align*}
		Integrating over $s\in[s_0,\ell-L]$ and using Fubini's theorem and \eqref{b1cint}, we obtain
		\begin{align*}
			\int_{s_0}^{\ell-L}\chi\ds\le\frac{\theta\hat C_\nu^2}{\nu}C_{11}+\frac{\eta\hat C_\nu^2}{2\nu}=:C_{12}.
		\end{align*}
		Since $(\ln\varphi)'=-\upsilon\le-\gamma_a+\chi/\gamma_a$, integration from $s_0$ to $s$ and \eqref{b0s0} yield
		\begin{align}\label{b1exact}
			\varphi_n(s)\le C_{13}e^{-\gamma_a(s-s_0)}\quad\text{for }s_0\le s\le\ell_n-L,\qquad C_{13}:=C_0^{1/2}e^{C_{12}/\gamma_a}.
		\end{align}
		
		\noindent\textbf{Step 7. An upper bound for $B_n$.} By \eqref{b0idN}, $B_n/\SN=I_1+I_2+I_3$, where $I_1$, $I_2$ and $I_3$ are the integrals of $e^{\kappa s}\varphi_n^q$ over $(-\infty,s_0)$, $[s_0,\ell_n-L]$ and $(\ell_n-L,\infty)$. Using \eqref{b0Linf}, $s_0\le\ell_n$ and $\kappa>q\gamma_a$, we obtain
		\begin{align*}
			I_1\le\frac{C_0^{q/2}}{\kappa}e^{\kappa s_0}\le\frac{C_0^{q/2}}{\kappa}e^{q\gamma_a s_0+(\kappa-q\gamma_a)\ell_n}.
		\end{align*}
		By \eqref{b1exact},
		\begin{align*}
			I_2\le C_{13}^qe^{q\gamma_a s_0}\int_{s_0}^{\ell_n-L}e^{(\kappa-q\gamma_a)s}\ds\le\frac{C_{13}^q}{\kappa-q\gamma_a}\,e^{q\gamma_a s_0+(\kappa-q\gamma_a)\ell_n}.
		\end{align*}
		For $t\ge\ell_n-L$, \eqref{b1compsig} with $\sigma=\ell_n-L$ and \eqref{b1exact} give
		\begin{align*}
			\varphi_n(t)\le C_{13}e^{-\gamma_a(\ell_n-L-s_0)}\frac{K_\nu(e^{t-\ell_n})}{K_\nu(e^{-L})}.
		\end{align*}
		It follows that
		\begin{align*}
			I_3\le C_{13}^qe^{q\gamma_a L}e^{q\gamma_a s_0+(\kappa-q\gamma_a)\ell_n}\int_{-L}^\infty e^{\kappa\tau}\Big(\frac{K_\nu(e^\tau)}{K_\nu(e^{-L})}\Big)^qd\tau,
		\end{align*}
		where the last integral is finite by (K4). Therefore
		\begin{align}\label{b1Bup}
			B_n\le C_{14}\,e^{q\gamma_a s_0+(\kappa-q\gamma_a)\ell_n}.
		\end{align}
		
		\noindent\textbf{Step 8. Conclusion.} By \eqref{b0Poh}, $|\lambda_n|=e^{-2\ell_n}$ and \eqref{b1Bup},
		\begin{align*}
			\beta_n=\frac{q\rho^2}{\kappa}\,\frac{e^{-2\ell_n}}{B_n}\ge\frac{q\rho^2}{\kappa C_{14}}\exp\big(-q\gamma_a s_0-(2+\kappa-q\gamma_a)\ell_n\big).
		\end{align*}
		By \eqref{b1alg}, $2+\kappa-q\gamma_a=q(1-\gamma_a)$, while \eqref{b0ellb} gives $e^{\gamma_as_0+(1-\gamma_a)\ell_n}\le C_5$. Hence
		\begin{align*}
			\beta_n\ge\frac{q\rho^2}{\kappa C_{14}}\exp\big(-q[\gamma_a s_0+(1-\gamma_a)\ell_n]\big)\ge\frac{q\rho^2}{\kappa C_{14}}\,C_5^{-q}>0,
		\end{align*}
		which contradicts $\beta_n\to0$. Thus $\beta_0>0$.\QED
	\end{proof}
	
	\begin{rem}[The case $q=q_c$]
		In Case~2 of the proof of Proposition~\ref{Pb0sup}, the upper bound \eqref{b0upb} has a rate $\nu<\gamma_a$, and Step~7 there requires $\nu>\frac{\gamma_a(2+\kappa)}q$. At $q=q_c$ we have $2+\kappa=q$, so no such $\nu$ exists. Step~6 removes this loss by giving the upper bound with the exact rate $\gamma_a$. Its only new input is \eqref{b1Wint}, which holds because $\kappa=q-2$ and $\theta<1$. The strict inequality $\kappa>q\gamma_a$, that is, $q_c<\frac ND$, is used only in Step~7. If $q_c=\frac ND$, then $\kappa=q\gamma_a$ and $I_2$ acquires the factor $\ell_n-s_0$, so that Step~8 gives only $\beta_n\ge c/(\ell_n-s_0)$, which tends to $0$. This is consistent with $\beta_0=0$ in that case.
	\end{rem}
    \subsection{ Proofs of Theorems~\ref{Tbeta0} and \ref{Tasym}}
	\noindent\textbf{Proof of Theorem~\ref{Tbeta0}.} \textit{(i)} If $q_c\ge\frac ND$ or $q>q^*$, then $\beta_0=0$ by Lemma~\ref{Lzero}. In particular, if $q_c\ge\frac ND$, then $q^*=q_c$, and $\beta_0=0$ for every $q\in[q_c,\t)$. If $q_c<\frac ND$ and $q\le q^*$, then $\beta_0>0$ by Proposition~\ref{Pb0sup} for $q_c<q\le q^*$ and by Proposition~\ref{Pb0crit} for $q=q_c$. The last statement of \textit{(i)} follows from Remark~\ref{b0Rrange}\textit{(ii)}.
    
	\noi\textit{(ii)} Let $\beta_0>0$. By Lemma~\ref{Lsol}\textit{(ii)}, $m_\beta(\rho)=\mathcal T$ for $0<\beta\le\beta_0$, and this level is not attained for $0<\beta<\beta_0$. Lemma~\ref{Lsol}\textit{(i)} then shows that every solution of \eqref{eq:P} in $S^{\mathcal Y}_\rho$ has energy larger than $\mathcal T$.\QED
	
	\textbf{Proof of Theorem~\ref{Tasym}.}	Let $\beta_n$ and $u_n$ be as in Theorem~\ref{Tasym}, and 
	write $A_n:=A(u_n)$ and $B_n:=B(u_n)$.  By Lemma~\ref{B2poh}, $u_n\in\mathcal M_{\rho,\beta_n}$, and $\mathcal E_{\beta_n}(u_n)=m_{\beta_n}(\rho)$ because the ground states given by Theorems~\ref{Tcrit} and~\ref{Tsup} have energy $m_{\beta_n}(\rho)$. In addition, $u_n>0$ by assumption, $u_n$ is radial by Proposition~\ref{Psym}, and $u_n\in C^2(\R^N\setminus\{0\})$ by Lemma~\ref{B1reg}.   The assumptions $q_c<\frac ND$ and $q\le q^*$ are not used in Step~1 of the proof of Proposition~\ref{Pb0sup} (for $q>q_c$), nor in its modification in the proof of Proposition~\ref{Pb0crit} (for $q=q_c$).  Hence \eqref{b0lim} holds. Together with $\mathcal T-C_1\beta_n\le m_{\beta_n}(\rho)\le\mathcal T$, which follows from \eqref{b0mlow} and Lemma~\ref{Lmm}\textit{(ii)}, this yields the four limits in Theorem~\ref{Tasym}.
	
	Next, we identify the limit profile. Let $\varphi_n(s):=e^{\gamma_a s}u_n(e^s)$. As in Step~2 of the proof of Proposition~\ref{Pb0sup}, Lemma~\ref{b0L:CC} provides $y_n\in\R$ and a positive solution $\psi\in H^1(\R)$ of \eqref{b0psiode} satisfying \eqref{b0psilow},  with $\nr{\psi}^2=\int_\R\psi^p\ds=\sr^{\frac N{2d}}/\SN$, such that, along a subsequence, $\varphi_n(\cdot+y_n)\to\psi$ in $H^1(\R)$ and in $L^\infty(\R)$. Set
	\begin{align*}
		\sigma_n:=e^{y_n},\qquad w_n(x):=\sigma_n^{\gamma_a}u_n(\sigma_nx),\qquad U(x):=|x|^{-\gamma_a}\psi(\ln|x|).
	\end{align*}
	Then $e^{\gamma_a s}w_n(e^s)=\varphi_n(s+y_n)$.  The functions $w_n-w_m$ and $w_n-U$ are radial and of class $C^1$ in $\R^N\setminus\{0\}$, and their Emden--Fowler transforms belong to $H^1(\R)$. By \eqref{b0EFgen},
	\begin{align*}
		\||x|^{-a}\na(w_n-w_m)\|_2^2&=\SN\nr{\varphi_n(\cdot+y_n)-\varphi_m(\cdot+y_m)}^2,\\
		\||x|^{-a}\na(w_n-U)\|_2^2&=\SN\nr{\varphi_n(\cdot+y_n)-\psi}^2.
	\end{align*}
	Hence $\{w_n\}$ is a Cauchy sequence in $\mathcal D^{1,2}_a(\R^N)$. Let $w_\infty$ be its limit. By the definition of $\sr$, a subsequence of $\{w_n\}$ converges to $w_\infty$ almost everywhere. On the other hand, $w_n\to U$ pointwise in $\R^N\setminus\{0\}$, because $\varphi_n(\cdot+y_n)\to\psi$ uniformly. Thus $w_\infty=U$, and $w_n\to U$ in $\mathcal D^{1,2}_a(\R^N)$. Furthermore, by \eqref{b0EFgen}, the substitution $r=e^s$ and $p(b+\gamma_a)=N$,
	\begin{align*}
		\||x|^{-a}\na U\|_2^2=\SN\nr{\psi}^2=\sr^{\frac N{2d}},\qquad\||x|^{-b}U\|_p^p=\SN\int_\R\psi^p\ds=\sr^{\frac N{2d}}.
	\end{align*}
	Consequently,
	\begin{align*}
		\frac{\||x|^{-a}\na U\|_2^2}{\||x|^{-b}U\|_p^{2}}=\sr^{\frac N{2d}(1-\frac2p)}=\sr,
	\end{align*}
	that is, $U$ attains $\sr$.
	
	It remains to study the scales $\sigma_n$. If $y_n\to\infty$ along a subsequence, then \eqref{b0idN} gives
	\begin{align*}
		\frac{\rho^2}{\SN}\ge e^{2y_n}\int_0^1e^{2\sigma}\varphi_n(\sigma+y_n)^2d\sigma\ge e^{2y_n}\Big(\int_0^1\psi^2d\sigma-o(1)\Big)\to\infty,
	\end{align*}
	which is impossible. Hence $\sup_n\sigma_n<\infty$. If $a>\frac{N-4}{2}$, then $\gamma_a<1$, and the argument of Step~2 of the proof of Proposition~\ref{Pb0sup}, which uses only $\gamma_a<1$, \eqref{b0idN}, \eqref{b0prof} and \eqref{b0psilow},  shows that $y_n\to-\infty$, that is, $\sigma_n\to0$.\QED
	
	\begin{rem}\label{RasymOpen}
		Since $M(w_n)=\sigma_n^{-2}\rho^2$, the scales $\sigma_n$ measure how much of the mass is kept by the profile. If $a\le\frac{N-4}{2}$, Theorem~\ref{Tasym} gives only $\sup_n\sigma_n<\infty$. In this case, we do not know whether $\sigma_n$ remains bounded away from $0$, whether its limit is determined, or whether part of the mass is lost in the limit. Note that $U\in L^2(\R^N;|x|^{-2a})$ if $a<\frac{N-4}{2}$, by Proposition~\ref{C11}\textit{(iv)}.
	\end{rem}
	
	\renewcommand{\thesection}{A}
	\section{Appendix}
	
	Appendix~\ref{AppA} collects the properties of solutions used throughout the paper: regularity, the Pohozaev identity, positivity, and the symmetry of ground states. Appendix~\ref{AppB} contains the estimates for the cut-off extremals of $\sr$ used in Sections~\ref{Sfib}--\ref{Sbeta0}.
	
	\subsection{Properties of solutions}\label{AppA}
	In this subsection, $c_0>0$ and $\mu_1,\mu_2,\mu_3\in\R$ are constants. We consider weak solutions $u\in\mathcal{Y}$ of
	\begin{align}\label{auxeq}
		-c_0\operatorname{div}(|x|^{-2a}\na u)=\mu_1|x|^{-2a}u+\mu_2|x|^{-bq}|u|^{q-2}u+\mu_3|x|^{-b\t}|u|^{\t-2}u\quad\text{in }\R^N.
	\end{align}
	We call \eqref{auxeq} the auxiliary equation. Problem \eqref{eq:P} is the special case $c_0=1$, $\mu_1=\lambda$, $\mu_2=\beta$, $\mu_3=1$. 
	As noted in Section~\ref{Sfib}, every critical point $u\in\mathcal X$ of $\e|_{\S}$ is a weak solution of \eqref{eq:P} in $\R^N$, so Lemmas~\ref{B1reg}--\ref{B3pos} apply to it.
	\begin{lemma}\label{B1reg}
		\begin{itemize}
			\item[(i)]Every weak solution $u\in\mathcal Y$ of \eqref{auxeq} belongs to $C^2(\R^N\setminus\{0\})$.
			\item[(ii)] If $u\in\mathcal X$ solves \eqref{eq:P} for some $\lambda\in\R$, then $u$ is bounded in a neighborhood of the origin and $\lim_{x\to0}u(x)$ exists.
		\end{itemize}
	\end{lemma}
	
	\begin{proof}\textit{(i)} On compact subsets of $\R^N\setminus\{0\}$ the weights are smooth and bounded from above and below, and the nonlinearities have Sobolev-subcritical growth, since $d<1$ implies $\t<\frac{2N}{N-2}$. The Brezis--Kato argument \cite{brezis1979remarks} and Schauder estimates therefore yield $u\in C^2(\R^N\setminus\{0\})$.
		
	\noindent\textit{(ii)} Let $u\in\mathcal X$ solve \eqref{eq:P}, and set $\varphi(s):=e^{\gamma_a s}u(e^s)$, with $\gamma_a$ and $\kappa$ as in \eqref{b0gk}. The proof of Lemma~\ref{b0L:EF} uses only $u\in\mathcal X\cap C^2(\R^N\setminus\{0\})$ and \eqref{eq:P}. Hence, by \textit{(i)}, that lemma applies here, although Section~\ref{Sbeta0} assumes $q\ge q_c$. Thus $\varphi\in C^2(\R)\cap H^1(\R)$, $\varphi(s)\to0$ as $s\to\pm\infty$, and
		\begin{align*}
			-\varphi''+(\gamma_a^2-\Upsilon)\varphi=0\quad\text{in }\R,\qquad\Upsilon(s):=\lambda e^{2s}+\beta e^{\kappa s}|\varphi(s)|^{q-2}+|\varphi(s)|^{\t-2}.
		\end{align*}
		Since $\kappa>0$, we have $\Upsilon(s)\to0$ as $s\to-\infty$.
		
		We first show that $\varphi$ decays exponentially at $-\infty$. Let $\epsilon_0:=\min\{2,\kappa,\frac{(\t-2)\gamma_a}{2}\}>0$, fix $\gamma_a'\in[\frac{\gamma_a}2,\gamma_a)$ with $\gamma_a-\gamma_a'<\epsilon_0$, and choose $s_1\le0$ such that $|\Upsilon|\le\gamma_a^2-\gamma_a'^2$ on $(-\infty,s_1]$. The functions $w_\pm(s):=|\varphi(s_1)|e^{\gamma_a'(s-s_1)}\mp\varphi(s)$ satisfy
		\begin{align*}
			-w_\pm''+(\gamma_a^2-\Upsilon)w_\pm=|\varphi(s_1)|\big(\gamma_a^2-\Upsilon-\gamma_a'^2\big)e^{\gamma_a'(s-s_1)}\ge0\quad\text{on }(-\infty,s_1),
		\end{align*}
		$w_\pm(s_1)\ge0$, and $w_\pm(s)\to0$ as $s\to-\infty$. Moreover, $\gamma_a^2-\Upsilon\ge\gamma_a'^2>0$ there. Lemma~\ref{b0L:comp}, applied to $s\mapsto w_\pm(-s)$ on $[-s_1,\infty)$, therefore gives $w_\pm\ge0$, that is, $|\varphi(s)|\le Ce^{\gamma_a's}$ for $s\le s_1$. As $\varphi$ is bounded, the definition of $\Upsilon$ then yields
		\begin{align*}
			|\Upsilon(s)|\le Ce^{\epsilon s}\quad(s\le s_1),\qquad\epsilon:=\min\{2,\kappa,(\t-2)\gamma_a'\}\ge\epsilon_0>\gamma_a-\gamma_a'.
		\end{align*}
		
		We now deduce the behavior of $u$ near the origin. Let $g(s):=u(e^s)=e^{-\gamma_a s}\varphi(s)$, so that $(e^{2\gamma_a s}g')'=-e^{2\gamma_a s}\Upsilon g$. Since $\varphi''=(\gamma_a^2-\Upsilon)\varphi$ is integrable on $(-\infty,s_1)$, $\varphi'$ is bounded there, and hence $e^{2\gamma_a s}g'=e^{\gamma_a s}(\varphi'-\gamma_a\varphi)\to0$ as $s\to-\infty$. Integrating, and using $|g(\sigma)|\le Ce^{(\gamma_a'-\gamma_a)\sigma}$, we obtain, for $s\le s_1$,
		\begin{align*}
			|g'(s)|\le\int_{-\infty}^se^{2\gamma_a(\sigma-s)}|\Upsilon(\sigma)||g(\sigma)|\,d\sigma\le C\int_{-\infty}^se^{2\gamma_a(\sigma-s)}e^{\delta\sigma}\,d\sigma=\frac{C\,e^{\delta s}}{2\gamma_a+\delta},
		\end{align*}
		where $\delta:=\epsilon+\gamma_a'-\gamma_a>0.$
		Thus $g'\in L^1(-\infty,s_1)$. Consequently, $g(s)=g(s_1)-\int_s^{s_1}g'\,d\sigma$ has a finite limit as $s\to-\infty$, and $g$ is bounded on $(-\infty,s_1]$. That is, $u$ is bounded on $B_{e^{s_1}}$, and $\lim_{x\to0}u(x)$ exists.\QED
	\end{proof}
	
	\begin{lemma}\label{B2poh}
		Every weak solution $u\in\mathcal Y$ of the auxiliary equation satisfies
		\begin{align*}
			c_0\frac{N-2-2a}{2}\||x|^{-a}\na u\|_2^2=\mu_1\frac{N-2a}{2}\||x|^{-a}u\|_2^2+\mu_2\frac{N-bq}{q}\||x|^{-b}u\|_q^q+\mu_3\frac{N-b\t}{\t}\||x|^{-b}u\|_\t^\t
		\end{align*}
		and
		\begin{align*}
			c_0\||x|^{-a}\na u\|_2^2=\mu_2\delta_q\||x|^{-b}u\|_q^q+\mu_3\||x|^{-b}u\|_\t^\t.
		\end{align*}
		In particular, every solution $u\in S^{\mathcal{Y}}_\rho$ of \eqref{eq:P} satisfies $P_\beta(u)=0$ and $\lambda\rho^2=\beta(\delta_q-1)\||x|^{-b}u\|_q^q$. Consequently, $\lambda<0$ if $u\ne0$.
	\end{lemma}
	\begin{proof}
		Let $\zeta\in C_c^\infty(B_2)$ be radial, with $0\le\zeta\le1$ and $\zeta=1$ on $B_1$, and set
		\begin{align*}
			\zeta_k(x):=\zeta(x/k)\big(1-\zeta(kx)\big),\qquad A_k:=\big\{\tfrac1k\le|x|\le\tfrac2k\big\}\cup\{k\le|x|\le2k\}.
		\end{align*}
		Then $\zeta_k$ has compact support in $\R^N\setminus\{0\}$, $\na\zeta_k$ vanishes outside $A_k$, and $|x||\na\zeta_k|\le C$. By Lemma~\ref{B1reg}\textit{(i)}, $\zeta_k\,x\cdot\na u$ is an admissible test function. We test the auxiliary equation with it and integrate by parts, using
		\begin{align*}
			\operatorname{div}(|x|^{-\sigma}x)=(N-\sigma)|x|^{-\sigma}\qquad\text{for }\sigma\in\{2a,bq,b\t\}.
		\end{align*}
		This gives the first identity with $\zeta_k$ inside each integral, up to remainder terms containing $\na\zeta_k$. These remainders are bounded by
		\begin{align*}
			C\int_{A_k}\Big(|x|^{-2a}\big(|\na u|^2+u^2\big)+|x|^{-bq}|u|^q+|x|^{-b\t}|u|^\t\Big)\dx,
		\end{align*}
		which tends to $0$ as $k\to\infty$, because the integrand belongs to $L^1(\R^N)$ by \eqref{c1} and the definition of $\sr$. Letting $k\to\infty$ yields the first identity.
		
		Next, testing the auxiliary equation with $u$ gives
		\begin{align*}
			c_0\||x|^{-a}\na u\|_2^2=\mu_1\||x|^{-a}u\|_2^2+\mu_2\||x|^{-b}u\|_q^q+\mu_3\||x|^{-b}u\|_\t^\t.
		\end{align*}
		We multiply this identity by $\frac{N-2a}{2}$ and subtract the first one. In view of
		\begin{align*}
			\frac{N-2a}{2}-\frac{N-bq}{q}=\delta_q,\qquad\frac{N-2a}{2}-\frac{N-b\t}{\t}=1,
		\end{align*}
		we obtain the second identity. For \eqref{eq:P}, the second identity reads $P_\beta(u)=0$. Inserting it into the identity obtained by testing with $u$, we find $\lambda\rho^2=\beta(\delta_q-1)\||x|^{-b}u\|_q^q$. Since $\delta_q<1$, it follows that $\lambda<0$ if $u\ne0$.\QED
	\end{proof}

	\begin{lemma}\label{B3pos}
		Let $u\in\mathcal X$, $u\ge0$, $u\not\equiv0$, be a weak solution of \eqref{eq:P} for some $\lambda\in\R$. Then $u\in C(\R^N)$ and $u>0$ in $\R^N$; in particular, $u(0)>0$.
	\end{lemma}
	\begin{proof}
		As $u$ is radial, we write $u=u(r)$ and set
		\begin{align*}
			F(r):=\lambda r^{-2a}u+\beta r^{-bq}u^{q-1}+r^{-b\t}u^{\t-1}.
		\end{align*}
		By Lemma~\ref{B1reg}, $u\in C^2(0,\infty)$ is bounded near $0$ and satisfies
		\begin{align*}
			-\big(r^{N-1-2a}u'\big)'=r^{N-1}F(r),\qquad r>0.
		\end{align*}
		
		We first show that $u>0$ on $(0,\infty)$. If $u(r_0)=0$ for some $r_0>0$, then $r_0$ is a minimum point of $u\ge0$, so that $u'(r_0)=0$. Since $q,\t>2$, the right-hand side of the equation is locally Lipschitz in $u$ for $r$ in compact subsets of $(0,\infty)$, and uniqueness for the initial value problem gives $u\equiv0$, a contradiction.
		
		Next, we show that $u$ is continuous at $0$. Since $\int_0^\infty r^{N-1-2a}|u'|^2\,dr<\infty$ and $N-1-2a>1$, there are $r_k\to0$ with $r_k^{N-1-2a}u'(r_k)\to0$. Moreover, $s^{N-1}F\in L^1(0,1)$, because $u$ is bounded near $0$ and $2a<bq<b\t<N$. Integrating the equation over $(r_k,r)$ and letting $k\to\infty$, we obtain
		\begin{align}\label{B3int}
			r^{N-1-2a}u'(r)=-\int_0^rs^{N-1}F(s)\,ds,\qquad r>0.
		\end{align}
		Consequently, $|u'(r)|\le Cr^{1+2a-b\t}$ on $(0,1)$, and $u'$ is integrable near $0$ because $b\t<2+2a$. Therefore $u(0):=\lim_{r\to0}u(r)$ exists, and $u\in C(\R^N)$.
		
		It remains to show that $u(0)>0$. If $\lambda\ge0$, then $F(r)\ge r^{-b\t}u(r)^{\t-1}>0$ for $r>0$, since $u>0$. By \eqref{B3int}, $u'<0$, so $u$ is decreasing and $u(0)\ge u(1)>0$. Let now $\lambda<0$, and assume, by contradiction, that $u(0)=0$. We have $s^{N-1}F(s)\ge-|\lambda|s^{N-1-2a}u(s)$, because $u\ge0$. With $\bar u(r):=\max_{[0,r]}u$, \eqref{B3int} yields
		\begin{align*}
			r^{N-1-2a}u'(r)\le|\lambda|\int_0^rs^{N-1-2a}u(s)\,ds\le\frac{|\lambda|}{N-2a}\,\bar u(r)\,r^{N-2a}.
		\end{align*}
		The function $\bar u$ is nondecreasing, so $u'(s)\le\frac{|\lambda|}{N-2a}\bar u(r)s$ for $0<s\le r$. Integrating from $0$ and using $u(0)=0$, we obtain $u(s)\le\frac{|\lambda|}{2(N-2a)}\bar u(r)s^2$ for $0\le s\le r$, and therefore
		\begin{align*}
			\bar u(r)\le\frac{|\lambda|}{2(N-2a)}\,r^2\,\bar u(r).
		\end{align*}
		For $r$ small, this forces $\bar u(r)=0$, which contradicts $u>0$ on $(0,\infty)$.\QED
	\end{proof}
	
	We now show that, up to sign, ground states are positive and radial (Proposition~\ref{Psym}). The proof rests on the substitution $v=|x|^{-a}u$ followed by Schwarz symmetrization of $v$. For extremals of the Caffarelli--Kohn--Nirenberg inequality, this argument is used in the proof of \cite[Lemma 2.1]{dolbeault2009symmetry}. We symmetrize $v$ rather than $u$, since Schwarz symmetrization of $u$ itself need not decrease the weighted gradient term; see \cite[p.~138]{chou1993best}.
	
	\begin{lemma}\label{B4sym}
		For $u\in\mathcal Y$ let $v:=|x|^{-a}u$, let $v^*$ be the Schwarz symmetrization of $|v|$, and set $u^\sharp:=|x|^av^*$.
		\begin{itemize}
			\item[(i)] $u^\sharp\in\mathcal X$, $\||x|^{-a}u^\sharp\|_2=\||x|^{-a}u\|_2$, $\||x|^{-a}\na u^\sharp\|_2\le\||x|^{-a}\na u\|_2$, $\||x|^{-b}u^\sharp\|_r\ge\||x|^{-b}u\|_r$ for $r\in\{q,\t\}$, and $(t\star u)^\sharp=t\star u^\sharp$ for every $t\in\R$. In particular, $\e(t\star u^\sharp)\le\e(t\star u)$ for every $t\in\R$.
			\item[(ii)] If $\||x|^{-b}u^\sharp\|_q=\||x|^{-b}u\|_q$, then $|u|=u^\sharp$; in particular, $|u|$ is radial and $|x|^{-a}|u|$ is radially nonincreasing.
		\end{itemize}
	\end{lemma}
	
	\begin{proof}
		We first relate the norms of $u$ and $v$ for smooth $u$ and then extend these relations by density. Set $\mu:=a(N-2-a)$. Then $0<\mu<\frac{(N-2)^2}{4}$, because $0<a<\frac{N-2}{2}$.
		
		We first prove \eqref{b4id} for smooth functions. Let $u\in C_c^\infty(\R^N\setminus\{0\})$ and $v:=|x|^{-a}u$. Then $\na u=|x|^a\big(\na v+a\,v\,\frac{x}{|x|^2}\big)$, so that
		\begin{align*}
			|x|^{-2a}|\na u|^2=|\na v|^2+a\,\frac{x}{|x|^2}\cdot\na(v^2)+a^2\frac{v^2}{|x|^2}.
		\end{align*}
		Since $\operatorname{div}\big(\frac{x}{|x|^2}\big)=\frac{N-2}{|x|^2}$ in $\R^N\setminus\{0\}$, an integration by parts gives
		\begin{align}\label{b4id}
			\int_{\R^N}|x|^{-2a}|\na u|^2\dx=\int_{\R^N}|\na v|^2\dx-\mu\int_{\R^N}\frac{v^2}{|x|^2}\dx,\qquad\int_{\R^N}|x|^{-2a}u^2\dx=\int_{\R^N}v^2\dx.
		\end{align}
		Furthermore, $|x|^{-br}|u|^r=|x|^{-(b-a)r}|v|^r$ for $r>0$. By Hardy's inequality and $1-\frac{4\mu}{(N-2)^2}=\big(\frac{N-2-2a}{N-2}\big)^2$,
		\begin{align}\label{b4two}
			\Big(\frac{N-2-2a}{N-2}\Big)^2\int_{\R^N}|\na v|^2\dx\le\int_{\R^N}|x|^{-2a}|\na u|^2\dx\le\int_{\R^N}|\na v|^2\dx.
		\end{align}
		Next, $C_c^\infty(\R^N\setminus\{0\})$ is dense in $\mathcal Y$. Indeed, a cut-off at infinity reduces the claim to functions with compact support. For such $u$ and $\zeta\in C_c^\infty(B_2)$ with $\zeta=1$ on $B_1$, we have $|\na\zeta(\cdot/\var)|\le C|x|^{-1}$, and the weighted Hardy inequality $\int_{\R^N}|x|^{-2a-2}u^2\dx\le\big(\frac{2}{N-2-2a}\big)^2\int_{\R^N}|x|^{-2a}|\na u|^2\dx$ gives $\int_{B_{2\var}}|x|^{-2a-2}u^2\dx\to0$. Hence $u(1-\zeta(\cdot/\var))\to u$ in $\mathcal Y$, and mollification completes the argument. Consequently, by \eqref{b4two}, the map $u\mapsto|x|^{-a}u$ extends to an isomorphism of $\mathcal Y$ onto $\mathcal D^{1,2}(\R^N)\cap L^2(\R^N)$, and \eqref{b4id} holds for every $u\in\mathcal Y$. 
		
		Let $u\in\mathcal Y$ and $v=|x|^{-a}u$. By \cite[Theorem 6.17]{lieb2001analysis}, $|\na|v||=|\na v|$ a.e., so the P\'olya--Szeg\H{o} inequality \cite[Lemma 7.17]{lieb2001analysis}, applied to $|v|$, yields
		\begin{align*}
			\int_{\R^N}|\na v^*|^2\dx\le\int_{\R^N}|\na v|^2\dx,\qquad\int_{\R^N}(v^*)^2\dx=\int_{\R^N}v^2\dx.
		\end{align*}
		For the weighted integrals, note that $|x|^{-2}$ is symmetric decreasing, and so is $|x|^{-(b-a)r}$ for $r\in\{q,\t\}$, since $b>a$. As $(|v|^r)^*=(v^*)^r$, the Hardy--Littlewood inequality \cite[Theorem 3.4]{lieb2001analysis} gives
		\begin{align*}
			\int_{\R^N}\frac{v^2}{|x|^2}\dx\le\int_{\R^N}\frac{(v^*)^2}{|x|^2}\dx,\qquad\int_{\R^N}\frac{|v|^r}{|x|^{(b-a)r}}\dx\le\int_{\R^N}\frac{(v^*)^r}{|x|^{(b-a)r}}\dx.
		\end{align*}
		These inequalities, \eqref{b4id} and $\mu>0$ imply the relations between the norms of $u^\sharp$ and $u$ in \textit{(i)}. Moreover, $u^\sharp\in\mathcal Y$ is radial, so $u^\sharp\in\mathcal X$. It remains to prove the claims on $t\star u$. We have
		\begin{align*}
			|x|^{-a}(t\star u)(x)=e^{Nt/2}v(e^tx),
		\end{align*}
		and Schwarz symmetrization commutes with $v\mapsto e^{Nt/2}v(e^t\,\cdot)$. Hence $(t\star u)^\sharp=t\star u^\sharp$. Applying the relations between the norms to $t\star u$ and using $\beta>0$, we obtain $\e(t\star u^\sharp)\le\e(t\star u)$. This proves \textit{(i)}.

		Let $f:=|x|^{-(b-a)q}$ and $g:=|v|^q$. Since $b>a$, $f$ is strictly symmetric decreasing, so $f^*=f$. In addition, $g^*=(v^*)^q$, and $g$ vanishes at infinity because $v\in L^2(\R^N)$. By \eqref{c1}, $\int_{\R^N}fg\dx=\||x|^{-b}u\|_q^q<\infty$. Since $u^\sharp=|x|^av^*$, the hypothesis of \textit{(ii)} reads
		\begin{align*}
			\int_{\R^N}fg\dx=\int_{\R^N}f^*g^*\dx.
		\end{align*}
		By the case of equality in \cite[Theorem 3.4]{lieb2001analysis}, $g=g^*$. Hence $|v|=v^*$, that is, $|u|=u^\sharp$.\QED
	\end{proof}
	
	\begin{prop}\label{Psym}
		Assume one of the following:
		\begin{itemize}
			\item[(i)] $2<q<q_c$ and $\beta<\beta_1$;
			\item[(ii)] $q=q_c$ and $\beta_0<\beta<\beta^*$;
			\item[(iii)] $q_c<q<\t$ and $\beta>\beta_0$.
		\end{itemize}
		Let $\tilde u\in\S$ be the positive radial solution of \eqref{eq:P} with $\e(\tilde u)=m_\beta(\rho)$ given by Theorem~\ref{C1}\textit{(i)} in case \textit{(i)}, by Theorem~\ref{Tcrit}\textit{(i)} in case \textit{(ii)}, and by Theorem~\ref{Tsup} in case \textit{(iii)}. Then $\tilde u$ is a ground state in $S^{\mathcal Y}_\rho$. Furthermore, every ground state $w\in S^{\mathcal Y}_\rho$ has constant sign, and $|w|\in\S$ is positive in $\R^N$, with $|x|^{-a}|w|$ radially nonincreasing.
	\end{prop}
	
	\begin{proof}
		The proofs of Lemmas~\ref{C12}--\ref{C33} use only \eqref{c1}, the definition of $\sr$ and the fiber maps $t\mapsto t\star u$, and not the radial symmetry. Consequently, these lemmas remain valid with $\S$ replaced by $S^{\mathcal Y}_\rho$. We also set
		\begin{align*}
			\mathcal G_{R_0}^{\mathcal Y}:=\{u\in S^{\mathcal Y}_\rho:\||x|^{-a}\na u\|_2<R_0\}.
		\end{align*}
		
		\noindent\textbf{Step 1.} \textbf{The levels.} We claim that
		\begin{align}\label{b4lev}
			m_\beta(\rho)=\inf_{\mathcal G_{R_0}^{\mathcal Y}}\e\quad\text{in case}~(i),\qquad m_\beta(\rho)=\inf_{u\in S^{\mathcal Y}_\rho}\max_{t\in\R}\e(t\star u)\quad\text{in cases}~(ii),~ (iii).
		\end{align}
		By Lemmas~\ref{C31} and~\ref{Lmm}\textit{(i)}, \eqref{b4lev} holds with $\S$ and $\mathcal G_{R_0}$ in place of $S^{\mathcal Y}_\rho$ and $\mathcal G_{R_0}^{\mathcal Y}$. Passing to the larger sets does not increase the infima, and it does not decrease them either: for $u\in S^{\mathcal Y}_\rho$, Lemma~\ref{B4sym}\textit{(i)} gives $u^\sharp\in\S$, $u^\sharp\in\mathcal G_{R_0}$ if $u\in\mathcal G_{R_0}^{\mathcal Y}$, and $\e(t\star u^\sharp)\le\e(t\star u)$ for every $t\in\R$. This proves the claim.
		
		\noindent\textbf{Step 2. {\boldmath${\tilde u}$} is a ground state.} Let $w\in S^{\mathcal Y}_\rho$ be a solution of \eqref{eq:P}. Then $P_\beta(w)=0$ by Lemma~\ref{B2poh}.
		
		\emph{Case (i).} By Lemma~\ref{C12}, $\mathcal M^0_{\rho,\beta}=\emptyset$, which leaves two cases. If $(\Phi^w_\beta)''(0)<0$, then Lemma~\ref{C12}\textit{(v)} gives $\e(w)>0>m_\beta(\rho)$. If $(\Phi^w_\beta)''(0)>0$, then $w\in\mathcal G_{R_0}^{\mathcal Y}$ by Lemma~\ref{C12}\textit{(v)}, and $\e(w)\ge m_\beta(\rho)$ by \eqref{b4lev}.
		
		\emph{Cases (ii) and (iii).} By Lemma~\ref{C10}, the map $t\mapsto\e(t\star w)$ has a unique critical point, namely its strict maximum point. Its derivative at $t=0$ is $P_\beta(w)=0$ (Remark~\ref{CA1}). Hence this point is $t=0$, and \eqref{b4lev} gives $\e(w)=\max_{t\in\R}\e(t\star w)\ge m_\beta(\rho)$.
		
		Thus $\e(w)\ge m_\beta(\rho)=\e(\tilde u)$ in all cases, and since $w$ is an arbitrary solution, $\tilde u$ is a ground state.
		
		\noindent\textbf{Step 3. {\boldmath${|w|=w^\sharp}$} for every ground state {\boldmath${w}$}.} Let $w\in S^{\mathcal Y}_\rho$ be a ground state, so that $\e(w)=m_\beta(\rho)$. We first construct a function $z$ with $\e(z^\sharp)=\e(z)$.
		
		\emph{Case (i).} By Lemma~\ref{C31}, $\e(w)=m_\beta(\rho)<0$, so $w\in\mathcal G_{R_0}^{\mathcal Y}$ by Step~2. We set $z:=w$. Then \eqref{b4lev} gives $m_\beta(\rho)\le\e(z^\sharp)\le\e(z)=m_\beta(\rho)$.
		
		\emph{Cases (ii) and (iii).} Let $t'$ be the maximum point of $t\mapsto\e(t\star w^\sharp)$, and $z:=t'\star w$. Then $z^\sharp=t'\star w^\sharp$, and
		\begin{align*}
			m_\beta(\rho)\le\e(z^\sharp)\le\e(z)\le\max_{t\in\R}\e(t\star w)=\e(w)=m_\beta(\rho).
		\end{align*}
		Thus $\e(z^\sharp)=\e(z)$ in all cases. On the other hand, when $z$ is replaced by $z^\sharp$, the gradient term of $\e$ does not increase, and the other two terms do not decrease, by Lemma~\ref{B4sym}\textit{(i)}. Together with $\beta>0$ and the equality of the two values of $\e$, this shows that each term is unchanged; in particular, $\||x|^{-b}z^\sharp\|_q=\||x|^{-b}z\|_q$. Lemma~\ref{B4sym}\textit{(ii)} then gives $|z|=z^\sharp$, and therefore $|w|=w^\sharp\in\mathcal X$.
		
		\noindent\textbf{Step 4. {\boldmath$|w|$} is a positive solution, and {\boldmath$w$} has constant sign.} By \cite[Theorem 6.17]{lieb2001analysis}, the mass and the three terms of $\e$ take the same values at $|w|$ and at $w$. Hence $\e(|w|)=m_\beta(\rho)$ and $P_\beta(|w|)=0$.
		
		\emph{Case (i).} By \eqref{b4lev}, $|w|$ minimizes $\e$ on $\mathcal G_{R_0}^{\mathcal Y}$. On the other hand, Lemma~\ref{C31} gives, for $r>0$ small,
		\begin{align*}
			m_\beta(\rho)<\inf\big\{\e(u):u\in S^{\mathcal Y}_\rho,\ R_0-r\le\||x|^{-a}\na u\|_2\le R_0\big\}.
		\end{align*}
		Thus $|w|$ lies in $\mathcal G_{R_0-r}^{\mathcal Y}$, which is open in $S^{\mathcal Y}_\rho$, so $|w|$ is a local minimizer of $\e$ on $S^{\mathcal Y}_\rho$.
		
		\emph{Cases (ii) and (iii).} For $u\in S^{\mathcal Y}_\rho$, let $t_u$ be the maximum point of $t\mapsto\e(t\star u)$, and $J(u):=\e(t_u\star u)$. By Lemma~\ref{C10}\textit{(iii)}, $J$ is of class $C^1$, and by \eqref{b4lev}, $J\ge m_\beta(\rho)=J(|w|)$, since $t_{|w|}=0$. Thus $|w|$ minimizes $J$ on $S^{\mathcal Y}_\rho$. Moreover, $\frac{d}{dt}\e(t\star u)|_{t=t_u}=0$, and the chain rule then gives $dJ(|w|)[\phi]=d\e(|w|)[\phi]$ for $\phi\in T_{|w|}S^{\mathcal Y}_\rho$.
		
		Hence, in both cases, $|w|$ is a critical point of $\e|_{S^{\mathcal Y}_\rho}$, that is, a nonnegative solution of \eqref{eq:P} for some $\lambda\in\R$. As $|w|\in\mathcal X$, Lemma~\ref{B3pos} yields $|w|>0$ in $\R^N$. It remains to show that $w$ has a constant sign. By Lemma~\ref{B1reg}, $w$ is continuous in $\R^N\setminus\{0\}$, which is connected, and $w$ does not vanish there. Hence $w=|w|$ or $w=-|w|$. Finally, $|x|^{-a}|w|=v^*$ with $v=|x|^{-a}w$, and $v^*$ is radially nonincreasing.\QED
	\end{proof}
	
	\subsection{Estimates for the cut-off extremals}\label{AppB}
	\begin{prop}\label{C11}
		Let $\zeta:\R^N\to[0,1]$ be a $C_c^{\infty}$ radial function such that $\zeta(x)=1$ for $|x|<R$ and $\zeta(x)=0$ for $|x|>2R$. Let
		\begin{equation}\label{c221}
			u_{\varepsilon}(x):=\zeta(x)U_{\varepsilon}(x),
		\end{equation}
		and let \(\phi\in L^{\infty}_{loc}(\R^N)\). Then the following estimates hold:
		\begin{itemize}
			\item[(i)] $\||x|^{-a}\na u_{\varepsilon}\|_2^2= \sr^{\frac{N}{2d}}+O(\varepsilon^{\frac{N-2d}{2d}})$,
			\item[(ii)] $\||x|^{-b}u_{\varepsilon}\|_{\t}^{\t}=\sr^{\frac{N}{2d}}+O(\varepsilon^{\frac{N}{2d}})$,
			\item[(iii)]
			\begin{equation*}
				\||x|^{-b}u_{\varepsilon}\|_q^q\asymp
				\begin{cases}
					\varepsilon^{\frac{q(N-2d)}{4d}},&\text{if}~~2<q<\frac ND,\\
					\varepsilon^{\frac{q(N-2d)}{4d}}|\log\var|,&\text{if}~~q=\frac ND,\\
					\varepsilon^{\frac{(2N-Nq+2qd)(N-2d)}{4d(N-2-2a)}},&\text{if}~\frac ND<q<\t,
				\end{cases}
			\end{equation*}
			\item[(iv)]
			\begin{equation*}
				\||x|^{-a}u_{\varepsilon}\|_2^2\asymp
				\begin{cases}
					\varepsilon^{\frac{N-2d}{d(N-2-2a)}},&\text{if}~~0<a<\max\{0,\frac{N-4}{2}\},\\
					\varepsilon^{\frac{N-2d}{d(N-2-2a)}}|\log\varepsilon|,&\text{if}~~0<a=\frac{N-4}{2},\\
					\varepsilon^{\frac{N-2d}{2d}},&\text{if}~~\max\{0,\frac{N-4}{2}\}<a<\frac{N-2}{2},
				\end{cases}
			\end{equation*}
			\item[(v)] \(\displaystyle\int_{\R^N}\frac{\phi u_{\var}^{\t-1}}{|x|^{\t b}}\dx=O(\var^{\frac{N-2d}{4d}})\),
			\item[(vi)] \(\displaystyle\int_{\R^N}\frac{\phi u_{\var}}{|x|^{2a}}\dx=O(\var^{\frac{N-2d}{4d}})\).
		\end{itemize}
		Here $f\asymp g$ means $c_1g\le f\le c_2g$ for constants $0<c_1\le c_2$ independent of $\var$. The two-sided bounds in \textit{(iii)} and \textit{(iv)} are used in Lemmas~\ref{Lzero} and~\ref{C72}.
	\end{prop}
	
	\begin{proof}
		We first fix notation. Set
		\begin{align*}
			k:=\frac{N-2d}{2d},\qquad c_U:=\big(2\cdot2^\sharp A_0\big)^{\frac{N-2d}{4d}}.
		\end{align*}
		Then $\alpha k=N-2-2a$ and, by \eqref{c96},
		\begin{equation}\label{Ueps}
			U_\var(x)=\frac{c_U\,\var^{k/2}}{(\var+|x|^\alpha)^{k}},\qquad U_\var(\var^{1/\alpha}y)=\var^{-k/2}\,U_1(y).
		\end{equation}
		In particular,
		\begin{align*}
			U_\var(x)\le c_U\,\var^{k/2}|x|^{-\alpha k}=c_U\,\var^{k/2}|x|^{-(N-2-2a)}.
		\end{align*}
		Moreover, $0\le u_\var\le U_\var$, because $0\le\zeta\le1$. The function $U_\var$ attains $\sr$ and satisfies $\||x|^{-a}\na U_\var\|_2^2=\||x|^{-b}U_\var\|_\t^\t$ \cite[Section 7]{Mishra}. The definition of $\sr$ and $1-\frac2\t=\frac{2d}{N}$ then give
		\begin{align*}
			\||x|^{-a}\na U_\var\|_2^2=\||x|^{-b}U_\var\|_\t^\t=\sr^{\frac N{2d}}.
		\end{align*}
		We prove \textit{(ii)} first, since the lower bound in \textit{(i)} relies on it.
		
	\noindent
		\textit{(ii)} From $0\le u_\var\le U_\var$ we obtain
		\begin{align*}
			\||x|^{-b}u_\var\|_\t^\t\le\||x|^{-b}U_\var\|_\t^\t=\sr^{\frac N{2d}}.
		\end{align*}
		Conversely, since $\zeta=1$ on $B_R$, the above bound on $U_\var$ gives
		\begin{align*}
			\||x|^{-b}u_\var\|_\t^\t&\ge\sr^{\frac{N}{2d}}-\int_{|x|>R}|x|^{-b\t}U_\var^{\t}\dx\\
			&\ge\sr^{\frac{N}{2d}}-c_U^{\t}\,\var^{\frac{\t k}{2}}\int_{|x|>R}|x|^{-b\t-\t(N-2-2a)}\dx.
		\end{align*}
		Here $\frac{\t k}{2}=\frac N{2d}$, and the last integral is finite because
		\begin{align*}
			b\t+\t(N-2-2a)=N+\frac{\t}{2}(N-2-2a)>N.
		\end{align*}
		This proves \textit{(ii)}.
		
	\noindent\textit{(i)} The upper bound
		\begin{align*}
			\||x|^{-a}\na u_\var\|_2^2\le\sr^{\frac{N}{2d}}+C\var^{\frac{N-2d}{2d}}
		\end{align*}
		is the first estimate of \cite[Proposition 7.1]{Mishra}, where the exponent $p$ is our $\t$ and the error is $O(\var^{\frac{2}{\t-2}})$, with $\frac{2}{\t-2}=\frac{N-2d}{2d}$.
		
		For the lower bound, the definition of $\sr$, \textit{(ii)}, the inequality $(1-s)^{2/\t}\ge1-s$ for $s\in[0,1]$, and $1+\frac N{2d}\cdot\frac2\t=\frac N{2d}$ give, for $\var$ small,
		\begin{align*}
			\||x|^{-a}\na u_\var\|_2^2\ge\sr\,\||x|^{-b}u_\var\|_\t^2\ge\sr\Big(\sr^{\frac N{2d}}-C\var^{\frac N{2d}}\Big)^{\frac2\t}\ge\sr^{\frac N{2d}}-C\var^{\frac N{2d}}.
		\end{align*}
		Since $\frac N{2d}>\frac{N-2d}{2d}$, this proves \textit{(i)}.
		
		\noindent The proofs of \textit{(iii)} and \textit{(iv)} reduce to the following integral. For $0\le\vartheta<N$, $\ell>0$ and $m>0$, consider
		\begin{align*}
			F_\var:=\int_{\R^N}\frac{|y|^{-\vartheta}\,\zeta(\var^{1/\alpha}y)^m}{(1+|y|^\alpha)^{\ell}}\dy,\qquad \tau:=\vartheta+\alpha\ell .
		\end{align*}
		The integrand behaves like $|y|^{-\vartheta}$ near the origin and like $|y|^{-\tau}$ at infinity. Let $0<\var<\min\{1,R^\alpha\}$ and $\varrho_i:=iR\var^{-1/\alpha}>1$, $i=1,2$. Since $\zeta=1$ on $B_R$ and $\zeta=0$ outside $B_{2R}$, we have
		\begin{align*}
			\int_{B_{\varrho_1}}g\dy\le F_\var\le\int_{B_{\varrho_2}}g\dy,\qquad\text{where } g(y):=\frac{|y|^{-\vartheta}}{(1+|y|^\alpha)^{\ell}}.
		\end{align*}
		We estimate $g$ separately on $B_1$ and outside $B_1$. On $B_1$ we have $1\le1+|y|^\alpha\le2$, so that, in spherical coordinates,
		\begin{align*}
			\frac{2^{-\ell}\SN}{N-\vartheta}\le I_0:=\int_{B_1}g\dy\le\frac{\SN}{N-\vartheta}.
		\end{align*}
		Outside $B_1$ we have $|y|^\alpha\le1+|y|^\alpha\le2|y|^\alpha$, so $2^{-\ell}|y|^{-\tau}\le g(y)\le|y|^{-\tau}$. Therefore
		\begin{equation}\label{Fsand}
			I_0+2^{-\ell}I(\varrho_1)\le F_\var\le I_0+I(\varrho_2),\qquad\text{where } I(\varrho):=\SN\int_1^{\varrho}r^{N-1-\tau}\,dr .
		\end{equation}
		
		\noindent\textbf{Case 1:} $\tau>N$. Then $0\le I(\varrho)\le\frac{\SN}{\tau-N}$ for every $\varrho>1$, and \eqref{Fsand} gives $F_\var\asymp1$.
		
		\noindent\textbf{Case 2:} $\tau=N$. Then
		\begin{align*}
I(\varrho_i)=\SN\log\varrho_i=\SN\Big(\frac1\alpha|\log\var|+\log(iR)\Big).
		\end{align*}
		Since $|\log\var|\to\infty$ as $\var\to0$, \eqref{Fsand} gives $F_\var\asymp|\log\var|$ for $\var$ small.
		
		\noindent\textbf{Case 3:} $\tau<N$. Then
		\begin{align*}
			I(\varrho_i)=\frac{\SN}{N-\tau}\Big((iR)^{N-\tau}\var^{-\frac{N-\tau}{\alpha}}-1\Big).
		\end{align*}
		Consequently, by \eqref{Fsand}, $F_\var\asymp\var^{-\frac{N-\tau}{\alpha}}$ for $\var$ small.
		
		\noindent In summary,
		\begin{equation}\label{Feps}
			F_\var\asymp1\ \text{ if }\tau>N,\qquad F_\var\asymp|\log\var|\ \text{ if }\tau=N,\qquad F_\var\asymp\var^{-\frac{N-\tau}{\alpha}}\ \text{ if }\tau<N.
		\end{equation}
		
		\noindent\textit{(iii)} By \eqref{Ueps} and the change of variables $x=\var^{1/\alpha}y$,
		\begin{align*}
			\||x|^{-b}u_\var\|_q^q=c_U^q\,\var^{e_q}F_\var,\qquad e_q:=\frac{N-bq}{\alpha}-\frac{kq}{2}=\frac{(2N-q(N-2d))(N-2d)}{4d(N-2-2a)},
		\end{align*}
		where $F_\var$ is the model integral with $\vartheta=bq<N$, $\ell=kq$ and $m=q$, so that $\tau=q(b+\alpha k)=qD$. Since $e_q-\frac{N-qD}{\alpha}=\frac{kq}{2}=\frac{q(N-2d)}{4d}$, \eqref{Feps} gives
		\begin{align*}
			\||x|^{-b}u_\var\|_q^q\asymp
			\begin{cases}
				\var^{\frac{q(N-2d)}{4d}}, & \text{if } q<\frac ND\ (\tau<N),\\[1mm]
				\var^{\frac{q(N-2d)}{4d}}|\log\var|, & \text{if } q=\frac ND\ (\tau=N),\\[1mm]
				\var^{e_q}, & \text{if } q>\frac ND\ (\tau>N).
			\end{cases}
		\end{align*}
		
		\noindent\textit{(iv)} In the same way, $\||x|^{-a}u_\var\|_2^2=c_U^2\,\var^{E_a}F_\var$, with
		\begin{align*}
			E_a:=\frac{N-2a}{\alpha}-k=\frac{N-2d}{d(N-2-2a)},
		\end{align*}
		where now $\vartheta=2a$, $\ell=2k$, $m=2$ and $\tau=2a+2\alpha k=2N-4-2a$. Thus $\tau>N$, $\tau=N$ or $\tau<N$ according as $a<\frac{N-4}{2}$, $a=\frac{N-4}{2}$ or $a>\frac{N-4}{2}$, and in the last case $E_a-\frac{N-\tau}{\alpha}=k=\frac{N-2d}{2d}$. Now \textit{(iv)} follows from \eqref{Feps}.
		
		\noindent\textit{(v)} Recall that $u_\var$ vanishes outside $B_{2R}$ and $0\le u_\var\le U_\var$. Then \eqref{Ueps} and the change of variables $x=\var^{1/\alpha}y$ give
		\begin{align*}
			\Big|\int_{\R^N}\frac{\phi\, u_\var^{\t-1}}{|x|^{\t b}}\dx\Big|\le\|\phi\|_{L^\infty(B_{2R})}\int_{\R^N}|x|^{-b\t}U_\var^{\t-1}\dx=\|\phi\|_{L^\infty(B_{2R})}\,\var^{\frac{N-2d}{4d}}\int_{\R^N}|y|^{-b\t}U_1^{\t-1}\dy,
		\end{align*}
		because $\frac{N-b\t}{\alpha}=\frac{N}{2d}=\frac{\t k}{2}$. The last integral is finite: since $U_1\le c_U\min\{1,|y|^{-\alpha k}\}$, the integrand is at most $c_U^{\t-1}|y|^{-b\t}$ in $B_1$, where $b\t<N$, and at most $c_U^{\t-1}|y|^{-N-\alpha}$ outside $B_1$, because $b+\frac{\alpha k}{2}=\frac{N-2d}{2}$ and $\frac{\t-2}{2}=\frac1k$ imply $b\t+(\t-1)\alpha k=N+\alpha$.
		
		\noindent\textit{(vi)} Since $u_\var$ vanishes outside $B_{2R}$ and $U_\var(x)\le c_U\var^{\frac{N-2d}{4d}}|x|^{-(N-2-2a)}$,
		\begin{align*}
			\Big|\int_{\R^N}\frac{\phi\, u_\var}{|x|^{2a}}\dx\Big|\le c_U\|\phi\|_{L^\infty(B_{2R})}\var^{\frac{N-2d}{4d}}\int_{B_{2R}}|x|^{-(N-2)}\dx=2c_U\SN R^2\|\phi\|_{L^\infty(B_{2R})}\var^{\frac{N-2d}{4d}}.
		\end{align*}
		This proves \textit{(vi)}.\QED
		
	\end{proof}
	
	\begin{prop}\label{C79} (i) If $\frac ND<q_c$, then
		\begin{equation*}
			\frac{\||x|^{-b}\ue\|_{q_c}^{q_c}}{\||x|^{-a}\ue\|_2^{\frac{4d}{N-2a+2b}}}\asymp\begin{cases}
				1,\quad&\text{if}~0<a<\max\{0,\frac{N-4}{2}\},\\
				|\log \var|^{-\frac{2d}{N-2a+2b}},&\text{if}~0<a=\frac{N-4}{2},\\
				\var^{\frac{(N-2d)(4+2a-N)}{(N-2-2a)(N-2a+2b)}},&\text{if}~\max\{0,\frac{N-4}{2}\}<a<\frac{N-2}{2}.
			\end{cases}
		\end{equation*}
		(ii) If $q_c\leq \frac ND$, then
		\begin{equation*}
			\frac{\||x|^{-b}\ue\|_{q_c}^{q_c}}{\||x|^{-a}\ue\|_2^{\frac{4d}{N-2a+2b}}}\asymp\begin{cases}
				\var^{\frac{N-2d}{2d}}|\log \var|,\quad &\text{if}~~q_c= \frac ND,\\
				\var^{\frac{N-2d}{2d}}, &\text{if}~~q_c<\frac ND.
			\end{cases}
		\end{equation*}
		Here $\asymp$ is as in Proposition~\ref{C11}.
	\end{prop}
	
	\begin{proof}
		Since $q_c-2=\frac{4d}{N-2a+2b}$, the quotient equals
		\begin{align*}
			Q_\var:=\||x|^{-b}u_\var\|_{q_c}^{q_c}\Big(\||x|^{-a}u_\var\|_2^2\Big)^{-\frac{q_c-2}{2}}.
		\end{align*}
		We estimate the two factors by Proposition~\ref{C11}\textit{(iii)},~\textit{(iv)}, with $e_q$ and $E_a$ from its proof.
		
		\noindent\textbf{Case 1:} $q_c\le\frac ND$. By Remark~\ref{b0Rrange}\textit{(ii)}, $a>\frac{N^2-8}{2(N+2)}>\frac{N-4}{2}$, and Proposition~\ref{C11}\textit{(iii)},~\textit{(iv)} give
		\begin{align*}
			\||x|^{-b}u_\var\|_{q_c}^{q_c}\asymp\var^{\frac{N-2d}{4d}q_c}\quad(\text{times }|\log\var|\text{ if }q_c=\tfrac ND),\qquad\||x|^{-a}u_\var\|_2^2\asymp\var^{\frac{N-2d}{2d}}.
		\end{align*}
		Inserting these bounds into the quotient and using $\frac{N-2d}{4d}q_c-\frac{N-2d}{2d}\cdot\frac{q_c-2}{2}=\frac{N-2d}{2d}$, we obtain $Q_\var\asymp\var^{\frac{N-2d}{2d}}$ (times $|\log\var|$ if $q_c=\frac ND$), which is \textit{(ii)}.
		
		\noindent\textbf{Case 2:} $q_c>\frac ND$. By the third case of Proposition~\ref{C11}\textit{(iii)}, $\||x|^{-b}u_\var\|_{q_c}^{q_c}\asymp\var^{e_{q_c}}$. Since $N-2a+2b=N-2d+2$, we have $2N-q_c(N-2d)=\frac{8d}{N-2a+2b}$, and therefore
		\begin{align*}
			e_{q_c}=\frac{2(N-2d)}{(N-2a+2b)(N-2-2a)}=E_a\,\frac{q_c-2}{2}.
		\end{align*}
		It remains to insert the three cases of part \textit{(iv)} of that proposition. If $a<\frac{N-4}{2}$, then $\||x|^{-a}u_\var\|_2^2\asymp\var^{E_a}$, so $Q_\var\asymp1$. If $a=\frac{N-4}{2}$, then $\||x|^{-a}u_\var\|_2^2\asymp\var^{E_a}|\log\var|$, so $Q_\var\asymp|\log\var|^{-\frac{q_c-2}{2}}=|\log\var|^{-\frac{2d}{N-2a+2b}}$. If $a>\frac{N-4}{2}$, then $\||x|^{-a}u_\var\|_2^2\asymp\var^{\frac{N-2d}{2d}}$, so
		\begin{align*}
			Q_\var\asymp\var^{e_{q_c}-\frac{N-2d}{2d}\cdot\frac{q_c-2}{2}}=\var^{\frac{(N-2d)(4+2a-N)}{(N-2-2a)(N-2a+2b)}}.
		\end{align*}
		This is \textit{(i)}.\QED
	\end{proof}
	\medskip
	      	  \noindent \textbf{Funding} \\
              The second author gratefully acknowledge the financial support provided by the Anusandhan National Research Foundation (ANRF), India, under the project sanction number\\ ANRF/ARGM/2025/002376/MTR.
              
              \medskip
	  \noindent \textbf{Competing Interests} \\
The author declares that there is no conflict of interest. 
	\bibliographystyle{siam}
	\bibliography{ref}
\end{document}